\documentclass[a4paper,11pt,reqno,noindent]{amsart}
\usepackage[centertags]{amsmath}
\usepackage{amsfonts,amssymb,amsthm} 
\usepackage{mathrsfs}
\usepackage[english]{babel}
\usepackage{newlfont}
\usepackage{color}
\usepackage[body={15cm,21.5cm},centering]{geometry} 
\usepackage{fancyhdr}
\usepackage{esint}

\usepackage{enumerate}

\newtheorem{theo}{Theorem}
\newtheorem{lemma}{Lemma}
\newtheorem{prop}{Proposition}
\newtheorem{coro}{Corollary}

\theoremstyle{definition}
\newtheorem{rem}{Remark}
\newtheorem{defi}{Definition}
\newtheorem{hypo}{Hypothesis}

\numberwithin{equation}{section}
\numberwithin{lemma}{section} 
\numberwithin{defi}{section} 
\numberwithin{rem}{section} 

\newcommand \dps{\displaystyle }

\newcommand{\Km}{K_{\mathrm{max},\ell}}
\newcommand{\loc}{\mathrm{loc}}

\newcommand{\ho}{\mathrm{hom}}

\newcommand{\R}{\mathbb{R}}
\newcommand{\Z}{\mathbb{Z}}

\newcommand{\N}{\mathbb{N}}
\newcommand{\T}{\mathbb{T}}
\newcommand{\Md}{\mathcal{M}_d(\R)}
\newcommand{\Id}{\mathrm{Id}}
\newcommand{\e}{\varepsilon}

\newcommand{\calF}{\mathcal{F}}

\newcommand{\Sc}{\mathcal{S}(\mathbb R^d)}

\newcommand{\calM}{\mathcal{M}}

\newcommand{\calL}{\mathcal{L}}

\newcommand{\cc}{\mathbf{c}}

\newcommand{\Pm}{\mathbb{P}}

\mathchardef\emptyset="001F

\def\aa{\boldsymbol a}
\newcommand{\bb}{\boldsymbol b}

\newcommand{\expec}[1]{\mathbb{E}\left[ #1 \right]}
\newcommand{\expecs}[1]{\mathbb{E}\big[ #1 \big]}

\newcommand{\step}[1]{\noindent \textit{Step} #1.}
\newcommand{\substep}[1]{\noindent \textit{Substep} #1.}

\renewcommand{\phi}{\varphi}

\definecolor{green}{rgb}{0,0.75,0.2}
\definecolor{blue}{rgb}{0,0.2,0.75}
\title[Homogenization and asymptotic ballistic transport of classical waves]{Long-time homogenization and asymptotic ballistic transport of classical waves}
\author[A. Benoit]{Antoine Benoit}
\author[A. Gloria]{Antoine Gloria}
\address[Antoine Benoit]{Universit\'e du Littoral C\^{o}te d'Opale  \\ Calais, France}
\email{antoine.benoit@univ-littoral.fr}
\address[Antoine Gloria]{Sorbonne Universit\'e, UMR 7598, Laboratoire Jacques-Louis Lions, F-75005, Paris, France 
\newline
Universit\'e Libre de Bruxelles, Belgium}
\email{antoine.gloria@upmc.fr}
\begin{document}
\maketitle
%%----------------------------------------------------------------------------------------------------------------------

Consider an elliptic operator in divergence form with symmetric coefficients.
If the diffusion coefficients are periodic, the Bloch theorem allows one to diagonalize the elliptic operator, which is 
key to the spectral properties of the elliptic operator and the usual starting point for the study of its long-time homogenization.
When the coefficients are not periodic (say, quasi-periodic, almost periodic, or random with decaying 
correlations at infinity), the Bloch theorem does not hold and both the spectral properties and the long-time behavior of the associated
operator are unclear.
At low frequencies, we may however consider a formal Taylor expansion of Bloch waves (whether they exist or not) based on correctors in elliptic homogenization.
The associated Taylor-Bloch waves diagonalize the elliptic operator up to an error term (an ``eigendefect''), which we express with the help of a new family of extended correctors.
We use the Taylor-Bloch waves with eigendefects % on a un eigendefect puis plusieurs
 to quantify the transport properties and homogenization error over large times
for the wave equation in terms of the spatial growth of these extended correctors.
On the one hand,  this quantifies the validity of homogenization over large times (both for the standard homogenized equation and higher-order versions).
On the other hand, this allows us to prove asymptotic ballistic transport of classical waves 
at low energies for almost periodic and random operators.

\smallskip


\medskip

Consid\'erons un op\'erateur elliptique sous forme divergence \`a coefficients sym\'etriques non constants.
Si ces coefficients sont p\'eriodiques, la th\'eorie de Floquet-Bloch permet de diagonaliser l'op\'erateur elliptique, ce qui est crucial
pour l'\'etude des propri\'et\'es spectrales de l'op\'erateur et le point de d\'epart usuel pour l'\'etude des propri\'et\'es d'homog\'en\'eisation en temps
long de l'op\'erateur des ondes associ\'e. Quand les coefficients ne sont pas p\'eriodiques (disons quasi-p\'eriodiques, presque p\'eriodiques, ou al\'eatoires
stationnaires ergodiques), la th\'eorie de Floquet-Bloch ne s'applique plus et les propri\'et\'es spectrales ainsi que le comportement en temps long de l'op\'erateur 
des ondes associ\'e ne sont pas claires a priori.
Aux basses fr\'equences, nous pouvons cependant consid\'erer un d\'eveloppement de Taylor formel des ondes de Bloch (que celles-ci existent ou non) en se basant
sur des correcteurs introduits en homog\'en\'eisation elliptique. Ces ondes de Taylor-Bloch diagonalisent l'op\'erateur elliptique \`a un terme d'erreur pr\`es (un ``d\'efaut propre''),
que nous exprimons \`a l'aide d'une nouvelle famille de correcteurs \'etendus. Nous utilisons cette formulation des d\'efauts propres pour quantifier les propri\'et\'es
de transport et d'homog\'en\'eisation en temps long pour l'\'equation des ondes associ\'ee en termes de croissance spatiale des correcteurs \'etendus.
D'une part, cela quantifie la validit\'e de l'homog\'en\'eisation en temps long (\`a la fois pour l'op\'erateur homog\'en\'eis\'e standard et pour des versions d'ordre sup\'erieur).
D'autre part, cela nous permet d'\'etablir le transport balistique asymptotique des ondes classiques aux basses \'energies pour des op\'erateurs presque p\'eriodiques et al\'eatoires.

\smallskip

Mots-cl\'es : homog\'en\'eisation, p\'eriodique, presque p\'eriodique, al\'eatoire, ondes, temps long, transport balistique.

\medskip


\medskip

\tableofcontents

\section{Introduction}

Let $\aa$ be a periodic symmetric coefficient field, and consider the rescaled wave operator $\square_\e:=\partial^2_{tt}-\nabla \cdot \aa(\frac \cdot\e)\nabla$.
It is known since the pioneering works in homogenization that for fixed final time $T<\infty$, the operator $\square_\e$ can 
be replaced by the homogenized wave operator $\square_\ho:=\partial^2_{tt}-\nabla \cdot \aa_\ho\nabla$, where $\aa_\ho$ are the
homogenized (and constant) coefficients associated with $\aa$ through elliptic homogenization (see \cite{BFM,FM} where the question of the corrector and convergence of the energy
is also addressed, and Section~\ref{sec:TB} for precise definitions).
Let $u_0 \in \Sc$, the Schwartz class (most of the results of this article hold for initial conditions in some Hilbert space $H^m(\R^d)$ for $m$ large enough), let $u_\e \in L^\infty(\R_+,L^2(\R^d))$ be the unique weak solution of
\begin{equation}\label{i1}
\left\{
\begin{array}{rcl}
\square_\e u_\e(t,x)&=&0,
\\
u_\e(0,\cdot)&=&u_0,
\\
\partial_t u_\e(0,\cdot)&=&0.
\end{array}
\right.
\end{equation}
Then for all $T>0$, $\lim_{\e \downarrow 0} \sup_{0\le t \le T} \|u_{\e}(t,\cdot)-u_\ho(t,\cdot)\|_{L^2(\R^d)}=0$, where $u_\ho$ solves 
the homogenized equation
\begin{equation}\label{i2}
\left\{
\begin{array}{rcl}
\square_\ho u_\ho(t,x)&=&0,
\\
u_\ho(0,\cdot)&=&u_0,
\\
\partial_t u_\ho(0,\cdot)&=&0.
\end{array}
\right.
\end{equation}
Refining this result received much attention in the recent years --- and in particular the large-time behavior of $u_\e$ with 
fixed or oscillating initial conditions.
For fixed initial conditions $u_0$ (independent of $\e$), one expects dispersive effects --- which are not accounted for by \eqref{i2} --- to appear at times of order $\e^{-2}T$ (see \cite{SS} for pioneering works in this direction, \cite{DLS1,DLS2,L} for the first rigorous results, and \cite{AGS,AP} for numerical methods).
For oscillating initial conditions, the medium interacts with the initial conditions much more, which yields even finer dispersive effects (see \cite{APR1,APR2}).
Both refinements crucially rely on spectral properties of the operator $-\nabla \cdot \aa \nabla$, namely that it is diagonalized 
by Floquet-Bloch waves (Bloch in short, see \cite{AC1,AC2,APR1,APR2,DLS1,DLS2}): the spectrum of  $-\nabla \cdot \aa \nabla$ is purely absolutely continuous,
and extended states are semi-explicit (see below).
Hence, there is a clear starting point to study the above questions: project the initial condition on the Bloch wave basis, and treat the wave equation \eqref{i1} as an ODE. 
From a spectral point of view, the Bloch theory implies that $-\nabla \cdot \aa \nabla$ has purely absolutely continuous spectrum in form  of possibly overlapping bands (the first one including $0$).

\medskip

The Bloch theory crucially relies on the periodicity of $\aa$, and can be seen as a variant of the Fourier transform (with which it coincides
when $\aa$ is a constant matrix).
The main idea is to look for extended states of the operator $-\nabla \cdot \aa \nabla$ in the form of modulated plane waves $x\mapsto e^{ik\cdot x} \psi_{k}(x)$, where $\psi_k$ is a periodic function.
Such a function $\psi_k$ is then solution of the magnetic eigenvalue problem on the torus
$$
-(\nabla + ik)\cdot \aa (\nabla +ik)\psi_k \,=\,\lambda_k \psi_k
$$
for some $\lambda_k$. By the Rellich theorem, $-(\nabla + ik)\cdot \aa (\nabla +ik)$ has compact resolvent, which allows one to define a family of eigenvectors and eigenvalues, on which the Bloch decomposition relies.
Replace $\aa$ by the sum of two periodic functions with incommensurable periods, and the whole picture breaks down: the magnetic operator is now lifted to a higher-dimensional torus, it is hypo-elliptic, and does not have compact resolvent any longer. In particular, we do not know whether the $\psi_k$ exist.
For more general coefficients (say almost-periodic, or random), the Bloch theory simply does not hold. Indeed, this theory 
implies that the elliptic operator $-\nabla \cdot \aa \nabla$ has purely absolutely continuous spectrum, whereas 
it is known that this operator has some discrete spectrum in any dimension for some representative examples of $\aa$, cf. \cite[Theorem~3.3.6]{S}.

\medskip

Questions regarding oscillating initial data explore the entire spectrum of the operator $-\nabla \cdot \aa\nabla$, and we expect a completely different behavior for periodic and non-periodic coefficients, since their spectrum is of different type.
This is the realm of challenging questions of spectral analysis \cite{S,AW} and radiative transport \cite{P,RPK}.
For non-oscillating initial data however, only the bottom of the spectrum of $-\nabla \cdot \aa \nabla$ is relevant, and we are in the realm of 
homogenization. For final times $T<\infty$ independent of $\e$, qualitative theory for the elliptic operator is enough to 
prove the convergence of \eqref{i1} to \eqref{i2}, and we just need to know that the 
solution operator $(-\nabla \cdot \aa(\frac \cdot\e) \nabla)^{-1}$ converges to the homogenized solution operator  $(-\nabla \cdot \aa_\ho \nabla)^{-1}$
as $\e \downarrow 0$. 
If we happen to have quantitative information on this convergence in terms of $\e$ (in a broad sense), we might be able to consider larger time frames 
$[0,\e^{-\alpha}T]$ (with $\alpha>0$) and gain information on the large-time behavior of $u_\e$.
The aim of this contribution is to develop such an approach for operators that are beyond the reach of the classical Bloch theory.

\medskip

As a first and crucial step, we introduce a proxy for the Bloch wave{\color{green}s} decomposition.
Since we are only interested in low frequencies, we only need a proxy for Bloch waves at low frequencies.
In the case of periodic coefficients, it is well-known that Bloch waves $\psi_k$ are essentially analytic functions of $k$, and that 
their derivatives are related to cell-problems in elliptic homogenization (e.g.~\cite{CV,ABV}).
Whereas eigenvectors $\psi_k$ might not exist (even at low frequencies), one may still consider their formal Taylor expansion $\psi_{k,j}$ of order $j$ 
for all $0\le |k|\ll 1$ based on correctors (up to order $j$)
provided the latter exist, which gives rise to what we call Taylor-Bloch waves $x\mapsto e^{ik\cdot x} \psi_{k,j}(x)$. 
These waves are only ``approximate'' extended states of the operator $-\nabla \cdot \aa \nabla$, so that the study of the defect in the eigenvector/eigenvalue relation (which we call the ``eigendefect'') is equally important as the formula for the Taylor-Bloch waves itself. 
The study of the Taylor-Bloch expansion is the aim of Section~\ref{sec:TB},
where we introduce a new family of higher-order correctors, that are used to put the eigendefect in a suitable form for the rest of our analysis.

\medskip

The second step consists in constructing an approximate solution to equation \eqref{i1} using the Taylor-Bloch waves, cf. Section~\ref{sec:wave}.
We first replace the initial condition by a well-prepared initial condition in the form of a Taylor-Bloch expansion --- which simply amounts
to replacing $e^{ik\cdot x}$ by $e^{ik\cdot x} \psi_{\e k,j}(x)$ in the Fourier inversion formula for $u_0$. 
The difference between the solutions of both initial value problems is then proved to be of order $\e$ uniformly in time by energy methods,
and it remains to solve the problem with well-prepared initial data. To this aim, we use that Taylor-Bloch waves diagonalize the elliptic operator (up to the eigendefect) to construct an approximate solution by explicit time-integration. Next, we estimate the error due to the eigendefect by energy methods on the wave equation, which yields a control over large times that depend on the growth of the extended correctors.
To conclude, we simplify the approximate solution by throwing away most of the corrections, while keeping sufficient accuracy in $L^2(\R^d)$.
This final approximation is accurate up to times $\e^{-\alpha}T$, where $\alpha>0$ depends on the growth of the extended correctors. 

\medskip

Equipped with the Taylor-Bloch approximation of the solution to equation \eqref{i1}, we turn to the main two results of this article: the long-time homogenization of $\partial_{tt}^2-\nabla \cdot \aa(\frac \cdot \e)\nabla$ (cf. Section~\ref{sec:disper}) and the (asymptotic) ballistic transport properties of $-\nabla \cdot \aa \nabla$ at low energies (cf. Section~\ref{sec:deloc}).
The range of application of these results crucially depends on the control we have on the extended correctors.
Although this is an important issue, this does not constitute the original part of this article: the analysis uses (rather than develops) methods 
introduced in recent independent works on quantitative homogenization of linear elliptic operators in divergence form that started with \cite{NS,GO1,GO2,GNO1,GO3}, and culminated in \cite{AS,AKM2} and \cite{GNO3,GNO3b,GO4}, see also \cite{AGK}.
For completeness, and in order to stress the interest of our results, we quickly display in Appendix~\ref{sec:corr} the estimates on the growth of the extended correctors that are expected to hold for some almost periodic and random coefficients (the proofs of these results are however not straightforward). 

\medskip

Let us start with the long-time homogenization results.
Based on the Taylor-Bloch approximation of the solution to equation \eqref{i1}, we prove the validity of the approximation of \eqref{i1} by \eqref{i2} up to times $\e^{-\alpha}T$ for some $0<\alpha<2$ depending on the coefficients and dimension (cf. Appendix~\ref{sec:corr}).
Provided the second extended correctors are essentially bounded (which is always the case for periodic coefficients,  holds under some conditions
for almost periodic coefficients, and can only hold in dimensions $d>4$ for random coefficients with decay of correlations at infinity), we obtain
the validity of an approximation involving dispersive effects up to times $\e^{-\alpha}T$ for some $2< \alpha<4$. The interpretation of dispersion in the approximate solution goes through a fourth-order equation,
which is a higher-order homogenized equation, and we essentially follow \cite{DLS2} (which deals with periodic coefficients). 
For non-periodic coefficients, all the results are new.
For periodic coefficients, besides we also treat systems, we improve  \cite{DLS2}  in three respects: we do not require the coefficient field $\aa$ to be smooth, we obtain error estimates
valid over larger times, and we generalize the result to any order (which yields a new family of higher-order homogenized equations parametrized by $n\in \N$ that are valid
up to times $\e^{-2(n+1)^{-}}T$).
Incidentally, our analysis also yields new insight in the homogenization of elliptic problems: it allows to define
higher-order homogenized elliptic operators and to control the associated multiscale homogenization error when the right-hand side (RHS in short) is well-prepared. This seems to be new even in the periodic setting and extends (in the symmetric setting) the recent independent work \cite{BFFO}  to any order.

\medskip

We conclude with the asymptotic 
ballistic transport for $-\nabla \cdot \aa \nabla$.
The characterization of the spectrum of  $-\nabla \cdot \aa \nabla$ (or that of the random Schr\"odinger operator $-\triangle+V$, with a random potential $V$), and in particular the understanding of the (expected) transition between discrete and continuous spectrum, is a major open problem of mathematical physics. 
In the case of the elliptic operator $-\nabla \cdot \aa \nabla$, the bottom of the spectrum is very peculiar (indeed, spectral localization can be proved at band edges, but not at $0$, cf.~\cite{FK1}), and one might expect the spectrum to have a continuous part in the neighborhood of $0$.  A stronger statement of the existence of continuous spectrum would be the ballistic transport of suitable initial conditions for arbitrarily long times. The approach based on Taylor-Bloch waves yields a first \emph{asymptotic} result in that direction, and allows us to prove the ballistic transport of initial conditions of ``energy'' $0<\e\ll 1$ on time frames $[0,\e^{-\alpha}T]$ for some specific $\alpha\ge 0$ depending on the structure of the coefficients and dimension.
For periodic coefficients, our estimates can be proved to be uniform, and yield ballistic transport at all times provided $0<\e\ll1$ (thus establishing ballistic transport at small energies without explicitly appealing to the Bloch theorem). For diophantine quasi-periodic coefficients (for which the Bloch theory does not hold), one can choose $\alpha>0$ arbitrarily large. However, as opposed to the periodic setting, the multiplicative constants in the estimates  blow up too fast as $\alpha \uparrow+\infty$ to prove ballistic transport at all times. For random coefficients, $\alpha$ depends on the dimension and the correlations: for Gaussian coefficient fields with integrable correlations, we have some asymptotic ballistic transport in dimensions $d>2$.

\medskip

In the core of this article we use scalar notation.
All the results also hold for systems, and we refer the reader to Appendix~\ref{sec:systems} for the necessary adaptations of the arguments to that setting.

\medskip

A similar strategy based on approximate spectral analysis (and various notions of approximate Taylor-Bloch waves) can be used to establish asymptotic ballistic transport of quantum waves, and we refer the reader to \cite{DGS} for such results on the Schr\"odinger equation with periodic, quasi-periodic, and random potentials.

%%%%%%%%%%%%%%%%%
%%%%%%%%%%%%%%%%%
%%%%%%%%%%%%%%%%%
%%%%%%%%%%%%%%%%%
%%%%%%%%%%%%%%%%%
%%%%%%%%%%%%%%%%%

\section{Taylor-Bloch waves}\label{sec:TB}

Let $\Lambda\ge 1$. Throughout this contribution we assume that $\aa:\R^d\to \Md$ is a measurable uniformly elliptic \emph{symmetric }coefficient field that satifies for all $\xi \in \R^d$
and almost all $x\in \R^d$ 
$$
\xi \cdot \aa(x)\xi \ge |\xi|^2, \quad |\aa(x)\xi|\le \Lambda |\xi|.
$$
All the constants in our estimates acquire a dependence on $\Lambda$.
(In view of the quantitative stochastic homogenization results used in Appendix~\ref{sec:corr}, all the results of this contribution 
hold true in the case of strongly elliptic systems, cf.~Appendix~\ref{sec:systems}).
This section is inspired by \cite{ABV}, where the authors derive equations satisfied by the derivatives
of Bloch waves at $0$ for periodic coefficients. 
The main additional insight compared to \cite{ABV} is the identification of the structure of these derivatives.
In particular, we rewrite them in terms
of suitable correctors, which allows us to turn the remainder in the Taylor expansion in divergence form plus a higher-order term, cf. the eigendefect in 
Definition~\ref{defi:Bloch}. In the periodic setting, this is not fundamental. In other settings however (like almost periodic or random), this allows us to
show that the Taylor-Bloch expansion is one order more accurate than expected (cf. Section~\ref{sec:wave}), which is crucial to capture both the correct long-time accuracy in homogenization (cf. Section~\ref{sec:disper}) and the correct dimensions
for asymptotic ballistic transport (cf. Section~\ref{sec:deloc}).

\medskip

We assume that $\aa$ are  $\Z^d$-stationary ergodic coefficients (this class includes periodic, quasiperiodic, almost-periodic, and random coefficients with decaying correlation at infinity).
We denote by $\expec{\cdot}$ the expectation (in the periodic setting, the expectation can be dropped).

\subsection{Extended correctors and higher-order homogenized coefficients}

We start with the definition of a family of extended correctors which will serve as the basis for the definition of the Taylor-Bloch waves, and momentarily fix
a unit direction $e\in \R^d$.
\begin{defi}\label{defi:ext-corr}
For all $\ell\ge 0$, we say that $(\phi_j,\sigma_j,\chi_j)_{0\le j\le \ell}$ are the first $\ell$ extended correctors in direction $e$ if these 
functions are locally square-integrable,  if for all $0<j\le \ell$ the functions $(\nabla \phi_j,\nabla \sigma_j)$ are $\Z^d$-stationary and
satisfy $\expec{\int_Q |(\nabla \phi_j,\nabla \sigma_j)|^2}<\infty$, if for all $0<j<\ell$ the functions $(\phi_j,\sigma_j,\nabla \chi_j)$ are $\Z^d$-stationary and satisfy
$\expec{\int_Q (\phi_j,\sigma_j,\nabla \chi_j)}=0$ and $\expec{\int_Q |( \phi_j, \sigma_j, \nabla \chi_j)|^2}<\infty$, and if the following extended corrector equations on $\R^d$
are satisfied:
\begin{itemize}
\item $\phi_0\equiv 1$, and for all $j\ge 1$, $\phi_j$ is a scalar field that satisfies
$$-\nabla \cdot \aa\nabla \phi_j=\nabla \cdot (-\sigma_{j-1}e+\aa e\phi_{j-1}+\nabla \chi_{j-1});$$
\item for all $j\ge 0$, the symmetric matrix $\tilde \aa_j$, the symmetric $(j+2)$-th order tensor $\bar \aa_j$, and the scalar $\lambda_j$ are given by
$$\bar \aa_j  e^{\otimes(j+1)} =\tilde \aa_j e:= \expec{\int_Q\aa (\nabla \phi_{j+1}+ e \phi_j)}, \quad \lambda_j:= e\cdot \tilde \aa_j  e,  \quad e^{\otimes (j+1)}:=\underbrace{e\otimes \dots \otimes e}_{j+1\text{ times }}
;$$
\item $\chi_0\equiv 0$, $\chi_1\equiv 0$, and for all $j\ge 2$, $\chi_j$ is a scalar field that satisfies
$$
-\triangle \chi_j=\nabla\chi_{j-1}\cdot e +\sum^{j-1}_{p=1}\lambda_{j-1-p}\phi_{p};
$$
\item for all $j\ge 1$, $q_j$ is a vector field (a higher-order flux) given by
$$q_j:=  \aa (\nabla\phi_j+e\phi_{j-1})-\tilde \aa_{j-1}e+\nabla \chi_{j-1}-\sigma_{j-1}e, \quad \expec{\int_Q q_j}=0;$$
\item $\sigma_0\equiv 0$, and for all $j\ge 1$, $\sigma_j$ is a skew-symmetric matrix field (a higher-order flux corrector), i.e. $\sigma_{jkl}=-\sigma_{jlk}$, that satisfies
$$
-\triangle \sigma_j = \nabla \times q_j, \quad \nabla \cdot \sigma_j= q_j,
$$
with the three-dimensional notation: $[\nabla \times q_j]_{mn}=\nabla_m [q_j]_n-\nabla_n [q_j]_m$,
and where the divergence is taken with respect to the second index, i.~e.~$(\nabla \cdot \sigma_j)_m:=\sum_{n=1}^d \partial_n \sigma_{jmn}$.
\end{itemize}
\qed
\end{defi}
Let us make a few comments on this definition.
\begin{itemize}
\item Let $\ell_1>\ell_2$ be two integers, and denote the families of correctors associated with $\ell_1$ and $\ell_2$ by $(\phi_j^{\ell_1},\sigma_j^{\ell_1},\chi_j^{\ell_1})_{0\le j\le \ell_1}$ and by  $(\phi_j^{\ell_2},\sigma_j^{\ell_2},\chi_j^{\ell_2})_{0\le j\le \ell_2}$, respectively. Then, for all $0\le j\le \ell_2$, $(\phi_j^{\ell_1},\sigma_j^{\ell_1},\chi_j^{\ell_1})=(\phi_j^{\ell_2},\sigma_j^{\ell_2},\chi_j^{\ell_2})$. In particular, if well-defined, correctors at order $j$ do not depend on $\ell\ge j$ in Definition~\ref{defi:ext-corr}.
However, depending on the assumptions we make on the distribution of the coefficient field $\aa$, there is a maximal $\ell$ for which the formal Definition~\ref{defi:ext-corr} makes sense (i.e.~for which correctors of order $j>\ell$ are not well-defined --- we see as a property of $\aa$).
\item The correctors $\phi_j$ are related but do not coincide (for $j>2$) with the higher-order correctors classically used in the multiscale expansion for periodic coefficients (or random coefficients in \cite{Gu}), and we refer the reader to \cite{ABV} for a discussion of these differences in the periodic setting. 
\item The higher-order flux $q_j$ is chosen to be divergence-free, so that it is an exact $(d-1)$-form
and hence admits a ``vector potential'', that is a $(d-2)$-form, which can be represented
by the skew-symmetric tensor $\sigma_j$ (the equation for $\sigma_j$ is the natural choice of gauge).
These definitions are natural generalizations to any order of the extended correctors $(\phi,\sigma)$ considered in \cite{GNO3} (see below).
For the the existence, uniqueness, and properties of these extended correctors (depending on the
properties of the field $\aa$ and the dimension $d$), we refer the reader to Appendix~\ref{sec:corr}. 
These correctors for $j=2$ were considered independently by Bella, Fehrmann, Fischer, and Otto in  \cite{BFFO}.
\item The correctors $\phi_j$ are variants of those defined in \cite{ABV}. They are however not normalized the same way (see in particular Remark~\ref{rem:normalization} below), which is crucial to consider unbounded higher-order correctors.
\item The correctors $(\sigma_j,\chi_j)$ are primarily introduced to develop an approximate spectral theory. This can be used to study the wave equation with well-prepared initial data (the main aim of this article), but also to study the elliptic equation with well-prepared RHS. In this case, these correctors allow one to write the remainder of the multiscale expansion in divergence form plus a higher-order term, which turns out to be new (to our knowledge) and directly yields sharp convergence results for the higher-order expansion with well-prepared RHS (which strictly generalizes the quantitative two-scale expansion of \cite{GNO3}, as well as the 
second order expansion of \cite{BFFO}, which both treat non-symmetric coefficients; see also \cite[Section~5]{ABV}), cf.~Theorem~\ref{th:2scale-ell} and Corollary~\ref{coro:alaBFFO}.
\item Let us quickly show that the first extended correctors (of order $j=1$) are indeed the standard correctors
in elliptic homogenization. The equation satisfied by $\phi_1$ takes the form
$$
-\nabla \cdot \aa (\nabla \phi_1+e)\,=\,0
$$
so that $\phi_1$ is the classical corrector in stochastic homogenization, i.e. the unique sublinear at infinity
solution  of $-\nabla\cdot \aa (\nabla \phi_1+e)=0$ with stationary gradient (for conditions on $\aa$ under which $\phi_1$ is stationary, we refer
to \cite{GO1,GO3,GNO3,GO4}). Thus $\bar \aa_0=\tilde \aa_0=\aa_\ho$ (the homogenized coefficients), $\lambda_0=e\cdot \aa_\ho e$.
Hence, $q_1=\aa(\nabla \phi_1+e)-\aa_\ho e$ (the flux of
the corrector minus the homogenized flux), so that  $\sigma_1$ is nothing but the flux corrector (the existence of which is proved in \cite{GNO3} for stationary ergodic coefficients $\aa$).
\item We refer the reader to Appendix~\ref{sec:systems} for the precise extension of Definition~\ref{defi:ext-corr} to systems.
\end{itemize}

\medskip

We conclude this paragraph with some important properties of the higher-order homogenized coefficients  $\lambda_j$.
For notational convenience we assume that $\aa$ enjoys continuum stationarity (for discrete stationarity, it suffices to replace
$\expec{\cdot}$ by $\expec{\int_Q\cdot }$).
\begin{prop}\label{prop:properties-lambda}
Let $\ell\in \N_0$ and assume that the correctors $(\phi_{j},\sigma_j,\chi_j)_{0\le j\le \ell+1}$ are well-defined in the sense of Definition~\ref{defi:ext-corr},
so that the higher-order homogenized coefficients $\{\lambda_{j}=\expec{e\cdot \aa (\nabla \phi_{j+1}+ e \phi_{j})}\}_{0\le j\le \ell}$ are well-defined.
Then:
\begin{itemize}
\item[(i)] if $0\le j\le \ell$ is odd, then $\lambda_{j}=0$;
\item[(ii)] $\lambda_0>0$ and $\lambda_{2} \ge 0$ (provided $\ell\ge 2$).
\end{itemize}
\qed
\end{prop}
For periodic coefficients, this proposition is standard: Statement (i) is due to the symmetry of the coefficient field and follows from spectral theory (see e.g. \cite{AC1}, and the proof of Proposition~\ref{prop:corr-perio} in Appendix~\ref{sec:corr}), and the nice observation (ii) for $j=2$ was first proved in \cite{COV} in the scalar setting. 
Although the proof of (ii) in \cite{COV} essentially extends to the stochastic setting, and the proof of (i) follows from the result in the periodic setting by a suitable approximation argument (the so-called periodization method),
we display elementary proofs of both results that do not rely on spectral theory, hold mutadis mutandis for systems, and use in a systematic way the ``algebraic'' properties of the extended correctors of Definition~\ref{defi:ext-corr}.
(The proof of $\lambda_{2j+1}=0$ extends the result $\lambda_1=0$ proved in \cite{BFFO}.)

\begin{proof}[Proof of Proposition~\ref{prop:properties-lambda}]
We start with the proof of (i) by induction, and then turn to the proof of~(ii).
In this proof we systematically use the symmetry of $\aa$ to change the order in scalar products
without transposing $\aa$.
This proof mainly exploits the algebraic structure of correctors and of differential operators,
which makes it rather dry.

\medskip

\step{1} Proof of (i).

\noindent 
The aim of this step is to prove that some quantity vanishes.
In particular, we have to unravel cancellations.
This goes through a careful reformulation of the quantity combined with an 
induction argument. We split the proof of (i) into several substeps.
In the first substep we prove a crucial identity for quadractic forms of the correctors
which is both at the basis of (i) and (ii).
In the second substep, we formulate a particular (and useful) case of this identity.
The third substep is dedicated to the proof of (i) by induction.

\medskip

\substep{1.1} Proof of the identity: for all $j\ge 1$ and $l\ge j+1$,
\begin{multline}\label{a.2.1}
\expec{\nabla \phi_j\cdot \aa \nabla \phi_l-\phi_{j-1}\phi_{l-1}e\cdot \aa e}
\,=\,\expec{-\nabla \phi_{j+1}\cdot \aa \nabla \phi_{l-1}+\phi_{j}\phi_{l-2}e\cdot \aa e}
\\
-\expec{\sum_{m=1}^{j-1} \lambda_{j-1-m}\phi_{l-1}\phi_m+\sum_{m=1}^{l-2} \lambda_{l-2-m} \phi_m \phi_j}.
\end{multline}
Starting point is the equation for $\phi_l$ in the form 
\begin{eqnarray*}
\expec{\nabla \phi_j\cdot \aa \nabla \phi_l}&=&\expec{-\nabla \phi_j \cdot (\aa e \phi_{l-1}-\sigma_{l-1}e+\nabla \chi_{l-1})}
\\
&=&\expec{-\nabla \phi_j \cdot \aa e \phi_{l-1}}+\expec{\nabla \phi_j \cdot \sigma_{l-1}e}+\expec{-\nabla \phi_j\cdot \nabla \chi_{l-1}},
\end{eqnarray*}
and we reformulate the last two RHS terms. 
By stationarity of $\nabla \cdot (\phi_j\sigma_{l-1}e)$ in the form $0=\expec{\nabla \cdot (\phi_j\sigma_{l-1}e)}=\expec{\nabla \phi_j \cdot \sigma_{l-1}e}+\expec{ \phi_j  \cdot (\nabla \cdot \sigma_{l-1}e)}$,
the skew-symmetry of $\sigma_{l-1}$ in the form $\expec{ \phi_j  \cdot (\nabla \cdot \sigma_{l-1}e)}=-\expec{ \phi_j  e\cdot (\nabla \cdot \sigma_{l-1})}$,
the defining property $\nabla \cdot \sigma_{l-1}=q_{l-1}$, the fact that $\expec{\phi_j}=0$,
and the skew-symmetry of $\sigma_{l-2}$ in the form $e\cdot \sigma_{l-2}e=0$, we have
\begin{eqnarray*}
\expec{\nabla \phi_j \cdot \sigma_{l-1}e}&=&-\expec{ \phi_j  \cdot (\nabla \cdot \sigma_{l-1}e)}
\\
&=&\expec{ \phi_j e \cdot (\nabla \cdot \sigma_{l-1})}
\\
&=&\expec{ \phi_j e \cdot (\aa(\nabla \phi_{l-1}+e \phi_{l-2})-\lambda_{l-1}{\tilde{a}_{l-2}e}+\nabla \chi_{l-2}+\sigma_{l-2}e)}
\\
&=&\expec{ \phi_j e \cdot (\aa(\nabla \phi_{l-1}+e \phi_{l-2})+\nabla \chi_{l-2})}.
\end{eqnarray*}
By the equation for $\chi_{l-1}$ tested with $\phi_j$, we also have
\begin{equation*}
\expec{-\nabla \phi_j\cdot \nabla \chi_{l-1}}\,=\,\expec{-\nabla \chi_{l-2} \cdot e\phi_j}-\expec{\sum_{m=1}^{l-2} \lambda_{l-2-m} \phi_m \phi_j}.
\end{equation*}
The combination of these last four identities yields
\begin{equation}\label{a.2.1-0}
\expec{\nabla \phi_j\cdot \aa \nabla \phi_l}\,=\,\expec{ \phi_j e \cdot \aa \nabla \phi_{l-1}-\nabla \phi_j \cdot \aa e \phi_{l-1}}
+\expec{ \phi_j \phi_{l-2}e \cdot \aa e -\sum_{m=1}^{l-2} \lambda_{l-2-m} \phi_m \phi_j}.
\end{equation}
We then appeal to the equation for $\phi_{j+1}$ in the form
$$
\expec{  \nabla \phi_{l-1}\cdot \aa \phi_j e} \,=\, \expec{-\nabla \phi_{l-1}\cdot \aa \nabla \phi_{j+1}}-\expec{\nabla \phi_{l-1}\cdot (\nabla \chi_{j}-\sigma_{j}e)}.
$$
We need to reformulate the second RHS term. By the stationarity of $\nabla \cdot(\phi_{l-1}\sigma_{j}e)$ in the form $0=\expec{\nabla \cdot (\phi_{l-1}\sigma_{j}e)}=\expec{\nabla  \phi_{l-1}\cdot \sigma_{j}e)}+\expec{\phi_{l-1}\nabla \cdot \sigma_{j}e}$, followed by the skew-symmetry of $\sigma_{j}$ in form of $\expec{\phi_{l-1}\nabla \cdot \sigma_{j}e}=-\expec{\phi_{l-1}e \cdot (\nabla \cdot \sigma_{j})}$,
we have $\expec{\nabla \phi_{l-1}\cdot \sigma_{j}e}=-\expec{\phi_{l-1}e \cdot (\nabla \cdot \sigma_{j})}$.
Hence, by using the equation for $\chi_{j}$ to reformulate $\expec{\nabla \phi_{l-1}\cdot \nabla \chi_{j}}$, we obtain
\begin{multline*}
-\expec{\nabla \phi_{l-1}\cdot (\nabla \chi_{j}-\sigma_{j}e)}\,\\
=\,\expec{-  \nabla \chi_{j-1}\cdot e \phi_{l-1}-\sum_{m=1}^{j-1} \lambda_{j-1-m}\phi_m\phi_{l-1}}
+\expec{\phi_{l-1}e \cdot (\nabla \cdot \sigma_j)}.
\end{multline*}
We then insert the formula $\nabla \cdot \sigma_j= \aa(\nabla \phi_j+e\phi_{j-1})-{\tilde{a}_je}+\nabla \chi_{j-1}+\sigma_{j-1} e$ and obtain
(since $\expec{\phi_{l-1}}=0$  and $e\cdot \sigma_{j-1}e\equiv 0$ by skew-symmetry)
\begin{equation*}
-\expec{\nabla \phi_{l-1}\cdot (\nabla \chi_{j}-\sigma_{j}e)}\,
=\,-\expec{\sum_{m=1}^{j-1} \lambda_{j-1-m}\phi_m\phi_{l-1}}
+\expec{\phi_{l-1}e \cdot\aa(\nabla \phi_j+e\phi_{j-1})},
\end{equation*}
so that \eqref{a.2.1-0} takes the form
\begin{multline*}
\expec{\nabla \phi_j\cdot \aa \nabla \phi_l}\,=\,\expec{\phi_{l-1}\phi_{j-1} e \cdot\aa e}+\expec{ -\nabla \phi_{j+1}\cdot \aa \nabla \phi_{l-1}+\phi_j \phi_{l-2}e \cdot \aa e} 
\\-\expec{\sum_{m=1}^{j-1} \lambda_{j-1-m}\phi_m\phi_{l-1} +\sum_{m=1}^{l-2} \lambda_{l-2-m} \phi_m \phi_j}.
\end{multline*}
This yields \eqref{a.2.1}.

\medskip

\substep{1.2} Proof of the identity: for all $j\ge 1$,
\begin{equation}\label{a.2.2}
\expec{\nabla \phi_{j}\cdot \aa \nabla \phi_{j+1}-\phi_{j-1}\phi_{j}e\cdot \aa e}
\,=\,-\expec{\sum_{m=1}^{j-1} \lambda_{j-1-m} \phi_m \phi_j}.
\end{equation}
This identity is a direct consequence of \eqref{a.2.1-0}, which for $l=j+1$ takes the simpler form
\begin{equation*}
\expec{\nabla \phi_j\cdot \aa \nabla \phi_{j+1}}\,=\,\underbrace{\expec{ \phi_j e \cdot \aa \nabla \phi_{j}-\nabla \phi_j \cdot \aa e \phi_{j}}}_{\dps =0}
+\expec{ \phi_j \phi_{j-1}e \cdot \aa e -\sum_{m=1}^{j-1} \lambda_{j-1-m} \phi_m \phi_j}.
\end{equation*}

\medskip

\substep{1.3} Conclusion. 

\noindent We are now in the position to prove that for all $j\in \N_0$, 
$$
\lambda_{2j+1}\,=\,\expec{e\cdot \aa(\nabla \phi_{2(j+1)}+e \phi_{2j+1})}\,=\,0.
$$
We start by using the equation for $\phi_1$ and the definition $\phi_0\equiv 1$ to turn the above into
$$
\lambda_{2j+1}\,=\,\expec{-\nabla \phi_1 \cdot \aa \nabla \phi_{2(j+1)} + \phi_{2j+1}\phi_0 e\cdot \aa e}.
$$
We then apply identity~\eqref{a.2.1} $j$ times to obtain 
\begin{multline*}
\lambda_{2j+1}\,=\, \expec{(-1)^{j+1}\nabla \phi_{j+1} \cdot \aa \nabla \phi_{j+2}+(-1)^j \phi_{j}\phi_{j+1}e\cdot \aa e}
\\
+\sum_{m_1=1}^j   (-1)^{m_1+1} \expec{\sum_{m_2=1}^{m_1-1} \lambda_{m_1-1-m_2}\phi_{m_2}\phi_{2(j+1)-m_1}+\sum_{m_2=1}^{2j+1-m_1} \lambda_{2j+1-m_1-m_2} \phi_{m_2} \phi_{m_1}},
\end{multline*}
in which we insert  identity~\eqref{a.2.2}. This yields
\begin{multline}\label{a.2.3-0}
\lambda_{2j+1}\,=\, (-1)^{j} \expec{\sum_{m_2=1}^{j} \lambda_{j-m_2} \phi_{m_2} \phi_{j+1}}
\\
+\sum_{m_1=1}^j   (-1)^{m_1+1} \expec{\sum_{m_2=1}^{m_1-1} \lambda_{m_1-1-m_2}\phi_{m_2}\phi_{2(j+1)-m_1}+\sum_{m_2=1}^{2j+1-m_1} \lambda_{2j+1-m_1-m_2} \phi_{m_2} \phi_{m_1}},
\end{multline}
and it remains to argue that the RHS vanishes.
We proceed by induction and assume that for all $l < j$, $\lambda_{2l+1}=0$. 
To initialize the induction, it is enough to note that $\lambda_1=0$ since the RHS of \eqref{a.2.3-0} is obviously zero  for $j=0$.
We then rewrite \eqref{a.2.3-0} in a more suitable form to unravel the cancellations:
\begin{equation*}
\lambda_{2j+1}\,=\, \sum_{p=1}^{2j+1} (-1)^{p+1} \sum_{l=1}^{2j+1-p} \lambda_{2j+1-p-l} \expec{ \phi_{l} \phi_{p}}.
\end{equation*}
If $p+l$ is odd then $(-1)^{l+1}+(-1)^{p+1}=0$, whereas if $p+l$ is even then $\lambda_{2j+1-p-l}=0$ by the induction assumption.
Hence, for all $p,l\ge 1$ with $p+l\le 2j+1$, we have $((-1)^{l+1}+(-1)^{p+1})\lambda_{2j+1-p-l}=0$.
This concludes the proof of (i) since by symmetry one may rewrite the above sum as
\begin{equation*}
\lambda_{2j+1}\,=\, \frac12 \sum_{p=1}^{2j+1} \sum_{l=1}^{2j+1-p} ((-1)^{l+1}+(-1)^{p+1}) \lambda_{2j+1-p-l} \expec{ \phi_{l} \phi_{p}}=0.
\end{equation*}

\medskip

\step{2} Proof of (ii).

\noindent For $j=0$, the result reduces to the well-known ellipticity of $\aa_\ho$, which follows from the formula
$$
\lambda_0\,=\,\expec{e\cdot \aa(\nabla \phi_1+e)}\,=\, \expec{(\nabla \phi_1+e)\cdot \aa (\nabla \phi_1+e)} \ge 1,
$$
a direct consequence of the corrector equation for $\phi_1$ in the form $\expec{\nabla \phi_1 \cdot \aa (\nabla \phi_1+e)}=0$,
and of the ellipticity assumption on $\aa$.
We now turn to $\lambda_2$, and shall prove that 
\begin{equation}\label{e.fo-lambda2}
\lambda_2\,=\, \expec{\nabla (\phi_2-\frac{\phi_1^2}{2}) \cdot \aa \nabla (\phi_2-\frac{\phi_1^2}{2})}\,\ge\,0.
\end{equation}
Starting point is the equation for $\phi_1$ followed by identity \eqref{a.2.1} for $j=1$ and $l=3$:
\begin{eqnarray*}
\lambda_2&=&\expec{e\cdot \aa(\nabla \phi_3+e\phi_2)}
\\
&=& \expec{-\nabla \phi_1\cdot \nabla \phi_3+\phi_0\phi_2 e\cdot \aa e}
\\
&=& \expec{\nabla \phi_2\cdot \aa \nabla \phi_2-\phi_1^2 e\cdot \aa e+\phi_1^2\lambda_0}.
\end{eqnarray*}
We reformulate the last two terms using the property $\nabla \cdot \sigma_1=\aa(\nabla \phi_1+e)-\bar \aa_0e$,
and the stationarity of $\nabla \cdot(\phi_1^2 e \cdot \sigma_1)$ and the skew-symmetry of $\sigma_1$ in the form
$\expec{\phi_1^2 e \cdot (\nabla \cdot \sigma_1)}=\expec{\nabla \cdot(\phi_1^2 e \cdot \sigma_1)}- \expec{\nabla \phi_1^2 \otimes e \cdot \sigma_1}
=0+\expec{\nabla \phi_1^2 \cdot \sigma_1 e}$:
\begin{eqnarray*}
\expec{-\phi_1^2 e\cdot \aa e+\phi_1^2\lambda_0}
&=& \expec{-\phi_1^2e \cdot (\aa(\nabla \phi_1+e)-\bar\aa_0 e)+\phi_1^2 e\cdot \aa \nabla \phi_1}
\\
&=&\expec{-\phi_1^2e\cdot  (\nabla \cdot \sigma_1)+\phi_1^2 e\cdot \aa \nabla \phi_1}
\\
&=&\expec{-\nabla \phi_1^2\cdot  \sigma_1 e+\phi_1^2 e\cdot \aa \nabla \phi_1}.
\end{eqnarray*}
We then appeal to the definition formula for $\sigma_2$ together with the property $\nabla \cdot (\nabla \cdot \sigma_2)=0$ 
in the weak form $\expec{\nabla \phi_1^2 \cdot (\nabla \cdot \sigma_2)}=0$, and the equation for $\phi_1$ in the
form $\expec{-\nabla \phi_1^3\cdot \aa e}=\expec{\nabla \phi_1^3\cdot \aa \nabla \phi_1}$, and obtain
\begin{eqnarray*}
\lefteqn{\expec{-\nabla \phi_1^2\cdot  \sigma_1 e+\phi_1^2 e\cdot \aa \nabla \phi_1}}
\\
&=&\expecs{\nabla \phi_1^2 \cdot \underbrace{(\aa(\nabla \phi_2+\phi_1e)-\sigma_1e)}_{\dps = \nabla \cdot \sigma_2}
-\nabla \phi_1^2\cdot \aa \nabla \phi_2\underbrace{-\nabla \phi_1^2 \cdot \aa e \phi_1+\phi_1^2 e\cdot \aa \nabla \phi_1}_{\dps=\,-\frac13 \nabla \phi_1^3\cdot \aa e }}
\\
&=&\expec{-\nabla \phi_1^2\cdot \aa \nabla \phi_2+\frac13 \nabla \phi_1^3\cdot \aa \nabla \phi_1}.
\end{eqnarray*}
Combining these last three identities, we finally conclude by the Leibniz rule  that
$$
\lambda_2\,=\, \expec{\nabla \phi_2\cdot \aa \nabla \phi_2-\nabla \phi_1^2\cdot \aa \nabla \phi_2+\frac13 \nabla \phi_1^3\cdot \aa \nabla \phi_1}\,=\, \expec{\nabla (\phi_2-\frac{\phi_1^2}{2}) \cdot \aa \nabla (\phi_2-\frac{\phi_1^2}{2})},
$$
as claimed.
\end{proof}

\begin{rem}
Whereas the identity $\lambda_{2j+1}=0$ is generic, it is unclear to us whether we generically have a sign for $\lambda_{2j}$.
In particular, arguing as in the proof above we obtain
$$
\lambda_4\,=\,\expec{-\nabla \phi_3\cdot \aa \nabla \phi_3+\phi_2^2e\cdot \aa e-\phi_2^2 \lambda_0+\lambda_2\phi_1^2},
$$
which we are presently unable to reformulate as a square (or sum of squares) as in~\eqref{e.fo-lambda2}. \qed
\end{rem}

\subsection{Taylor-Bloch wave, eigenvalue, and eigendefect}

In the following, we assume that extended correctors exist.
We define Taylor-Bloch waves as follows.
\begin{defi}\label{defi:Bloch}
Let $k:=\kappa e$ with $\kappa \in \R$ and let $(\phi_j,\lambda_j,\sigma_j,\chi_j)_{0\le j\le \ell}$ be as in Definition~\ref{defi:ext-corr}.
The Taylor-Bloch wave $\psi_{k,\ell}$, Taylor-Bloch eigenvalue $\tilde \lambda_{k,\ell}$, and Taylor-Bloch eigendefect $\mathfrak{d}_{k,\ell}$ of order $\ell$ in direction $k$ are defined by
\begin{multline*}
\psi_{k,\ell} \,:=\,\sum_{j=0}^\ell (i\kappa)^j \phi_j, \quad \tilde \lambda_{k,\ell}\,:=\,\kappa^2 \sum_{j=0}^{\ell-1} (i\kappa)^j \lambda_j,
\\
 \mathfrak{d}_{k,\ell}= \nabla \cdot (-\sigma_{\ell}e+\aa e\varphi_{\ell} +\nabla \chi_{\ell})+i\kappa \Big(e\cdot \aa e \varphi_{\ell}-\sum_{j=1}^\ell \sum_{l=\ell-j}^{\ell-1} (i\kappa)^{j+l-\ell}  \lambda_{l}\phi_{j}\Big) .
\end{multline*}
\qed
\end{defi}
Note that by Proposition~\ref{prop:properties-lambda}, $\tilde \lambda_{k,\ell}$ is real-valued since $\lambda_{2j+1}=0$ for all $j$.
The interest of this definition is the following proposition, which establishes that the Taylor-Bloch wave $\psi_{k,\ell}$ is an
eigenvector of the magnetic operator $-(\nabla+ik)\cdot \aa(\nabla +ik)$ on $\R^d$ for the eigenvalue $\tilde \lambda_{k,\ell}$ up to the eigendefect $\mathfrak{d}_{k,\ell}$.
\begin{prop}\label{prop:bloch}
Let $k=\kappa e$, and $\psi_{k,\ell},\tilde \lambda_{k,\ell},\mathfrak{d}_{k,\ell}$ be as in Definition~\ref{defi:Bloch} for some $\ell\ge 1$.
Then we have
\begin{equation}
-(\nabla +ik) \cdot \aa (\nabla + ik )\psi_{k,\ell} \,=\, \tilde \lambda_{k,\ell} \psi_{k,\ell} -
(i\kappa)^{\ell+1} \mathfrak{d}_{k,\ell}.
\end{equation}
\qed
\end{prop}
\begin{rem}\label{rem:normalization}
In the periodic setting, the Taylor-Bloch wave $\psi_{k,\ell}$ can be compared to the Taylor expansion $\tilde \psi_{k,\ell}$ defined in \cite[Remark~3.2]{ABV} of the 
Bloch wave $\tilde \psi_k$ on the unit torus $\mathbb T$ (which is known to exist).
The latter satisfies $\|\tilde \psi_{k,\ell}\|_{L^2(\mathbb T)}=1+O(\kappa^{\ell+1})$ due to the choice of 
unit normalization of Bloch waves in $L^2(\mathbb T)$, whereas the former satisfies at best
$\|\psi_{k,\ell}\|_{L^2(\mathbb T)}=1+O(\kappa)$.
The advantage of proceeding that way is twofold: we do not require the terms of the expansion to have bounded $L^2$-norm  and the algebraic structure is easier to unravel.
In the above we also chose the phase to be zero at all orders --- the choice of phase is different in  \cite{CV,ABV}.
\qed
\end{rem}
\begin{rem}
Proposition~\ref{prop:bloch} is extended to systems in Appendix~\ref{sec:systems}, in which case the ``eigenvalues'' $\tilde \lambda_{k,\ell}$ are matrices rather than scalars.
\qed
\end{rem}
\begin{proof}[Proof of Proposition~\ref{prop:bloch}]

We first compute for $\ell\ge 1$
\begin{eqnarray*}\label{expansion2}
\lefteqn{-(\nabla +ik) \cdot \aa  (\nabla + ik ) \psi_{k,\ell}} 
\\
 &=& -(i\kappa) \nabla\cdot \aa (e+\nabla \varphi_1) -\sum_{j=2}^\ell (i\kappa)^j(\nabla \cdot \aa \nabla \varphi_j+e\cdot \aa \nabla \varphi_{j-1}+\nabla \cdot(\aa e\varphi_{j-1}) +e\cdot \aa e \varphi_{\ell-2}) \\ \nonumber
&&-(i\kappa)^{\ell+1}(e\cdot \aa\nabla \varphi_{\ell}+\nabla \cdot(\aa e\varphi_{\ell}) +e\cdot \aa e \varphi_{\ell-1})\\ \nonumber
&&-(i\kappa)^{\ell+2}e\cdot \aa e \varphi_{\ell}
\end{eqnarray*}
and, after several resummations:
\begin{eqnarray*}
\tilde \lambda_{k,\ell} \psi_{k,\ell} &=& \sum_{j=0}^\ell\sum_{p=0}^{\ell-1} \kappa^{2+j+p} i^{j+p}\lambda_{p}\phi_{j}
\\
&=&-\sum_{j=0}^{\ell-1} (i\kappa)^{2+ j} \sum_{p_1+p_2= j} \lambda_{p_2} \phi_{p_1}-(i\kappa)^{\ell+2}  \sum_{j=1}^\ell \sum_{p=\ell-j}^{\ell-1} (i\kappa)^{j+p-\ell}  \lambda_{p}\phi_{j}
\\
&=&-\sum_{j=2}^{\ell+1} (i\kappa)^{j} \sum_{p=0}^{j-2} \lambda_{p} \phi_{j-2-p}-(i\kappa)^{\ell+2}  \sum_{j=1}^\ell \sum_{p=\ell-j}^{\ell-1} (i\kappa)^{j+p-\ell}  \lambda_{p}\phi_{j}.
\end{eqnarray*}
Hence,
\begin{eqnarray*}
\lefteqn{-(\nabla +ik) \cdot \aa  (\nabla + ik ) \psi_{k,\ell}-\tilde \lambda_{k,\ell} \psi_{k,\ell} }\\
&=& -(i\kappa) \nabla\cdot \aa (e+\nabla \varphi_1) \\
&&-\sum_{j=2}^\ell (i\kappa)^j(\nabla \cdot \aa \nabla \varphi_j+e\cdot \aa \nabla \varphi_{j-1}+\nabla \cdot(\aa e\varphi_{j-1}) +e\cdot \aa e \varphi_{j-2}-\sum_{p=0}^{j-2} \lambda_{p} \phi_{j-2-p}) \\ \nonumber
&&-(i\kappa)^{\ell+1}(e\cdot \aa\nabla \varphi_{\ell}+\nabla \cdot(\aa e\varphi_{\ell}) +e\cdot \aa e \varphi_{\ell-1}-\sum_{p=0}^{\ell-1} \lambda_{p} \phi_{\ell-1-p})\\ \nonumber
&&-(i\kappa)^{\ell+2}(e\cdot \aa e \varphi_{\ell}-  \sum_{j=1}^\ell \sum_{p=\ell-j}^{\ell-1} (i\kappa)^{j+p-\ell}  \lambda_{p}\phi_{j}).
\end{eqnarray*}
By definition of $\phi_1$, the term of order $\kappa$ vanishes. Let us now reformulate the RHS terms of order $\kappa^j$ for $j=2,\dots,\ell+1$.
To this aim we note that in view of $\phi_0\equiv 1$ and of the definition of $\chi_{j-2}$ and $\chi_{j-1}$,
$$
\sum_{p=0}^{j-2} \lambda_{p} \phi_{j-2-p}\,=\, \lambda_{j-2}+\sum_{p=0}^{j-3} \lambda_{p} \phi_{j-2-p}\,=\,\lambda_{j-2}-\triangle \chi_{j-1}-\nabla \chi_{j-2}\cdot e,
$$
so that, by the skew-symmetry of $\sigma$ in form of $e \cdot \sigma_{j-2} e=0$ and by definition of $q_{j-1}$ and $\sigma_{j-1}$,
\begin{eqnarray*}
\lefteqn{e\cdot \aa \nabla \varphi_{j-1} +e\cdot \aa e \varphi_{j-2}-\sum_{p=0}^{j-2} \lambda_{p} \phi_{j-2-p}}
\\
&=&e\cdot \aa( \nabla \varphi_{j-1}+e \varphi_{j-2})-\lambda_{j-2}+\nabla \chi_{j-2}\cdot e +\triangle \chi_{j-1}-e\cdot \sigma_{j-2} e
\\
&=&e \cdot \underbrace{(  \aa( \nabla \varphi_{j-1}+e \varphi_{j-2})- \tilde \aa_{j-2} e +\nabla \chi_{j-2}\cdot e-\sigma_{j-2} e)}_{\dps =q_{j-1}} +\triangle \chi_{j-1}
\\
&=&\nabla \cdot (-\sigma_{j-1}e+\nabla \chi_{j-1}).
\end{eqnarray*}
Combined with the defining equation for $\phi_{j-1}$, this shows that
the terms of order $\kappa^j$ for $j=2,\dots,\ell$  vanish:
\begin{multline*}
\nabla \cdot \aa \nabla \varphi_j+e\cdot \aa \nabla \varphi_{j-1}+\nabla \cdot(\aa e\varphi_{j-1}) +e\cdot \aa e \varphi_{j-2}-\sum_{p=0}^{j-2} \lambda_{p} \phi_{j-2-p}
\\
=\, \nabla \cdot \aa \nabla \varphi_j + \nabla\cdot (-\sigma_{j-1}e+\nabla \chi_{j-1}+\aa e\varphi_{j-1})\,=\,0.
\end{multline*}
This also implies that the term of order $\kappa^{\ell+1}$ is in divergence form:
\begin{equation*}
e\cdot \aa\nabla \varphi_{\ell}+\nabla \cdot(\aa e\varphi_{\ell}) +e\cdot \aa e \varphi_{\ell-1}-\sum_{p=0}^{\ell-1} \lambda_{p} \phi_{\ell-1-p}\,
=\,  \nabla\cdot (-\sigma_{\ell}e+\nabla \chi_{\ell}+\aa e\varphi_{\ell}).
\end{equation*}
We have thus proved 
\begin{eqnarray*}
{-(\nabla +ik) \cdot \aa  (\nabla + ik ) \psi_\ell-\tilde \lambda_{k,\ell} \psi_{k,\ell} }
&=&-(i\kappa)^{\ell+1}\nabla\cdot (-\sigma_{\ell}e+\nabla \chi_{\ell}+\aa e\varphi_{\ell})\\ \nonumber
&&-(i\kappa)^{\ell+2}(e\cdot \aa e \varphi_{\ell}-  \sum_{j=1}^\ell \sum_{p=\ell-j}^{\ell-1} (i\kappa)^{j+p-\ell}  \lambda_{p}\phi_{j}),
\end{eqnarray*}
which is the claim by definition of the eigendefect.
\end{proof}

%%%%%%%%%%%%%%%%%
%%%%%%%%%%%%%%%%%
%%%%%%%%%%%%%%%%%
%%%%%%%%%%%%%%%%%
%%%%%%%%%%%%%%%%%
%%%%%%%%%%%%%%%%%

\section{Taylor-Bloch approximate solution of the wave equation}\label{sec:wave}

\subsection{Main result and structure of the proof}\label{sec:Kmax}

Let $u_\e \in L^\infty(\R_+,L^2(\R^d))$ be the unique weak solution of 
\begin{equation}\label{e.eq-ueps}
\left\{
\begin{array}{rcl}
\square_\e u_\e&=&0,
\\
u_\e(0,\cdot)&=&u_0,
\\
\partial_t u_\e(0,\cdot)&=&0,
\end{array}
\right.
\end{equation}
where $u_0\in \Sc$ (the Schwartz class).
Let $\hat u_0:k\mapsto \calF u_0(k)=\int_{\R^d} u_0(x)e^{-ik\cdot x} dx$ be the Fourier transform of $u_0$ (which is also in $\Sc$).
Let $\ell \ge 1$, and assume that the (higher-order) homogenized tensors $\bar \aa_j$ of Definition~\ref{defi:ext-corr} are well-defined for all $0\le j\le \ell-1$, and set $\bar \Gamma_\ell:=\max_{0\le j \le \ell-1}|\bar \aa_j|<\infty$.
Since for all $j\ge 0$, $e \cdot \tilde \aa_{2j+1}e=0$, the quantity $\tilde \lambda_{k,\ell}$ (cf. Definition~\ref{defi:Bloch}) is real-valued.
On the other hand, since $\bar \aa_0\ge \Id$ is positive-definite, there exists $\Km>0$ (which only depends on $\bar \Gamma_\ell$) such that $|k|^{-2}\tilde \lambda_{k,\ell}\ge \frac14$ for all $|k|\le \Km$.
Let $\omega_{\ell}:\R_+ \to [0,1]$ be a smooth cut-off function which takes value $1$ on $[0,\frac12 \Km]$ and value 0 on $[\Km,\infty)$.
We define the approximation $u_{\e,\ell}$ of $u_\e$ as the following inverse Fourier transform:
\begin{eqnarray}
u_{\e,\ell}(t,x)&:=&\calF^{-1}\big(\omega_{\ell}(\e |\cdot|) \hat u_0\cos(\e^{-1} \Lambda_\ell(\e \cdot)t)\big)(x)\nonumber
\\
&=&  \frac{1}{(2\pi)^d} \int_{\R^d} \omega_{\ell}(\e |k|) \hat u_0(k) e^{ik\cdot x} \cos(\e^{-1} \Lambda_\ell(\e k)t)dk,\label{e.def-uepsell}
\end{eqnarray}
where $\Lambda_\ell(k)\,:=\,  \sqrt{\tilde \lambda_{k,\ell}}$, which is well-defined since $\tilde \lambda_{\e k,\ell}\ge 0$ when $\omega_{\ell}(\e |k|) \ne 0$.
Note that $u_{\e,\ell}$ is real-valued since $\Lambda_\ell(-k)=\Lambda_\ell(k)$, so that $k\mapsto   \omega_{\ell}(\e |k|)\hat u_0(k)\cos(\e^{-1} \Lambda_\ell(\e k)t)$ is Hermitian.
The main result of this section establishes the accuracy of the approximation of $u_\e$ by $u_{\e,\ell}$ over large times depending on the growth of the 
extended correctors, which we formulate as an assumption (we refer the reader to Appendix~\ref{sec:corr} for details).
\begin{hypo}\label{hypo:corr} 
Let $\ell\ge 1$, and assume that the extended correctors $(\phi_j,\sigma_j,\chi_j)_{0\le j\le \ell}$ are well-defined for all directions $e \in \R^d$, $|e|=1$, and satisfy the following
growth properties:
\begin{itemize}
\item for all $0\le j\le \ell-1$, $\phi_j,\sigma_j,\nabla \chi_j$ are $\Z^d$-stationary and satisfy (uniformly over $e$)
\begin{equation}\label{hyp:growth1}
\sup_{x\in \R^d} \expec{\fint_{B(x)}|\phi_j|^2+|\sigma_j|^2+|\nabla \chi_j|^2}^\frac12 \,\lesssim\, 1
\end{equation}
where $B(x)$ denotes the unit ball centered at $x$;
\item for $j=\ell$, there exists $\alpha=(\alpha_1,\alpha_2)$ with $0\le \alpha_1<1$ and $0\le \alpha_2$ such that for all $x\in \R^d$ (uniformly over $e$)
\begin{equation}\label{hyp:growth2}
\expec{\fint_{B(x)} |\phi_j|^2+|\sigma_j|^2+|\nabla \chi_j|^2}^\frac12 \,\lesssim\, \mu_{\alpha}(|x|),
\end{equation}
where $\mu_{\alpha}(t)=(1+t)^{\alpha_1}\log^{\alpha_2}(2+t)$.
\end{itemize}
\qed
\end{hypo}
\begin{rem}
If the RHS of \eqref{hyp:growth2} is also a lower bound for the left-hand side (LHS in short) of  \eqref{hyp:growth2}, then correctors of order $\ell+1$ and more are not well-defined.
\qed
\end{rem}
In the rest of this article, $\lesssim$ stands for $\le C\times$ for a generic constant $C$ that might depend on $d$, $\ell$,  $\bar \Gamma_\ell$,
and the multiplicative constants in \eqref{hyp:growth1} \& \eqref{hyp:growth2} (but not on $T$, $\e>0$, and $u_0$).
\begin{theo}\label{t.wave-eq}
Let $\ell\ge 1$, and assume that $(\phi_j,\sigma_j,\chi_j)_{0\le j\le \ell}$ satisfy Hypothesis~\ref{hypo:corr}
for some $\alpha\in [0,1)\times \R_+$. Then we have for all $T\ge 0$ and all $1\ge \e>0$,
\begin{equation}\label{e.wave-eq-estim}
\sup_{t\le T} \expec{\|u_\e-u_{\e,\ell}\|_{L^2(\R^d)}^2}^\frac12\,
 \lesssim \, C_\ell(u_0) (\max\{\e,\e^{\ell}\mu_\alpha(\e^{-1})\}+\e^{\ell}T\mu_{\alpha}(\e^{-1}T)),
\end{equation}
where $C_\ell(u_0)$ is a generic norm of $u_0$ (that may change from line to line in the proofs) which only depends on $\ell$ and $d$, and is finite for $u_0\in \Sc$.
\qed
\end{theo}
\begin{rem}
For applications, the integer $\ell$ can be chosen arbitrary large  in Theorem~\ref{t.wave-eq} for periodic and (diophantine) quasi-periodic coefficients, whereas
for almost periodic and random (with decaying correlations at infinity) coefficients there is a maximal $\ell$ for which  Theorem~\ref{t.wave-eq} holds ---  see Appendix \ref{sec:corr} about the growth of correctors and the relation to Hypothesis~\ref{hypo:corr} . \qed
\end{rem}
\begin{rem}\label{rem:odd-even}
Since $\Lambda_{2\ell+1}\equiv\Lambda_{2\ell}$, we have $u_{\e,2\ell+1}\equiv u_{\e,2\ell}$, so that if \eqref{hyp:growth1} \&~\eqref{hyp:growth2} hold with $\ell$ replaced by $2\ell+1$, estimate \eqref{e.wave-eq-estim}
takes the equivalent form
\begin{equation}\label{e.wave-eq-estim-rem}
\sup_{t\le T} \expec{\|u_\e-u_{\e,2\ell}\|_{L^2(\R^d)}^2}^\frac12\, \lesssim \, C_{2\ell+1}(u_0) (\e+\e^{2\ell+1}T\mu_\alpha(\e^{-1}T)).
\end{equation}
\qed
\end{rem}
\begin{rem}\label{rem:sto-int}
In Hypothesis~\ref{hypo:corr} the growth of the correctors is measured in terms of their second stochastic moments. There is nothing
special about this quantity. If higher stochastic integrability is assumed, the result of Theorem~\ref{t.wave-eq} will hold with the corresponding higher stochastic integrability, as well as all the other results of this contribution. We consider this as a separate issue.
\qed
\end{rem}
\begin{rem}
Estimate \eqref{e.wave-eq-estim} does not necessarily improve as $\ell$ gets larger for large times.
Indeed, there is an interplay between the growth of the corrector and the final time: if the corrector $\varphi_{\ell}$ is unbounded, the $L^2(\R^d)$-norm of
the approximation of the solution blows up at large times whereas the $L^2(\R^d)$-norm of the solution remains bounded.
\qed
\end{rem}
\begin{rem}\label{rem:forcing}
Instead of considering the initial-value problem \eqref{e.eq-ueps} in Theorem~\ref{t.wave-eq} we can also consider the more general problem
\begin{equation*}
\left\{
\begin{array}{rcl}
\square_\e u_\e&=&f,
\\
u_\e(0,\cdot)&=&u_0,
\\
\partial_t u_\e(0,\cdot)&=&v_0,
\end{array}
\right.
\end{equation*}
where $f$ is compactly supported in time and in the Schwartz class in space, and $u_0$ and $v_0$ are in the Schwartz class. 
For $u_0=v_0=0$, we can in addition
prove similar results in the energy norm (by taking into account the correctors). We refer the reader to Appendix~\ref{sec:forcing} for 
this variant. 
\qed
\end{rem}
\begin{rem}
A similar result holds for systems, cf.~Appendix~\ref{sec:systems}.\qed
\end{rem}

The proof of Theorem~\ref{t.wave-eq} relies on three arguments:
\begin{itemize}
\item the fact that the initial condition can be replaced by a well-prepared initial condition $u_{0,\e,\ell}$ up to an error uniformly small in time by the assumptions \eqref{hyp:growth1} \&~\eqref{hyp:growth2}, cf.~Lemma~\ref{lem:prepare};
\item the fact that the Taylor-Bloch waves almost diagonalize the wave operator, and that the error due to the eigendefect can be controlled by suitable energy estimates on the wave equation 
with well-prepared initial condition and the assumptions \eqref{hyp:growth1} \&~\eqref{hyp:growth2}, cf.~Proposition~\ref{prop:energy-wave};
\item the fact that the almost solution $\tilde v_{\e,\ell}$ of the wave equation with well-prepared condition $u_{0,\e,\ell}$ can be well-approximated by $u_{\e,\ell}$ by the assumptions \eqref{hyp:growth1} \&~\eqref{hyp:growth2}, cf.~Lemma~\ref{lem:throw-away}.
\end{itemize}
In the following, we shall systematically use the notation $k=\kappa e$ with $\kappa \in \R$ and $|e|=1$,
and shall make the dependence of the correctors $\phi_j$ upon the direction $e$ explicit, and use the notation $\phi_j^e$.
We start with the preparation of the initial condition:
\begin{lemma}\label{lem:prepare}
Let $\ell\ge 1$, and let $u_{0,\e,\ell} \in L^2(\R^d)$ be defined by 
$$
u_{0,\e,\ell}\,:=\,\sum_{j=0}^{\ell}\varepsilon^j \varphi_j\big( \frac{x}{\varepsilon}\big) \cdot \nabla^j u_{0,\e}(x),
$$
where $u_{0,\e}:=\calF^{-1}(\omega_{\ell}(\e|\cdot|)\hat u_0)$, and $\varphi_j$ stands for the (symmetric) $j$-th order tensor such that $\varphi_j \cdot e^{\otimes j}=\varphi_j^e$, the $j$-th corrector in direction $e$.
Consider the unique weak solution $v_{\e,\ell}\in L^\infty(\R_+,L^2(\R^d))$ of the initial value problem
\begin{equation}\label{e.eq-velleps}
\left\{
\begin{array}{rcl}
\square_\e v_{\e,\ell}&=&0,
\\
v_{\e,\ell}(0,\cdot)&=&u_{0,\e,\ell},
\\
\partial_t v_{\e,\ell}(0,\cdot)&=&0.
\end{array}
\right.
\end{equation}
Then if Hypothesis~\ref{hypo:corr} holds, we have 
\begin{equation*}
\sup_{0 \leq t<\infty}\expec{ \|u_\e-v_{\e,\ell}\|_{L^2(\R^d)}^2}^\frac12 \,\lesssim \, C_\ell(u_0) \max\{\e,\e^{\ell}\mu_\alpha(\e^{-1})\}.
\end{equation*}
\qed
\end{lemma}
The initial condition $u_{0,\e,\ell}$ of Lemma~\ref{lem:prepare} also takes the following form (cf. Step~1 in the proof of Proposition~\ref{prop:energy-wave} below)
$$
u_{0,\e,\ell}\,=\,\frac{1}{(2\pi)^d}\int_{\R^d} \omega_{\ell}(\e|k|)\hat u_0(k)e^{ik\cdot x} \psi_{\e k,\ell} \big(\frac x\e\big) \ dk,
$$
which shows that it is well-prepared in terms of the Taylor-Bloch expansion.
We then turn to the core of the proof: the use of the Taylor-Bloch expansion and the eigendefects.
\begin{prop}\label{prop:energy-wave}
For $\ell\ge 1$, let $\tilde v_{\e,\ell} \in L^\infty(\R_+,L^2(\R^d))$ be the (real-valued) function defined by
\begin{equation}\label{e.def-vtildeepsell}
\tilde v_{\e,\ell}(t,x) \,:=\,\frac{1}{(2\pi)^d}\int_{\R^d} \omega_{\ell}(\e|k|)\hat u_0(k)e^{ik\cdot x}  \cos(\e^{-1} \Lambda_\ell(\e k)t) \psi_{\e k,\ell} \big(\frac x\e\big)\ dk,
\end{equation}
and let $v_{\e,\ell}$ be the unique weak solution of \eqref{e.eq-velleps}.
Then if Hypothesis~\ref{hypo:corr} holds we have for all $T\ge 0$ and $\e>0$,
\begin{equation*}
\sup_{t\le T} \expec{ \|v_{\e,\ell}-\tilde v_{\e,\ell}\|_{L^2(\R^d)}^2}^\frac12 \,\lesssim \, C_\ell(u_0)\e^{\ell}T\mu_{\alpha}(\e^{-1}T).
\end{equation*}
\qed
\end{prop}
We conclude with the approximation of $\tilde v_{\e,\ell}$ by neglecting the corrector terms.
\begin{lemma}\label{lem:throw-away}
For $\ell \ge 1$, let $u_{\e,\ell}$ and $\tilde v_{\e,\ell}$ be given by \eqref{e.def-uepsell} and \eqref{e.def-vtildeepsell}, respectively.
Then if Hypothesis~\ref{hypo:corr} holds we have for all $\e>0$ and for all $T>0$,
\begin{equation*}
\sup_{0\leq t\leq T} \expec{\|u_{\e,\ell}-\tilde v_{\e,\ell}\|_{L^2(\R^d)}^2}^\frac12 \,\lesssim \, C_\ell(u_0) \big(\e+\e^{\ell}\mu_\alpha(\e^{-1}T)\big).
\end{equation*}
\qed
\end{lemma}
Theorem~\ref{t.wave-eq} is a straightforward consequence of the combination of Lemma~\ref{lem:prepare}, Proposition~\ref{prop:energy-wave}, and  Lemma~\ref{lem:throw-away},
which are proved in the following three subsections.

\subsection{Proof of Lemma~\ref{lem:prepare}: Preparation of the initial condition}

We split the proof into two steps.
In the first step, we prove by an energy estimate that it is enough to control
the initial error $u_0-u_{0,\varepsilon , \ell}$ in $L^2(\R^d)$, which we estimate 
in Step~2 using Hypothesis~\ref{hypo:corr}.

\medskip

\step{1} Energy estimate for \eqref{i1}.

\noindent
We first derive an energy estimate for the wave equation \eqref{i1}. We integrate once \eqref{i1} in time,  multiply by $u_\varepsilon$, and integrate over $\left[0,t \right] \times \mathbb{R}^d$ to obtain
$$
\frac{1}{2}\left(\int_0^t \frac{d}{ds}\Vert u_\varepsilon(s,\cdot)\Vert^2_{L^2(\mathbb{R}^d)}ds+\int_0^t \frac{d}{ds}\Big\| \sqrt{\aa_\e}\nabla \int_{0}^s u_\varepsilon(r,\cdot)dr\Big\|_{L^2(\mathbb{R}^d)} ^2ds \right)=0, 
$$
from which, by uniform ellipticity of $\aa$, it immediately follows that 
\begin{equation}\label{estimation_initiale}
\Vert u_\varepsilon(t,\cdot)\Vert^2_{L^2(\mathbb{R}^d)}\leq \Vert u_0\Vert^2_{L^2(\mathbb{R}^d)}. 
\end{equation}
By linearity of \eqref{i1}, $u_\e -v_{\e , \ell}$ satisfies \eqref{i1} with initial condition $u_0-u_{0,\e, \ell}$. 
In view of the energy estimates \eqref{estimation_initiale} it is sufficient to estimate $\Vert u_{0,\e, \ell}-u_0\Vert_{L^2(\mathbb{R}^d)}$ to conclude the proof.

\medskip

\step{2} Estimate of $\Vert u_{0,\e, \ell}-u_0\Vert_{L^2(\mathbb{R}^d)}$.

\noindent By definition of $u_{0,\varepsilon , \ell}$,
\begin{equation*}
(u_{0,\varepsilon, \ell} -u_0)(x) \,=\,(u_{0,\e}-u_0)(x)+\sum_{j=1}^{\ell}\varepsilon^j \varphi_j\big( \frac{x}{\varepsilon}\big) \cdot \nabla^j u_{0,\e}(x),
\end{equation*}
so that
\begin{equation*}
\Vert u_{0,\varepsilon, \ell} -u_0\Vert_{L^2(\mathbb{R}^d)}^2 \leq \int_{\mathbb{R}^d} |(u_{0,\e}-u_0)(x)|^2 dx+\sum_{j=1}^{\ell}\varepsilon^{2j}\int_{\mathbb{R}^d} \big(\fint_{B_\e(x)} \big| \varphi_j\big( \frac{\cdot}{\varepsilon}\big)\big|^2\big)
\sup_{B_\e(x)}  |\nabla^j u_{0,\e}|^2 dx.
\end{equation*}
We start with the first RHS term, for which the Plancherel identity yields
\begin{eqnarray*}
\int_{\mathbb{R}^d} |(u_{0,\e}-u_0)(x)|^2dx&=&\frac1{(2\pi)^d}\int_{\R^d}(1-\omega_{\ell}(\e|k|))^2 |\hat u_0(k)|^2 dk
\\ 
&\le&\int_{|k|\ge \frac1{2\e} \Km}  |\hat u_0(k)|^2 dk\,\leq\, \Big(\frac{2\e}\Km\Big)^2 \int_{\R^d} |\nabla u_0|^2  \,\lesssim \, \e^2 C_\ell(u_0)^2.
\end{eqnarray*}
Taking the expectation of the other RHS terms and using assumption \eqref{hyp:growth1} for $0<j\le \ell-1$  and \eqref{hyp:growth2} for $j=\ell$, we obtain
\begin{multline*}
\expec{\sum_{j=1}^{\ell}\varepsilon^{2j}\int_{\mathbb{R}^d} \big(\fint_{B_\e(x)} \big| \varphi^e_j\big( \frac{\cdot}{\varepsilon}\big)\big|^2\big)
\sup_{B_\e(x)}  |\nabla^j u_{0,\e}|^2 dx} 
\\
\lesssim \sum_{j=1}^{\ell-1}\varepsilon^{2j} \int_{\mathbb{R}^d} \sup_{B_\varepsilon(x)}|\nabla^j u_{0,\e}|^2 dx+\varepsilon^{2\ell} \int_{\mathbb{R}^d} \sup_{B_\varepsilon(x)}|\nabla^j u_{0,\e}|^2 \mu_\alpha(|\frac x\e |)^2 dx.
\end{multline*}
Using the crude bound $\mu_\alpha({\left|\frac x\e\right|})\lesssim \mu_\alpha(\e^{-1})(1+|x|)$ together with the bound
\begin{eqnarray*}
\int_{\mathbb{R}^d} \sup_{B_\varepsilon(x)}|\nabla^j u_{0,\e}|^2 (1+|x|)^2 dx &\leq & \int_{\mathbb{R}^d} \Big(\sup_{\R^d} \{(1+|\cdot|)^{\frac{d+3}{2}} |\nabla^j u_{0,\e}|\}\Big)^2 (1+|x|)^{-d-1} dx 
\\
&\lesssim & \Big(\sup_{\R^d} \{(1+|\cdot|)^{\frac{d+3}{2}} |\nabla^j u_{0,\e}|\}\Big)^2 \le C_\ell(u_0)^2,
\end{eqnarray*}
the claim follows.

\subsection{Proof of Proposition~\ref{prop:energy-wave}: Almost diagonalization of the wave operator}
We split the proof into three steps.
In the first step we reformulate the initial condition in terms of a Taylor-Bloch expansion.
Doing so we may exploit that Taylor-Bloch waves approximately diagonalize the wave operator to write
an approximation of the solution as the explicit time-integration of the initial Taylor-Bloch expansion.
Since the diagonalization is approximate, the approximate solution does only solve the wave equation up to a remainder term, cf.~Step~2.
This term can be written using the eigendefect, which allows us to reformulate the error as a source term 
in conservative form plus higher-order term. Step~3 is then dedicated to the derivation of energy estimates to control the error 
generated by this source term, at the level of the solution. 
These energy estimates are rather subtle: they rely on the specific form of the source term (which allows explicit integrations by parts in space and time),
and on estimates of Fourier symbols.

\medskip

\step1 Reformulation of $u_{0,\e,\ell}$ in term of Taylor-Bloch waves.

\noindent We claim that
$$
u_{0,\e,\ell}\,=\,\frac{1}{(2\pi)^d}\int_{\R^d} \hat u_{0,\e}(k)e^{ik\cdot x} \psi_{\e k,\ell} \big(\frac x\e\big) \ dk,
$$
where we recall that $\hat u_{0,\e}(k)$ is a short-hand notation for  $\omega_{\ell}(\e|k|)\hat u_0(k)$
(note that $\psi_{\e k,\ell}$  is Hermitian by construction).
Indeed,
\begin{eqnarray*}
u_{0,\varepsilon, \ell}(x)&=&\sum_{j=0}^{\ell}\varepsilon^j \varphi_j\big( \frac{x}{\varepsilon}\big) \cdot \nabla^j u_{0,\e}(x)
\\
&=&\frac{1}{(2\pi)^d} \sum_{j=0}^{\ell}\varepsilon^j \varphi_j\big( \frac{x}{\varepsilon}\big) \cdot \int_{\mathbb{R}^d}(ik)^{\otimes j}\hat u_{0,\e}(k)e^{ik\cdot x}dk
\\
 &=&\frac{1}{(2\pi)^d}\int_{\mathbb{R}^d}\hat u_{0,\e}(k)e^{ik\cdot x}\sum_{j=0}^{\ell} (i\e k)^{\otimes j}\cdot \varphi_j^e\big( \frac{x}{\varepsilon}\big)dk,
\end{eqnarray*}
from which the claim follows by definition of $ \psi_{\e k,\ell}$.

\medskip

\step{2} Equation satisfied by $\widetilde{v}_{\varepsilon ,\ell}-{v}_{\varepsilon ,\ell}$.

\noindent 
In view of Proposition~\ref{prop:bloch}, $\widetilde{v}_{\varepsilon,\ell}$ satisfies 
\begin{multline*}
\square_\varepsilon \widetilde{v}_{\varepsilon ,\ell}  =\\ -\frac{\varepsilon^{\ell-1}}{(2\pi)^d}\int_{\mathbb{R}^d} (i\kappa)^{\ell+1} \hat u_{0,\e}(k)e^{ik\cdot x}(\nabla \cdot (-\sigma_{\ell}^e e+\aa e\varphi_{\ell}^e +\nabla \chi_{\ell}^e))\big(\frac{x}{\varepsilon}\big)\cos(\varepsilon^{-1}\Lambda_\ell(\varepsilon k )t)dk \\ 
-\frac{\varepsilon^{\ell}}{(2\pi)^d}\int_{\mathbb{R}^d}(i\kappa)^{\ell+2}\hat u_{0,\e}(k)e^{ik\cdot x} \big( e\cdot \aa e\varphi_{\ell}^e-\sum_{j=1}^{\ell} \sum_{p=\ell-j}^{\ell} (i\e\kappa)^{j+p-\ell}  \lambda_{p}\phi_{j}^e\big)
\big( \frac{x}{\varepsilon}\big)\cos(\varepsilon^{-1}\Lambda_\ell(\varepsilon k )t)dk.
\end{multline*}
Using that $(\nabla\cdot f)\left( \frac{x}{\varepsilon}\right)=\varepsilon \nabla\cdot ( f( \frac{x}{\varepsilon}))$, the above turns into
\begin{equation}
\square_\varepsilon \widetilde{v}_{\varepsilon ,\ell}  = -\frac{\varepsilon^{\ell}}{(2\pi)^d}\int_{\mathbb{R}^d}(i\kappa)^{\ell+1}\hat u_{0,\e}(k)
\Big( \nabla \cdot \big(e^{ik\cdot x}\Phi_{1,\ell}^e\big( \frac{x}{\varepsilon}\big) \big)+i\kappa e^{ik\cdot x}\Phi_{2,\ell,\kappa}^e\big( \frac{x}{\varepsilon}\big) \Big) \cos(\varepsilon^{-1}\Lambda_\ell(\varepsilon k )t) dk 
\end{equation}
where 
$$\Phi_{1,\ell}^e:=-\sigma_{\ell}^e e+\aa e\varphi_{\ell}^e +\nabla \chi_{\ell}^e \text{, } \quad \Phi_{2,\ell,\kappa}^e:=e\cdot \aa e\varphi_{\ell}^e-\sum_{j=1}^{\ell} \sum_{p=\ell-j}^{\ell} (i\e\kappa)^{j+p-\ell}  \lambda_{p}\phi_{j}^e$$  
are linear combinations of correctors.
Hence, $\widetilde{v}_{\varepsilon ,\ell}-v_{\e,\ell}$ solves the following equation 
\begin{equation}
\label{ondes_interieur}
\left\{
\begin{array}{rcl}
\square_\e (\widetilde{v}_{\varepsilon ,\ell}-v_{\e,\ell})(t,x)&=&\varepsilon^{\ell}(f_{\varepsilon , 1,\ell} +f_{\varepsilon , 2,\ell} ) ,
\\
\partial_t (\widetilde{v}_{\varepsilon ,\ell}-v_{\e,\ell})(0,\cdot)&=&0,
\\
(\widetilde{v}_{\varepsilon ,\ell}-v_{\e,\ell})(0,\cdot)&=&0.
\end{array}
\right.
\end{equation}
with the source terms given by 
\begin{eqnarray*}
f_{\e ,1,\ell}(t,x)&:=&-\frac{1}{(2\pi)^d} \int_{\mathbb{R}^d}(i\kappa)^{\ell+1} \hat u_{0,\e}(k)\nabla \cdot \left(e^{ik\cdot x}\Phi_{1,\ell}^e\big( \frac{x}{\varepsilon}\big) \right)\cos(\varepsilon^{-1}\Lambda_\ell(\varepsilon k )t)dk,
\\
 f_{\e ,2,\ell}(t,x)&:=&-\frac{1}{(2\pi)^d}\int_{\mathbb{R}^d}(i\kappa)^{\ell+2} \hat u_{0,\e}(k) e^{ik\cdot x} \Phi_{2,\ell,\kappa}^e\big( \frac{x}{\varepsilon}\big) \cos(\varepsilon^{-1}\Lambda_\ell(\varepsilon k )t)dk .
\end{eqnarray*}
In the third and last step, we establish energy estimates for equation~\eqref{ondes_interieur}. 

\medskip

\step{3} Energy estimates for \eqref{ondes_interieur}.

\noindent
By linearity of \eqref{ondes_interieur}, it is sufficient to consider the following two auxiliary wave equations:
\begin{equation}\label{sys_annexe}
\left\{
\begin{array}{rcl}
\square_\e v_{\e ,{p},\ell}(t,x)&=&\varepsilon^{\ell} f_{\e ,{p},\ell},
\\
\partial_t v_{\e ,p,\ell}(0,\cdot)&=&0,
\\
v_{\e ,p,\ell}(0,\cdot)&=&0,
\end{array}
\right. ,\: {p}\in \left\lbrace 1,2\right\rbrace .
\end{equation}
Since for ${p}=1$ \eqref{sys_annexe} involves a divergence term, the energy estimate is not completely standard and relies very much on the specific form of the RHS.
We give a complete proof of the desired estimate in Substeps~3.1 and~3.2.
The proof of the estimate for ${p}=2$, which takes a similar form as for ${p}=1$, is simpler and left to the reader.

\medskip

\substep{3.1} Estimate of  $\expec{\Vert \partial_t v_{\e ,1,\ell}(t,\cdot)\Vert_{L^2(\mathbb{R}^d)}^2+\Vert \nabla v_{\e ,1,\ell}(t,\cdot)\Vert_{L^2(\mathbb{R}^d)}^2}$.

\noindent  Multiplying \eqref{sys_annexe} by $\partial_t v_{\e ,1,\ell}$ and integrating over $\left[0,t \right] \times \mathbb{R}^d $, we obtain
by ellipticity of $\aa$
\begin{eqnarray}\label{step11}
\frac{1}{2}\left(\Vert \partial_t v_{\e ,1,\ell}(t,\cdot)\Vert_{L^2(\mathbb{R}^d)}^2 + \Vert \nabla v_{\e ,1,\ell}(t,\cdot)\Vert_{L^2(\mathbb{R}^d)}^2\right) \leq \frac{\varepsilon^{\ell}}{(2\pi)^d}  I_{1,\ell},
\end{eqnarray}
where 
\begin{multline*}
I_{1,\ell}:=\\-\int_{\left[0,t \right] \times \mathbb{R}^d}\left(\int_{\mathbb{R}^d}(i\kappa)^{\ell+1}\hat u_{0,\e}(k)\nabla\cdot \left( e^{ik\cdot x}\Phi_{1,\ell}^e\left(\frac{x}{\varepsilon}\right)\right)\cos(\varepsilon^{-1}\Lambda_\ell(\varepsilon k )s)dk\right)\partial_t v_{\varepsilon ,1,\ell}(s,x)dsdx .
\end{multline*}
The subtlety is the divergence term which is not bounded uniformly in $L^2(\R^d)$ with respect to $\e$.
To obtain a suitable energy estimate, we first integrate by parts in space. This yields the term $\nabla \partial_t v_{\varepsilon ,1,\ell}$, which 
we do not control a priori. What makes the argument possible is that one may in turn integrate by parts in time, and end up with the quantity $\nabla  v_{\varepsilon ,1,\ell}$ which we
can then absorb  in the LHS of \eqref{step11}.
More precisely, by Fubini's theorem and integration by parts in space and time, $I_{1,\ell}$ takes the form 
\begin{eqnarray*}\label{step12}
\lefteqn{I_{1,\ell}}
\\
&=& \int_{\left[0,t \right] \times \mathbb{R}^d}(i\kappa)^{\ell+1}\hat u_{0,\e}(k)\cos\left(  \varepsilon^{-1}\Lambda_\ell(\varepsilon k)s\right) \int_{\mathbb{R}^d}\nabla \partial_t v_{\e ,1,\ell}(s,x)\cdot e^{ik\cdot x}\Phi_{1,\ell}^e\big(\frac{x}{\varepsilon}\big)dxdsdk  \\ \nonumber
&=&\int_{\mathbb{R}^d}\nabla v_{\e ,1,\ell}(t,x)\cdot \int_{\mathbb{R}^d}(i\kappa)^{\ell+1}\hat u_{0,\e}(k)e^{ik\cdot x}\Phi_{1,\ell}^e\big(\frac{x}{\varepsilon}\big)\cos\left(  \varepsilon^{-1}\Lambda_\ell(\varepsilon k) t\right)dk dx\\ \nonumber
&+&\int_{\left[0,t \right]\times\mathbb{R}^d}\nabla v_{\varepsilon ,1,\ell}(s,x)\cdot \int_{\mathbb{R}^d}(i\kappa)^{\ell+1}\hat u_{0,\e}(k)e^{ik\cdot x}\Phi_{1,\ell}^e\big(\frac{x}{\varepsilon}\big)\frac{\Lambda_\ell(\varepsilon k )}{\varepsilon}\sin\left(  \varepsilon^{-1}\Lambda_\ell(\varepsilon k  )s\right)dk dx ds .
\end{eqnarray*}
Let $F_{1,t,\e,\ell}$ and $F_{2,t,\e,\ell}$ be the linear operators from $\Sc$ to $\Sc$ characterized by their Fourier symbols 
\begin{eqnarray}
\hat F_{1,t,\e,\ell}(k)&:=&(i\kappa)^{ \ell+1}\omega_{\ell}(\e|k|)\cos\left(  \varepsilon^{-1}\Lambda_\ell(\varepsilon k) t\right),\label{e.Fourier-sym1}\\
\hat F_{2,t,\e,\ell}(k)&:=&(i\kappa)^{\ell+1}\omega_{\ell}(\e|k|)\frac{\Lambda_\ell(\varepsilon k )}{\varepsilon}\sin\left(  \varepsilon^{-1}\Lambda_\ell(\varepsilon k  )t\right).\label{e.Fourier-sym2}
\end{eqnarray}
Note that the dependence on $t$ of these Fourier symbols and of their derivatives (with respect to $k$) will be crucial
for the estimates to come.
Proceeding as in Step~2 of the proof of Lemma~\ref{lem:prepare}, and using assumptions~\eqref{hyp:growth1} \&~\eqref{hyp:growth2}, we then obtain
\begin{multline*}
\expec{\int_{\mathbb{R}^d} \left|\int_{\mathbb{R}^d}(i\kappa)^{\ell+1}\hat{u}_{0,\e}(k)e^{ik\cdot x}\Phi_{1,\ell}^e\big(\frac{x}{\varepsilon}\big)\cos\left(  \Lambda_\ell(\varepsilon\vert k \vert )t\right)dk \right|^2 dx} \\
\lesssim \, \int_{\mathbb{R}^d}  \sup_{B_\varepsilon(x)}|F_{1,t,\e,\ell} u_{0}|^2 \mu_\alpha({|\frac x\e |})^2dx 
\end{multline*}
and 
\begin{multline*}
\expec{\int_{\mathbb{R}^d} \left|\int_{\mathbb{R}^d}(i\kappa)^{\ell+1}\hat{u}_{0,\e}(k)e^{ik\cdot x}\Phi_{1,\ell}^e\big(\frac{x}{\varepsilon}\big)\frac{\Lambda_\ell(\varepsilon k )}{\varepsilon}\sin\left(  \varepsilon^{-1}\Lambda_\ell(\varepsilon k  )s\right)dk \right|^2 dx}
\\
\lesssim  \, \int_{\mathbb{R}^d}  \sup_{B_\varepsilon(x)}|F_{2,t,\e,\ell} u_0|^2 \mu_\alpha({|\frac x\e |})^2dx ,
\end{multline*}
so that 
\begin{multline*}
\expec{I_{1,\ell}}\,\lesssim \, \expec{\| \nabla v_{\e ,1,\ell}(t,\cdot)\|_{L^2(\R^d)}^2}^\frac12 \Big(\int_{\mathbb{R}^d}  \sup_{B_\varepsilon(x)}|F_{1,t,\e,\ell} u_0|^2 \mu_\alpha({|\frac x\e|})^2dx\Big)^\frac12
\\
+ \int_0^t \expec{\| \nabla v_{\e ,1,\ell}(s,\cdot)\|_{L^2(\R^d)}^2}^\frac12 \Big(\int_{\mathbb{R}^d}  \sup_{B_\varepsilon(x)}|F_{2,s,\e,\ell} u_0|^2 \mu_\alpha({|\frac x\e |})^2dx\Big)^\frac12ds.
\end{multline*}
Combined with \eqref{step11} and Young's inequality, this yields the energy estimate for all $T\ge 0$
\begin{multline}\label{e.bound-before-symbol}
\sup_{0\le t\le T} \expec{ \Vert \partial_t v_{\e ,1,\ell}(t,\cdot)\Vert_{L^2(\mathbb{R}^d)}^2 + \Vert \nabla v_{\e ,1,\ell}(t,\cdot)\Vert_{L^2(\mathbb{R}^d)}^2}^\frac12
\\
\lesssim\,  \varepsilon^{\ell}  \Big(\int_{\mathbb{R}^d}  \sup_{B(x)}|F_{1,T,\e,\ell} u_0|^2 \mu_\alpha(|\frac x\e|)^2dx\Big)^\frac12
\\
+  \varepsilon^{\ell} \int_0^T  \Big(\int_{\mathbb{R}^d}  \sup_{B(x)}|F_{2,s,\e,\ell} u_0|^2 \mu_\alpha(|\frac x\e|)^2dx\Big)^\frac12ds.
\end{multline}
It remains to reformulate the RHS.
Recall that $\mu_\alpha(t)=(1+t)^{\alpha_1}\log^{\alpha_2}(2+t)$ for some $0\le \alpha_1 <1$ and some $\alpha_2\ge 0$.
We assume without loss of generality that $\alpha_2>0$ (otherwise the proof is simpler).
Starting point if the following elementary inequality: There exists $C<\infty$ such that for all $\beta>0$ and all $t\ge 0$,
$$
\log(1+t)\,\le \, \frac C\beta t^\beta,
$$
from which we deduce 
$$
\mu_\alpha(t) \,\le \, C^{\alpha_1+\alpha_2}(1+ t^{\alpha_1+\beta \alpha_2} \beta^{-\alpha_2}) \,\lesssim\, 1+ t^{\alpha_1+\beta \alpha_2} \beta^{-\alpha_2}.
$$
We then combine this estimate for $\beta\ll 1$ small enough so that $\alpha_1+\beta\alpha_2<1$ with H\"older's inequality in the form
\begin{eqnarray*}
\lefteqn{\int_{\mathbb{R}^d} |h|^2 \mu_\alpha(|\frac x\e|)^2dx}
\\
&\lesssim & \int_{\mathbb{R}^d} |h|^2 dx+\e^{-2(\alpha_1+\beta \alpha_2)}\beta^{-2\alpha_2}\int_{\mathbb{R}^d} |h|^2  |x|^{2(\alpha_1+\beta \alpha_2)} dx
\\
&\le &  \int_{\mathbb{R}^d} |h|^2 dx+\e^{-2(\alpha_1+\beta \alpha_2)}\beta^{-2\alpha_2} \Big(\int_{\mathbb{R}^d} |h|^2 dx\Big)^{1-\alpha_1-\beta \alpha_2} 
 \Big(\int_{\mathbb{R}^d} |h|^2 |x|^2 dx\Big)^{\alpha_1+\beta \alpha_2} .
\end{eqnarray*}
By the Sobolev embedding  
$\sup_{B(x)}|F_{1,T,\e,\ell} u_0|^2 \lesssim \|F_{1,T,\e,\ell} u_0\|_{H^m(B(x))}^2$ for some $m$ depending only on the dimension, this inequality takes the form
\begin{multline*}
{\int_{\mathbb{R}^d}\sup_{B(x)}|F_{1,T,\e,\ell} u_0|^2 \mu_\alpha(|\frac x\e|)^2dx}
\,
\le \, \|F_{1,T,\e,\ell} u_0\|_{H^m(\R^d)}^2\\
+\e^{-2(\alpha_1+\beta \alpha_2)}\beta^{-2\alpha_2}  \|F_{1,T,\e,\ell} u_0\|_{H^m(\R^d)}^{2(1-\alpha_1-\beta \alpha_2)} 
 \|F_{1,T,\e,\ell} u_0\|_{H^m(\R^d,(1+|x|^2)dx)}^{2(\alpha_1+\beta \alpha_2)} 
\end{multline*}
(note the weighted Sobolev norm).
In view of the definition \eqref{e.Fourier-sym1} of the Fourier symbol, $\|F_{1,T,\e,\ell} u_0\|_{H^m(\R^d)} \,\le \, C(u_0)$ (a high-norm of $u_0$),
whereas 
$$ 
\|F_{1,T,\e,\ell} u_0\|_{H^m(\R^d,|x|^2dx)}^2 \,\lesssim \, \int_{\R^d} (1+|k|^{2m}) |\nabla_k (\hat F_{1,T,\e,\ell}\hat u_0)(k)|^2 dk \,\le\, T^2 C(u_0)^2 
$$
(the factor $T$ appears when the derivative falls on the cosinus). Altogether, this yields
$$
\int_{\mathbb{R}^d}\sup_{B(x)}|F_{1,T,\e,\ell} u_0|^2 \mu_\alpha(|\frac x\e|)^2dx \,\le\, C(u_0) (1+ (\e^{-1}T)^{2\alpha_1} \times \beta^{-2\alpha_2} (\e^{-1}T)^{2\beta \alpha_2}).
$$
It remains to choose $\beta>0$ to minimize the RHS. The minimum is obtained for $\beta = \log (\e^{-1}T)^{-1}$ if $\e^{-1}T\gg 1$ (otherwise $\beta=1$ will do), and we finally have
$$
\int_{\mathbb{R}^d}\sup_{B(x)}|F_{1,T,\e,\ell} u_0|^2 \mu_\alpha(|\frac x\e|)^2dx \,\le\, C(u_0) (1+ \mu_\alpha(\e^{-1}T)^2).
$$
Proceeding the same way to bound the second RHS term of \eqref{e.bound-before-symbol}, it follows that
\begin{equation}\label{e.ener-estim-3.1}
\sup_{0\le t\le T} \expec{ \Vert \partial_t v_{\e ,1,\ell}(t,\cdot)\Vert_{L^2(\mathbb{R}^d)}^2 + \Vert \nabla v_{\e ,1,\ell}(t,\cdot)\Vert_{L^2(\mathbb{R}^d)}^2}^\frac12
\,
\le\, C_\ell(u_0) \varepsilon^{\ell}(1+T)\mu_\alpha(\e^{-1}T),
\end{equation}
where $C_\ell(u_0)$ is a suitable (finite) norm of $u_0$ that only depends on $\ell$ and $d$ (but not on~$T$ and~$\e$).

\medskip

\substep{3.2} Estimate of $\expec{\Vert  v_{\e ,1,\ell}(t,\cdot)\Vert_{L^2(\mathbb{R}^d)}^2}$.

\noindent 
We integrate  \eqref{sys_annexe} once in time, test with $v_{\e ,1,\ell}$, and integrate over $\left[ 0,t\right]\times\mathbb{R}^d$, so that
to obtain the energy estimate
\begin{equation}\label{e.ener-estim-3.2}
\frac{1}{2}\Vert v_{\e ,1,\ell}(t,\cdot)\Vert^2_{L^2(\mathbb{R}^d)} \,\leq\,\varepsilon^{\ell} \int_{\left[ 0,t\right]\times\mathbb{R}^d}\left(\int_0^sf_{\e ,1,\ell}(r,x) dr\right)v_{\e ,1,\ell}(s,x) dsdx.
\end{equation}
The time integration is explicit,
$$
\int_0^sf_{\e ,1,\ell}(r,x) dr\,=\,-\frac1{(2\pi)^d} \int_{\mathbb{R}^d}(i\kappa)^{\ell+1}\hat{u}_{0,\e}(k)\nabla\cdot\left( e^{ik\cdot x}\Phi_{1,\ell}^e\big(\frac{x}{\varepsilon}\big)\right)\frac{\varepsilon}{\Lambda_\ell(\varepsilon k )}\sin(\varepsilon^{-1}\Lambda_\ell(\varepsilon k )s) dk,
$$
so that by integration by parts in space, we may rewrite the RHS of the energy estimate in the form
\begin{multline*}
\int_{\left[ 0,t\right]\times\mathbb{R}^d}\left(\int_0^sf_{\e ,1,\ell}(r,x) dr\right)v_{\e ,1,\ell}(s,x) dsdx
\\
=\,\frac1{2\pi} \int_{\left[ 0,t\right]\times\mathbb{R}^d}\nabla v_{\e ,1,\ell}(s,x) \cdot  \int_{\mathbb{R}^d}(i\kappa)^{\ell+1}\hat{u}_{0,\e}(k)e^{ik\cdot x}\Phi_{1,\ell}^e\big(\frac{x}{\varepsilon}\big)\frac{\varepsilon}{\Lambda_\ell(\varepsilon k )}\sin(\varepsilon^{-1}\Lambda_\ell(\varepsilon k )s)  dkdsdx.
\end{multline*}
We then define the linear operator $F_{3,s,\e,\ell}$ from $\Sc$ to $\Sc$ characterized by its Fourier symbol
\begin{eqnarray*}
\hat F_{3,s,\e,\ell}(k)&:=&(i\kappa)^{\ell+1}\omega_{\ell}(\e|k|)\frac{\varepsilon}{\Lambda_\ell(\varepsilon k )}\sin(\e^{-1}\Lambda_\ell(\varepsilon k )s) ,
\end{eqnarray*}
and conclude as in Substep~3.1 that
\begin{multline*}
\expec{\int_{\mathbb{R}^d} \left|\int_{\mathbb{R}^d}(i\kappa)^{\ell+1}\hat{u}_{0,\e}(k)e^{ik\cdot x}\Phi_{1,\ell}^e\big(\frac{x}{\varepsilon}\big)\frac{\varepsilon}{\Lambda_\ell(\varepsilon k )}\sin(\e^{-1}\Lambda_\ell(\varepsilon k )s)  dk \right|^2 dx}^\frac12 \\
\lesssim \,  \Big(\int_{\mathbb{R}^d}  \sup_{B_\varepsilon(x)}|F_{3,s,\e,\ell} u_0|^2 \mu_\alpha(|\frac x \e|)^2dx \Big)^\frac12\,
\lesssim\, C_\ell(u_0) \mu_\alpha(\e^{-1}s) .
\end{multline*}
Combined with \eqref{e.ener-estim-3.2} and \eqref{e.ener-estim-3.1}, this turns into
\begin{eqnarray*}
\lefteqn{\expec{\Vert v_{\e ,1,\ell}(T,\cdot)\Vert^2_{L^2(\mathbb{R}^d)}}}
\\
& \lesssim &\varepsilon^{\ell} \sup_{0\le t\le T} \expec{\|\nabla v_{\e ,1,\ell}(t,x) \|_{L^2(\R^d)}^2}^\frac12
 \int_{0}^T  C_\ell(u_0) \mu_\alpha(\e^{-1}s)ds \\
 &\lesssim &   \Big(C_\ell(u_0)  \varepsilon^{\ell} (1+T)\mu_\alpha(\e^{-1} T)\Big)^2.
\end{eqnarray*}
Proceeding similarly for the estimate of $\expec{\Vert v_{\e ,2,\ell}(T,\cdot)\Vert^2_{L^2(\mathbb{R}^d)}}$, this concludes the proof.

\subsection{Proof of Lemma~\ref{lem:throw-away}: Simplification of the Taylor-Bloch expansion}

By definition,
\begin{eqnarray*}
(\tilde v_{\e,\ell}-u_{\e,\ell})(t,x)
&=&  \frac{1}{(2\pi)^{d}} \sum_{j=1}^{\ell} \e^j\phi_j(\frac x\e)\cdot \int_{\R^d} ( i k)^{\otimes j}\omega_{\ell}(\e |k|) \hat u_0(k) e^{ik\cdot x} \cos(\e^{-1} \Lambda_\ell(\e k)t) dk.
\end{eqnarray*}
We then define the linear operators $\{F_{4,t,\e,\ell,j}\}_{j=1,\dots,\ell}$ from $\Sc$ to $\Sc$ characterized by their Fourier symbols
\begin{eqnarray*}
\hat F_{4,t,\e,\ell,j}(k)&:=&(ik)^{\otimes j}\omega_{\ell}(\e|k|) \cos(\e^{-1} \Lambda_\ell(\e k)t) ,
\end{eqnarray*}
and conclude as in Substep~3.1 of the proof of Proposition~\ref{prop:energy-wave} that
for all $1\le j\le \ell-1$,
\begin{multline*}
\expec{\int_{\mathbb{R}^d} \left|\phi_j(\frac x\e)\cdot \int_{\R^d} ( i k)^{\otimes j}\omega_{\ell}(\e |k|) \hat u_0(k) e^{ik\cdot x} \cos(\e^{-1} \Lambda_\ell(\e k)t) dk\right|^2 dx}^\frac12 \\
\lesssim \,  \Big(\int_{\mathbb{R}^d}  \sup_{B_\varepsilon(x)}|F_{4,t,\e,\ell,j} u_0|^2dx \Big)^\frac12\,
\lesssim\, C_\ell(u_0) ,
\end{multline*}
whereas for $j=\ell$,
\begin{multline*}
\expec{\int_{\mathbb{R}^d} \left|\phi_{\ell}(\frac x\e)\cdot \int_{\R^d} ( i k)^{\otimes \ell}\omega_{\ell}(\e |k|) \hat u_0(k) e^{ik\cdot x} \cos(\e^{-1} \Lambda_\ell(\e k)t) dk\right|^2 dx}^\frac12 \\
\lesssim \,  \Big(\int_{\mathbb{R}^d}  \sup_{B_\varepsilon(x)}|F_{4,t,\e,\ell,\ell} u_0|^2\mu_\alpha({|\frac x\e |})dx \Big)^\frac12\,
\lesssim\, C_\ell(u_0) \mu_\alpha(\e^{-1}T).
\end{multline*}
This proves the claim.

%%%%%%%%%%%%%%%%%
%%%%%%%%%%%%%%%%%
%%%%%%%%%%%%%%%%%
%%%%%%%%%%%%%%%%%
%%%%%%%%%%%%%%%%%
%%%%%%%%%%%%%%%%%

\section{Long-time homogenization and higher-order homogenized operators}\label{sec:disper}

In this section, we draw the consequences of Theorem~\ref{t.wave-eq} for the approximation of equation~\eqref{e.eq-ueps}
by higher-order homogenized equations, extending previous results of \cite{DLS1,DLS2} for periodic coefficients to higher-order time-scales and to random coefficients. We also give the counterpart of these results for the associated elliptic equation, which extends
the recent independent analysis of \cite{BFFO} to any order, and makes quantitative the 
formal analysis of \cite[Section~5]{ABV}.

\subsection{Higher-order homogenized wave equations}

For all $\ell\ge 1$, we define the elliptic operator
\begin{equation}\label{e.eq-ell} 
\tilde \calL_{\ho,\e, \ell}\,:=\,-\sum_{j=0}^{\ell-1} \e^{j} \bar \aa_j \cdot \nabla^{j+2},
\end{equation}
where $\bar \aa_j \cdot \nabla^{j+2} v:=\prod_{h=1}^{j+2} \sum_{i_h=1}^d [\bar \aa_j]_{i_1,\dots,i_{j+2}} \nabla_{i_1}\dots \nabla_{i_{j+2}} v$
(recall that $\bar \aa_{2j+1}=0$ for all $j\in \N$), and observe that the function $u_{\e,\ell}$ defined in \eqref{e.def-uepsell} satisfies
\begin{equation}\label{e.eq-uepsell} 
\left\{
\begin{array}{rcl}
\partial^2_{tt} u_{\e,\ell}+\tilde \calL_{\ho,\e, \ell} u_{\e,\ell} &=&0,\\
u_{\e,\ell}(0,\cdot)&=&{u_{0,\e,\ell}},\\
\partial_t u_{\e,\ell}(0,\cdot)&=&0.
\end{array}
\right.
\end{equation}
Indeed we have
\begin{eqnarray}\nonumber
\partial^2_{tt}u_{\e,\ell}&=&\frac{1}{(2\pi)^d} \int_{\R^d} \omega_{\ell}(\e |k|) \hat u_0(k) e^{ik\cdot x} \e^{-2} \Lambda_\ell(\e k)^2\cos(\e^{-1} \Lambda_\ell(\e k)t)dk ,\\ \nonumber
&=& \frac{1}{(2\pi)^d} \int_{\R^d} \omega_{\ell}(\e |k|) \hat u_0(k) e^{ik\cdot x} \Big(\sum_{j=0}^{\ell-1}(i\e)^j\vert k \vert^{j+2}\lambda_j\Big) \cos(\e^{-1} \Lambda_\ell(\e k)t)dk,
\end{eqnarray}
where we recall that $\Lambda_\ell(k):=\tilde{\lambda}_{k,\ell}$. Equation \eqref{e.eq-uepsell} then follows from the definition of the $\lambda_j$ and of the $\bar \aa_j$ (see Definition \ref{defi:ext-corr}). 

For $\ell= 3$, equation~\eqref{e.eq-uepsell} is not well-posed since 
the higher-order operator $-\bar \aa_{2} \cdot \nabla^{4}$ in $\tilde \calL_{\ho,\e,3}$ is non-positive, as first noticed in \cite{COV}, cf.~Proposition~\ref{prop:properties-lambda}.
In order to circumvent this difficulty, we regularize the operator $\tilde \calL_{\ho,\e,\ell}$ by a higher-order term, and define for all $\ell\ge 1$
\begin{equation}\label{first_ell_operator}
\calL_{\ho,\e,\ell}\,:=\,\tilde \calL_{\ho,\e,\ell} - \gamma_\ell(i\e)^{2([\frac{\ell-1}{2}]+1)} \Id \cdot \nabla^{2([\frac{\ell-1}{2}]+2)}\end{equation}
for some $\gamma_\ell\ge0$ to be chosen below, and consider the higher-order homogenized wave equation
\begin{equation}\label{e.eq-higher2} 
\left\{
\begin{array}{rcl}
\partial_{tt}^2w_{\e,\ell}+\calL_{\ho,\e,\ell}w_{\e,\ell}&=&0,\\
w_{\e,\ell}(0,\cdot)&=&{u_{0}},\\
\partial_t w_{\e,\ell}(0,\cdot)&=&0,
\end{array}
\right.
\end{equation}
(Note that we didn't modify the initial condition in \eqref{e.eq-higher2}, as opposed to \eqref{e.eq-uepsell}.)
The main result of this section is the following long-time homogenization of the wave equation~\eqref{i1}.
\begin{theo}\label{t.hom-long}
Let $\ell\ge 1$, and assume that $(\phi_j,\sigma_j,\chi_j)_{0\le j\le \ell}$ satisfy Hypothesis~\ref{hypo:corr}
for some $\alpha\in [0,1)\times \R_+$. Assume that $\gamma_\ell\ge 0$ is large enough so that $\calL_{\ho,\e,\ell}$ is a positive elliptic operator (see Lemma~\ref{lem:elliptic-well-posed} below).
Let $u_0 \in \Sc$, and for all $\e>0$, let $u_{\e}$ and $w_{\e,\ell}$ denote the solutions of \eqref{i1} and \eqref{e.eq-higher2}, respectively.
Then we have for all $T\ge 0$
\begin{equation}\nonumber
\sup_{0\leq t\le T} \expec{\|u_\e-w_{\e,\ell}\|_{L^2(\R^d)}^2}^\frac12
\, \lesssim \, C_\ell(u_0)\big(\varepsilon  \\+\varepsilon^{\ell}T\mu_\alpha(\varepsilon^{-1}T)\big),
\end{equation}
where $C_\ell(u_0)$ is a generic norm of $u_0$ which only depends on $\ell$ and $d$, and is finite for $u_0\in \Sc$.
\qed
\end{theo}
\begin{rem}
A similar result holds when considering a source term rather than an initial condition,
cf.~Remark~\ref{rem:forcing} and Appendix~\ref{sec:forcing}.
\qed
\end{rem}
As mentioned in the introduction, when $\bar \aa_{2} \cdot \nabla^{4}$ is the higher order operator in \eqref{e.eq-uepsell} it is also possible to  reformulate this term in a way that yields a well-posed higher-order homogenized wave equation and so that $u_{\e,\ell}$ remains a ``nearly-solution'' on sufficiently large times.
This approach, due to \cite{DLS1,DLS2}, uses the so-called ``Boussinesq trick'' and is based on the following algebraic decomposition property:
\begin{lemma}\label{lemma:decomp}\cite[Lemma~2.5]{DLS2}
Let $\bar \aa_{2}$ be as in Definition~\ref{defi:ext-corr}.  There exists a symmetric positive semi-definite second order tensor $\bb$ and a symmetric positive semi-definite fourth order tensor $\cc$ such that  
\begin{equation}\label{decomposition_eq}
\bar \aa_{2}\cdot \nabla^4= (\bb \otimes\bar \aa_0)\cdot \nabla^{{4}}-\cc\cdot\nabla^4.
\end{equation}
\qed
\end{lemma}
\begin{rem}The construction of $\textbf{b}$ and $\textbf{c}$ given in \cite{DLS2} extends mutadis mutandis to systems of equations
under the assumption that $\bar \aa_2$ is a symmetric non-negative tensor, and $\bar \aa_0$ is positive-definite (as provided by Proposition~\ref{prop:properties-lambda}).
We do not know whether such a construction holds at higher orders --- if it does, we expect Corollary~\ref{coro.hom-long-alaDLS} 
to extend accordingly.
\qed
\end{rem}
In particular, an alternative higher-order homogenized wave equation for $3\le \ell \le 4$ takes the form
\begin{equation}\label{e.eq-higher1} 
\left\{
\begin{array}{rcl}
\partial_{tt}^2w_{\e,\ell}-\bar \aa_0\cdot \nabla^2w_{\e,\ell}-\varepsilon^2\textbf{b}\cdot \nabla^2\partial_{tt}^2w_{\e,\ell}+\varepsilon^2\textbf{c}\cdot \nabla^4w_{\e,\ell}&=&0,\\
w_{\e,\ell}(0,\cdot)&=&u_0,\\
\partial_t w_{\e,\ell}(0,\cdot)&=&0.
\end{array}
\right.
\end{equation}
\begin{rem}
In \eqref{e.eq-higher1}, the cross-derivative term $\textbf{b}\cdot \nabla^2\partial_{tt}^2$ comes from the reformulation of $(\bb \otimes\bar \aa_0)\cdot \nabla^{{4}}$ in \eqref{decomposition_eq} at leading order using that  $\partial_{tt}^2 \simeq \nabla \cdot \aa_0\nabla$ for error terms by (first-order) homogenization. We refer to Subsection \ref{subsection_proof_ondes} for more details. \qed
\end{rem}

In this case we have
\begin{coro}\label{coro.hom-long-alaDLS}
Let $3\le \ell\le 4$, and assume that $(\phi_j,\sigma_j,\chi_j)_{0\le j\le \ell}$ satisfy Hypothesis~\ref{hypo:corr}
for some $\alpha\in [0,1)\times \R_+$. Let $u_0\in \Sc$, and for all $\e>0$, let $u_{\e}$ and $w_{\e,\ell}$ denote the solutions of \eqref{i1} and \eqref{e.eq-higher1}, respectively.
Then we have for all $T\ge 0$
\begin{equation}\nonumber
\sup_{0 \leq t\le T} \expec{\|u_\e-w_{\e,\ell}\|_{L^2(\R^d)}^2}^\frac12
\, \lesssim \, C_\ell(u_0)\big(\varepsilon  \\+\varepsilon^{\ell}T\mu_\alpha(\varepsilon^{-1}T) \big),
\end{equation}
where $C_\ell(u_0)$ is a generic norm of $u_0$ which only depends on and $d$, and is finite for $u_0\in \Sc$.
\qed
\end{coro}
Let us now turn to the arguments in favor of Theorem~\ref{t.hom-long} and Corollary~\ref{coro.hom-long-alaDLS}.
The following elementary lemma (which can be proved by interpolation in Fourier space, and energy estimates) ensures that $\calL_{\ho,\e,\ell}$ is a positive elliptic operator provided $\gamma_\ell$ is chosen large enough, and yields the well-posedness of \eqref{e.eq-higher2}.
\begin{lemma}\label{lem:elliptic-well-posed}
Let $\ell\ge 1$, assume that the homogenized tensors $\{\bar \aa_j\}_{0\le j \le \ell-1}$ are well-defined,
and  recall that $\bar \Gamma_\ell=\max_{0\le j\le \ell-1} |\bar \aa_j|<\infty$.
Then there exist $\gamma_\ell\ge 0$ and $c_\ell>0$ depending only on $\bar \Gamma_\ell$ and $\ell$ such that for 
all $ v\in  H^{[\frac{\ell-1}{2}]+2}(\R^d)$ and all $\e>0$ we have
$$
(\calL_{\ho,\e,\ell} v,v)_{(H^{-([\frac{\ell-1}{2}]+2)},H^{[\frac{\ell-1}{2}]+2})(\R^d)}\,\ge \, c_\ell(\|\nabla v\|_{L^2(\R^d)}^2+\e^{2[\frac{\ell-1}{2}]+2}\|\nabla^{[\frac{\ell-1}{2}]+2} v\|_{L^2(\R^d)}^2).
$$
For $1\le \ell\le 2$, we may choose $\gamma_\ell=0$.

As a consequence:
\begin{itemize}
\item For all $u_0\in \Sc$, equation \eqref{e.eq-higher2} admits a unique solution $w_{\e,\ell} \in L^\infty(\R_+,L^2(\R^d))$.
\item For $d\ge 3$ and for all $f\in L^\frac{2d}{d+2}(\R^d)$ and all $\e>0$,
the equation
\begin{equation}\label{e.ell-higher-hom}
\calL_{\ho,\e,\ell} v_\e\,=\,f
\end{equation}
admits a unique weak solution $v_\e\in L^\frac{2d}{d-2}(\R^d)$ such that $\nabla v_\e \in H^{2[\frac{\ell-1}{2}]+2}(\R^d)$. In addition, we have
\begin{equation}\label{e:elliptic-well-posed}
\|\nabla v_\e\|_{L^2(\R^d)} \,\lesssim \, \| f\|_{L^\frac{2d}{d+2}(\R^d)},
\end{equation}
where the multiplicative constant is independent of $\e>0$.
\end{itemize}
\qed
\end{lemma}
\begin{rem}
Note that the constant $\gamma_\ell$ can be computed explicitly in function of $\bar \Gamma_\ell$, which makes this regularization procedure of practical interest.
\qed
\end{rem}
Noting that for all $\alpha\in [0,1)\times \R_+$, $T\ge 0$ and $\e>0$, we have $\varepsilon^{\ell}T\mu_\alpha(\varepsilon^{-1}T) \gtrsim \e^{2([\frac{\ell-1}{2}]+1)}T$, so that adding a regularizing term of higher order does not influence the error estimate in Theorem~\ref{t.hom-long}. 
Indeed, regularizing the equation simply amounts to filtering high frequencies as we did explicitly when replacing
$u_0$ by $u_{0,\ell,\e}$. Theorem~\ref{t.hom-long} and Corollary~\ref{coro.hom-long-alaDLS} follow by the triangle inequality from the combination of Theorem~\ref{t.wave-eq}, the well-posedness result of Lemma~\ref{lem:elliptic-well-posed},
and the following estimate.
\begin{lemma}\label{lem:long-time}
Let $\ell \ge 1$,  and assume that $(\phi_j,\sigma_j,\chi_j)_{0\le j\le \ell}$ satisfy Hypothesis~\ref{hypo:corr}
for some $\alpha\in [0,1)\times \R_+$.
Let $u_0\in \Sc$, and for all $\e>0$, let $u_{\e , \ell}$ be as in  \eqref{e.def-uepsell} and $w_{\e,\ell}$ be the solution of \eqref{e.eq-higher2}
or \eqref{e.eq-higher1} (in which case $3\le \ell \le 4$).
Then we have for all $T\ge0$
\begin{equation}\nonumber
\sup_{t\le  T} \|u_{\e , \ell}-w_{\e,\ell}\|_{L^2(\R^d)}^2
\, \lesssim \, C_{2([\frac{\ell-1}{2}]+2)}(u_0) (\e+ \e^{2([\frac{\ell-1}{2}]+1)}T),
\end{equation}
where for all $p\ge 1$,
$$
C_{p}(u_0) := \Big(\int_{\R^d} |\nabla^p u_0(x)|^2dx \Big)^\frac12.
$$
\qed
\end{lemma}
%
%%%

\subsection{Higher-order homogenized elliptic equations}

The approach developed above for the wave equation has a counterpart for elliptic equations 
and yields the validity of higher-order homogenized equations (which however do \emph{not} coincide with the standard two-scale expansion --- except for the  terms involving the first two correctors only).
Our analysis applies to the equation on the whole space $\R^d$. For convenience,  we restrict to $d\ge 3$ in the rest of this section (in which case the Hardy inequality 
\`a la Caffarelli-Kohn-Nirenberg
$$
\int_{\R^d} u(x)^2 |x|^{-2} \lesssim \int_{\R^d} |\nabla u(x)|^2
$$
allows one to consider Lax-Milgram solutions of the equation $-\nabla \cdot \aa_\e \nabla u_\e=f$ for $f$ in the Schwartz class).

\medskip

In view of the discussion of the wave operator, the natural candidate $\tilde \calL_{\ho,\ell,\e}$ defined in \eqref{e.eq-ell} for the higher-order homogenized elliptic equation is not  elliptic  for $\ell=3$, and we rather
use the regularized homogenized operator $\calL_{\ho,\e, \ell}$ defined in \eqref{first_ell_operator} for all $\ell\ge 1$ and $\e>0$.
Lemma~\ref{lem:elliptic-well-posed} then ensures that $\calL_{\ho,\e,\ell}$ is elliptic for a suitable choice of $\gamma_\ell\ge 0$.

\medskip

As noticed by Allaire, Briane, and Vanninathan in \cite{ABV}, the formal difference between the higher-order elliptic operator \eqref{e.eq-ell} 
(or its regularized version \eqref{first_ell_operator}) and the one obtained by the standard two-scale expansion is related to the well-preparedness of the RHS of the original equation.
In particular, whereas the usual two-scale expansion is based on the equation
\begin{equation}\label{e.origin-ell}
-\nabla \cdot \aa_\e \nabla u_\e \,=\,f,
\end{equation}
the Bloch-wave expansion (from which the higher-order elliptic operator \eqref{first_ell_operator} is derived) is based on the equation with well-prepared RHS
\begin{equation}\label{e.mod-ell-well}
-\nabla \cdot \aa_\e \nabla u_{\e,\ell} \,=\,\sum_{j=0}^\ell \e^j \phi_j(\frac{\cdot}{\e}) \cdot \nabla^j f,
\end{equation}
where $\varphi_j$ stands for the (symmetric) $j$-th order tensor such that $\varphi_j \cdot e^{\otimes j}=\varphi_j^e$, given in Definition~\ref{defi:ext-corr} for all $e\in \R^d$.
The following theorem quantifies the accuracy of the higher-order homogenized operator for well-prepared data.
\begin{theo}\label{th:2scale-ell}
Let $\ell\ge 1$, and assume that $(\phi_j,\sigma_j,\chi_j)_{0\le j\le \ell}$ satisfy Hypothesis~\ref{hypo:corr}
for some $\alpha\in [0,1)\times \R_+$. Let $f\in \Sc$, and for all $\e>0$, let $u_{\e,\ell}$ and $u_{\ho,\e, \ell}$ denote the solutions of \eqref{e.mod-ell-well} and \eqref{e.ell-higher-hom}, respectively.
Then we have
\begin{equation}\label{e.multiscale-estim0}
\expec{ \Big\|\nabla \big( u_{\e,\ell}-\sum_{j=0}^{\ell} \e^j \nabla^j u_{\ho,\e, \ell} \cdot \phi_j (\frac\cdot\e)\big)\Big\|_{L^2(\R^d)}^2}^\frac12
\lesssim \, \e^{\ell} \mu_\alpha(\varepsilon^{-1}) C_{ \e , \ell}(u_{\ho,\e,\ell}),
\end{equation}
where
\begin{multline}\nonumber
C_{\e, \ell}(u_{\ho,\e,\ell})\,:=\,\left(\int_{\mathbb{R}^d} \mu_\alpha^2(|x|) \sup_{B_\e(x)}\lbrace\vert \nabla^{\ell+1}u_{\ho,\e,\ell}\vert^2+(1+|x|)^2\vert \nabla^{\ell+2}u_{\ho,\e,\ell}\vert^2\rbrace dx\right)^{\frac{1}{2}}
 \\ 
 + \sum_{j=0}^{2[\frac{\ell-1}{2}]}\e^{j+1}
\left(\int_{\mathbb{R}^d}  (1+|x|)^2 \sup_{B_\e(x)}\lbrace\vert \nabla^{j+\ell+3}u_{\ho,\e,\ell}\vert^2\rbrace dx\right)^{\frac{1}{2}}.
\end{multline}
\qed
\end{theo}
To our knowledge, this result is new, even in the periodic setting.

\medskip

Let us give two corollaries of this result. The first corollary shows that for $\ell\le 2$, it is not necessary to prepare the RHS of \eqref{e.origin-ell}. 
On the one hand, this is not surprising since the first two correctors of the Bloch wave expansion coincide with the first two correctors of the usual two-scale expansion (for which the RHS of \eqref{e.origin-ell} needs not be prepared). On the other hand, the correction \eqref{e.mod-ell-well} to the RHS of \eqref{e.origin-ell} is of order $\e$ whereas the RHS of \eqref{e.multiscale-estim0} is of order $\e^2$ (provided the second corrector is essentially bounded), so that there must be subtle cancellations.
This result was first proved (in the non-symmetric setting) in \cite{BFFO}.
\begin{coro}\label{coro:alaBFFO}
Let $1\le \ell\le2$, and assume that $(\phi_j,\sigma_j,\chi_j)_{0\le j\le \ell}$ satisfy Hypothesis~\ref{hypo:corr}
for some $\alpha\in [0,1)\times \R_+$. Let $f\in \Sc$, and for all $\e>0$, let $u_{\e}$ and $u_{\ho}$ denote the solutions of \eqref{e.origin-ell}  (the RHS of which is not well-prepared) and
\begin{equation}\label{e:stand-hom}
-\nabla \cdot \bar \aa_0\nabla u_\ho\,=\,f,
\end{equation}
respectively.
Then we have
\begin{equation}\label{e.multiscale-BFFO}
\expec{ \Big\|\nabla \big( u_{\e}-\sum_{j=0}^{\ell} \e^j \nabla^j u_{\ho} \cdot \phi_j (\frac\cdot\e)\big)\Big\|_{L^2(\R^d)}^2}^\frac12
\lesssim \, \e^{\ell} \mu_\alpha(\varepsilon^{-1}) C_{\ell, \e}(u_{\ho}),
\end{equation}
where
\begin{equation*}
C_{\e, \ell}(u_\ho)\,:=\,\left(\int_{\mathbb{R}^d}\mu_\alpha^2(|x|)\sup_{B_\e(x)}\lbrace\vert \nabla^{\ell+1}u_\ho\vert^2\rbrace dx\right)^{\frac{1}{2}}.
\end{equation*}
\qed
\end{coro}
The second corollary makes use of the Boussinesq trick to avoid the higher-order regularization of \eqref{e.eq-ell} for $3\le \ell \le 4$.
\begin{coro}\label{coro:alaDLS}
Let $3\le \ell\le 4$, and assume that $(\phi_j,\sigma_j,\chi_j)_{0\le j\le \ell}$ satisfy Hypothesis~\ref{hypo:corr}
for some $\alpha\in [0,1)\times \R_+$. 
Let $\bb$ and $\cc$ be the (symmetric positive semi-definite) second and fourth order tensors defined in Lemma \ref{lemma:decomp}. 
Let $f\in \Sc$, and for all $\e>0$, let $u_{\e,\ell}$ denote the solution of \eqref{e.mod-ell-well} (the RHS of which is well-prepared), and
$u_{\ho,\e,\ell}$ denote the unique solution of  
\begin{equation}\label{e.hom-eq-alaDLS}
(- \nabla \cdot \bar \aa_0  \nabla +\e^{2}\textbf{c}\cdot \nabla^{4}) u_{\ho,\e,\ell}\,=\,f-\e^2\textbf{b}\cdot \nabla^2 f.
\end{equation}
Then we have
\begin{equation}\label{e.multiscale-DLS}
\expec{ \Big\|\nabla \big( u_{\e,\ell}-\sum_{j=0}^{\ell} \e^j \nabla^j u_{\ho,\e,\ell} \cdot \phi_j (\frac\cdot\e)\big)\Big\|_{L^2(\R^d)}^2}^\frac12
\lesssim \, \e^{\ell} \mu_\alpha(\varepsilon^{-1}) C_{\e, \ell}(u_{\ho,\e,\ell},f),
\end{equation}
where 
\begin{multline*}
C_{\e, \ell}(u_{\ho,\e,\ell},f)\,:=\,\left(\int_{\mathbb{R}^d} \mu_\alpha^2(|x|) \sup_{B_\e(x)}\lbrace\vert \nabla^{\ell+1}u_{\ho,\e,\ell}\vert^2+(1+|x|)^2\vert \nabla^{\ell+2}u_{\ho,\e,\ell}\vert^2\rbrace dx\right)^{\frac{1}{2}} 
\\ 
+\sum_{j=0}^2\e^{j} \Big(\int_{\R^d} (1+|x|)^2 \sup_{B_\e(x)}\{|\nabla^{\ell+j} f|^2\}dx\Big)^\frac12
\\+\e \Big(\int_{\R^d} (1+|x|)^2 \sup_{B_\e(x)}\{|\nabla^{\ell+3} u_{\ho,\e,\ell}|^2\}dx\Big)^\frac12.
\end{multline*}
\qed
\end{coro}
\begin{rem}
The RHS of \eqref{e.multiscale-estim0} and \eqref{e.multiscale-DLS} in Theorem~\ref{th:2scale-ell} and Corollary~\ref{coro:alaDLS}
involve the same number of derivatives of $f$. For $3\le \ell \le 4$,~\eqref{e.multiscale-estim0} and~\eqref{e.multiscale-DLS} involve $\ell+3$ derivatives of $f$ (for \eqref{e.multiscale-estim0} one needs to differentiate $\ell+3$ times \eqref{e.ell-higher-hom}, whereas for \eqref{e.multiscale-DLS} one needs to differentiate $\ell+1$ times \eqref{e.hom-eq-alaDLS}).
\qed
\end{rem}
%

%%%%%%%%%%%%%%%%%

\subsection{Proof of Lemma~\ref{lem:long-time}: Higher-order approximation of the wave equation\label{subsection_proof_ondes}}

We split the proof of this lemma into two steps, and distinguish between \eqref{e.eq-higher1} and \eqref{e.eq-higher2}.
We start with \eqref{e.eq-higher1}.

\medskip

\step{1} Estimate for \eqref{e.eq-higher1}.

\noindent
Assume that $3\le \ell \le 4$ and let $w_{\e , \ell}$ solve \eqref{e.eq-higher1}.
By Lemma \ref{lemma:decomp}, the error $u_{\e, \ell}-w_{\e , \ell}$ splits into two parts
{$u_{\e, \ell}-w_{\e , \ell}=h_{\e,\ell}+\tilde h_{\e,\ell}$} that satisfy
\begin{equation}\label{e.eq-higher1aux}
\left\{
\begin{array}{rcl}
\partial_{tt}^2h_{\e,\ell}-\bar \aa_0\cdot \nabla^2h_{\e,\ell}-\varepsilon^2\bb\cdot \nabla^2\partial_{tt}^2h_{\e,\ell}+\varepsilon^2\cc\cdot \nabla^4h_{\e,\ell}&=&\varepsilon^4(\bb \otimes\bar\aa_2)\cdot\nabla^6u_{\e, \ell},\\
h_{\e,\ell}(0,\cdot)&=&0,\\
\partial_t h_{\e,\ell}(0,\cdot)&=&0,
\end{array}
\right.
\end{equation}
and
\begin{equation}\label{e.eq-higher1aux-IC}
\left\{
\begin{array}{rcl}
\partial_{tt}^2\tilde h_{\e,\ell}-\bar \aa_0\cdot \nabla^2\tilde h_{\e,\ell}-\varepsilon^2\bb\cdot \nabla^2\partial_{tt}^2\tilde h_{\e,\ell}+\varepsilon^2\cc\cdot \nabla^4\tilde h_{\e,\ell}&=&0,\\
\tilde h_{\e,\ell}(0,\cdot)&=&{u_{0,\e,\ell}-u_0},\\
\partial_t \tilde h_{\e,\ell}(0,\cdot)&=&0.
\end{array}
\right.
\end{equation}
The estimate for $\tilde h_{\e,\ell}$ is similar to that of Lemma~\ref{lem:prepare} and we have
$$
\sup_{0\le t<\infty} \Vert \tilde h_{\e, \ell}(t,\cdot)\Vert^2_{L^2(\mathbb{R}^d)} \,\lesssim\,  \Vert u_{0,\e,\ell}-u_0 \Vert^2_{L^2(\mathbb{R}^d)} \lesssim \e^2 C(u_0).
$$
We then turn to the estimate of $ h_{\e, \ell}$. As in the proof of Lemma \ref{lem:prepare}, we integrate \eqref{e.eq-higher1aux} once in time, multiply by $h_{\e,\ell}$ and integrate over $\left[0,t \right]\times \mathbb{R}^d$. Since  $\bar \aa_0$, $\bb$, and $\cc$ are non-negative, their contributions to the energy estimate on the LSH are non-negative, and
we obtain
\begin{eqnarray}\nonumber
\Vert h_{\e, \ell}(t,\cdot)\Vert^2_{L^2(\mathbb{R}^d)} \lesssim \varepsilon^4\left| \int_0^t\int_{\mathbb{R}^d}h_{\e ,\ell}(s,x)R_{\e ,\ell}(s,x)dxds\right|,
\end{eqnarray}
where, by definiton of $u_{\e, \ell}$ and integration in time,
\begin{eqnarray}\nonumber
R_{ \e, \ell}&:=&\nabla^6\int_{\mathbb{R}^d}\omega(\varepsilon\vert k \vert)\hat{u}_0(k)e^{ik\cdot x}
\frac{\varepsilon}{\Lambda_\ell(\varepsilon\vert k \vert)}\sin(\varepsilon^{-1}\Lambda_\ell(\varepsilon\vert k \vert)s) dk.
\end{eqnarray}
We then proceed  as in Substep~3.1 of the proof of Proposition~\ref{prop:energy-wave}, and define $F_{s,\e,\ell}$ as the linear operator from $\Sc$ to $\Sc$ characterized by its Fourier symbol
\begin{eqnarray*}
\hat F_{s,\e,\ell}(k)&:=&(ik)^{\otimes 6}\omega(\e|k|)\frac{\varepsilon}{\Lambda_\ell(\varepsilon\vert k \vert)}\sin\left( \varepsilon^{-1}\Lambda_\ell(\varepsilon \vert k\vert ) s\right),
\end{eqnarray*}
which is bounded by $|k|^6$ uniformly over $\e>0$ and $s\ge 0$ by definition of $\omega$ and $\Lambda_\ell$.
Hence, by Cauchy-Schwarz inequality,
\begin{eqnarray}\nonumber
\Vert h_{\e , \ell}(t,\cdot)\Vert^2_{L^2(\mathbb{R}^d)}&\lesssim &\varepsilon^4\int_0^t \Vert h_{\e , \ell}(s,\cdot)\Vert_{L^2(\mathbb{R}^d)}\left(\int_{\mathbb{R}^d}\vert F_{s,\e,\ell}u_0\vert^2 dx\right)^\frac{1}{2} ds, \\ \nonumber
&\lesssim & \varepsilon^4 t \sup_{0\le s\leq t}\Vert h_{\e , \ell}(s,\cdot)\Vert_{L^2(\mathbb{R}^d)}\left(\int_{\mathbb{R}^d}|\nabla^6 u_0|^2 dx\right)^\frac{1}{2}.
\end{eqnarray}
Since the RHS is non-decreasing in $t$, we may replace the LHS by $\sup_{0 \le s\le T} \Vert h_{\e , \ell}(t,\cdot)\Vert^2_{L^2(\mathbb{R}^d)}$,
and obtain
\begin{eqnarray}\nonumber
\sup_{0\le t\leq T} \Vert h_{\e , \ell}(t,\cdot)\Vert_{L^2(\mathbb{R}^d)}\lesssim C_\ell(u_0)\varepsilon^{4}T,
\end{eqnarray}
which concludes the proof.

\medskip

\step{2} Estimate for \eqref{e.eq-higher2}.

\noindent
Assume that $\ell \ge 1$ and let $w_{\e , \ell}$ solve \eqref{e.eq-higher2}.
By definition, the error writes $u_{\e, \ell}-w_{\e , \ell}=h_{\e, \ell}+\tilde h_{\e,\ell}$, where $h_{\e, \ell}$ and $\tilde h_{\e, \ell}$ satisfy
\begin{equation}\label{e.eq-higher2aux}
\left\{
\begin{array}{rcl}
\partial^2_{tt} h_{\e,\ell}+\calL_{\ho,\e,\ell} h_{\e,\ell}&=&\e^{2([\frac{\ell-1}{2}]+1)}{R}_{\e , \ell},\\
h_{\e,\ell}(0,\cdot)&=&0,\\
\partial_t h_{\e,\ell}(0,\cdot)&=&0,
\end{array}
\right.
\end{equation}
where 
$${R}_{\e , \ell}:=-\gamma_{\ell}i^{2([\frac{\ell-1}{2}]+1)} \Id \cdot \nabla^{2([\frac{\ell-1}{2}]+2)}u_{\e,\ell},$$
and
\begin{equation}\label{e.eq-higher2aux}
\left\{
\begin{array}{rcl}
\partial^2_{tt} \tilde h_{\e,\ell}+\calL_{\ho,\e,\ell} \tilde h_{\e,\ell}&=&0,\\
\tilde h_{\e,\ell}(0,\cdot)&=&u_{0,\e,\ell}-u_0,\\
\partial_t \tilde h_{\e,\ell}(0,\cdot)&=&0.
\end{array}
\right.
\end{equation}
As in Lemma~\ref{lem:prepare} and Step~1 we have
$$
\sup_{0\le t<\infty} \Vert \tilde h_{\e, \ell}(t,\cdot)\Vert^2_{L^2(\mathbb{R}^d)} \,\lesssim\,  \Vert u_{0,\e,\ell}-u_0 \Vert^2_{L^2(\mathbb{R}^d)} \lesssim \e^2 C(u_0).
$$
We now establish the energy estimate for \eqref{e.eq-higher2aux}. As in Step~1, we integrate \eqref{e.eq-higher2aux} once in time, multiply by $h_{\e,\ell}$, and integrate over $\left[0,t \right]\times \mathbb{R}^d$.
We then obtain, using that $h_{\e,\ell}(0,\cdot)\equiv 0$,
\begin{multline*}
\Vert h_{\e , \ell}(t,\cdot) \Vert_{L^2(\mathbb{R}^d)}^2+\frac12 \bigg(\calL_{\ho,\e,\ell} \int_0^t h_{\e , \ell}(s,\cdot)ds, \int_0^t h_{\e , \ell}(s,\cdot)ds\bigg)_{(H^{-([\frac{\ell-1}{2}]+1)},H^{[\frac{\ell-1}{2}]+1})(\R^d)}
\\
=\,\e^{2([\frac{\ell-1}{2}]+1)} \int_0^t \int_{\mathbb{R}^d}h_{\e , \ell}(s,x)R_{\e , \ell}(s,x)dx ds. 
\end{multline*}
By Lemma~\ref{lem:elliptic-well-posed}, the second LHS term is non-negative, so that we obtain as in Step~1
$$
\sup_{0\le t\le T} \Vert h_{\e , \ell}(t,\cdot) \Vert_{L^2(\mathbb{R}^d)}^2 \,\lesssim \, \e^{4([\frac{\ell-1}{2}]+1)}T^2 \sup_{0\le t\le T} 
\int_{\R^d} R_{\e , \ell}^2(t,x)dx .
$$
It remains to control the last RHS factor. Proceeding in Step~1, we have in Fourier space,
$$
\int_{\R^d} R_{\e , \ell}^2(t,x)dx \,\lesssim \, \int_{\R^d} |k|^{4([\frac{\ell-1}{2}]+2)} |\hat u_0(k)|^2dk \,\lesssim \, \int_{\R^d} |\nabla^{2([\frac{\ell-1}{2}]+2)} u_0(x)|^2dx,
$$
and the proof is complete.

%%%%%%%%%%%%%%%%%

\subsection{Proof of Theorem~\ref{th:2scale-ell}: Higher-order approximation of the elliptic equation}
By scaling, it is enough to consider $\e=1$. 
We split the proof into four steps.
We start by deriving a representation formula for the residuum that relies on the algebra of the correctors, from which
all the estimates follow. 

\medskip

\step{1} Representation formula for the residuum.

\noindent  For all smooth functions $v$ and $m\ge 1$, set  $w_{m}(v):=\sum_{j=0}^{m}\varphi_j\cdot\nabla^j v$, and
for all $m\ge 2$, set
\begin{equation}\label{e.fo-S}
S_{m}(v):=\sum_{p=0}^{m-2}\sum_{j=1}^{m-1-p}(\varphi_j \otimes\bar \aa_p)\cdot \nabla^{p+j+2}v,
\end{equation}
with the understanding that $S_m\equiv 0$ for $m<2$.
We shall prove that that for $1\le \ell\le 2$,
\begin{equation}\label{e.2scale-0.1}
-\nabla \cdot \aa \nabla w_{\ell}(v) \,=\, -\nabla \cdot \bar \aa_0  \nabla v-\nabla \cdot \left[ (\aa \otimes\varphi_{\ell}-\sigma_{\ell})\cdot\nabla^{\ell+1}v\right].
\end{equation}
whereas for all $\ell\ge 1$, 
\begin{multline}\label{e.2scale-0.2}
-\nabla \cdot \aa \nabla w_{\ell}(v) \,=\, -\sum_{j=0}^{\ell-1} \bar \aa_j\cdot \nabla^{j+2}{v} -S_{\ell}(v) \\ 
+\nabla\chi_{\ell}\cdot \nabla^{\ell+2}v-\nabla \cdot \left[ (\aa \otimes\varphi_{\ell}-\sigma_{\ell}+\nabla\chi_{\ell})\cdot\nabla^{\ell+1}v\right],
\end{multline}
Theorem~\ref{th:2scale-ell} will then follow from \eqref{e.2scale-0.2}, whereas Corollary~\ref{coro:alaBFFO} will follow from \eqref{e.2scale-0.1}.
We split the proof of \eqref{e.2scale-0.1} \&~\eqref{e.2scale-0.2} into two substeps.

\medskip

\substep{1.1} First representation formula for the residuum, and proof that for all $m\ge 1$,
\begin{multline}\label{e.2scale-1.1}
\nabla \cdot \aa \nabla w_{m}(v) \,=\, \sum_{j=0}^{m-1} \bar \aa_j\cdot \nabla^{j+2} v+S_{m-1}(v)
\\-\nabla\chi_{m-1}\cdot \nabla^{m+1} v
+\nabla \cdot \left[ (\aa \otimes\varphi_{m}-\sigma_{m})\cdot\nabla^{m+1}v\right].
\end{multline}
We proceed by induction.
The result for $m=1$ is by now standard, and the very reason for introducing $\sigma_1$ in \cite{GNO3}.
More precisely, a direct calculation combined with the defining equation $-\nabla \cdot \aa(\Id+\nabla \phi_1)=0$ for $\phi_1$, and the definition $\nabla \cdot \sigma_1\,=\,\aa(\Id+\nabla \phi_1)-\bar \aa_0$ and skew-symmetry of $\sigma_1$ in form of
$\nabla^2 v\cdot \nabla \cdot \sigma_1=-\nabla \cdot (\sigma_1\cdot \nabla^2 v)$, yields
\begin{eqnarray*}
\nabla \cdot \aa \nabla w_{1}(v)&=& \nabla \cdot \aa  ((\Id+\nabla \phi_1)\nabla v +   \phi_1\cdot \nabla^2 v )
\\
&=& \nabla \cdot \bar \aa_0 \nabla v+\nabla \cdot \big(\aa (\Id+ \nabla \phi_1)-\bar \aa_0\big)\nabla v+\nabla \cdot \aa ( \phi_1\cdot \nabla^2 v )
\\
&=&\nabla \cdot \bar \aa_0 \nabla v+ \nabla^2 v \cdot \big(\aa (\Id+ \nabla \phi_1)-\bar \aa_0\big)+\nabla \cdot \aa ( \phi_1\cdot \nabla^2 v )
\\
&=&\nabla \cdot \bar \aa_0 \nabla v +\nabla \cdot \big[(\aa\varphi_1- \sigma_1)\cdot \nabla^2 v\big],
\end{eqnarray*}
as claimed. Assume now that \eqref{e.2scale-1.1} holds at step $m\ge 1$. 
Writing $w_{m+1}(v)=w_m(v) + \phi_{m+1}\cdot \nabla^{m+1} v$, and using  \eqref{e.2scale-1.1} at step~$m$, we obtain
\begin{eqnarray}\nonumber
\nabla \cdot \aa \nabla w_{m+1}(v) &=& \nabla \cdot \aa \nabla w_{m}(v)+\nabla \cdot [\aa(\nabla \varphi_{m+1}\cdot\nabla^{m+1}v)]+\nabla \cdot [(\aa\otimes \varphi_{m+1})\cdot\nabla^{m+2}v] \\ 
&=& \sum_{j=0}^{m-1} \bar \aa_j\cdot \nabla^{j+2} v+S_{m-1}(v)  -\nabla \chi_{m-1}\cdot\nabla^{m+1}v\nonumber\\ 
&&+ \nabla \cdot \left[ (\aa\nabla\varphi_{m+1}+\aa \otimes\varphi_{m}-\sigma_{m})\cdot\nabla^{m+1}v\right] 
+ \nabla \cdot [(\aa \otimes\varphi_{m+1})\cdot\nabla^{m+2}v] .\nonumber 
\\\label{e.2scale-1.3}
\end{eqnarray}
Let us reformulate the third RHS term.
We add and substract $\nabla \chi_m$ in order to use the defining equation  for $\phi_{m+1}$ in form of $\nabla \cdot  (\aa\nabla\varphi_{m+1}+\aa \otimes\varphi_{m} -\sigma_{m}+\nabla\chi_{m})=0$, and rewrite the third RHS as
\begin{multline*}
\nabla \cdot \left[ (\aa\nabla\varphi_{m+1}+\aa \otimes\varphi_{m}-\sigma_{m})\cdot\nabla^{m+1}v\right]=-\bigtriangleup \chi_{m}\cdot\nabla^{m+1}v \\ +  (\aa\nabla\varphi_{m+1}+\aa \otimes\varphi_{m}-\sigma_{m})\cdot\nabla^{m+2}v.
\end{multline*}
We then appeal to the defining formulas $-\bigtriangleup \chi_{m}:=\nabla\chi_{m-1} +\sum_{j=1}^{m-1} \varphi_j\otimes\bar \aa_{m-1-j}$
 and $q_{m+1}:=  \aa \nabla\phi_{m+1}+\aa\otimes \phi_{m}- \bar \aa_{m}+\nabla \chi_{m}-\sigma_{m}$, and obtain
\begin{multline*}
\nabla \cdot \left[ (\aa\nabla\varphi_{m+1}+\aa \otimes\varphi_{m}-\sigma_{m})\cdot\nabla^{m+1}v\right]= \sum_{j=1}^{m-1}(\varphi_j \otimes\bar \aa_{m-1-j})\cdot \nabla^{m+1}v
\\
+\nabla \chi_{m-1}\cdot\nabla^{m+1}v+(q_{m+1}+\bar \aa_{m}-\nabla \chi_{m})\cdot\nabla^{m+2}v.
\end{multline*}
Combining the above with \eqref{e.2scale-1.3} then yields
\begin{multline*}
\nabla \cdot \aa \nabla w_{m+1}(v) \,=\, \sum_{j=0}^{m} \bar \aa_j\cdot \nabla^{j+2} v+S_{m-1}(v) +\sum_{j=1}^{m-1}(\varphi_j \otimes\bar \aa_{m-1-j})\cdot \nabla^{m+1}v -\nabla \chi_{m}\cdot\nabla^{m+2}v\\ 
+q_{m+1}\cdot \nabla^{m+2}v+ \nabla \cdot [(\aa \otimes\varphi_{m +1})\cdot\nabla^{m+2}v] .
\end{multline*}
Using the property
\begin{equation}\label{relation_S}
S_{m}(v)=S_{m-1}(v)+\sum_{j=1}^{m}(\varphi_j\otimes\bar \aa_{m-j})\cdot\nabla^{m+{2}}v, 
\end{equation}
we conclude by the defining equation $\nabla \cdot \sigma_{m+1}=q_{m+1}$ and the skew-symmetry of $\sigma_{m+1}$
in form of $\nabla^{m+2}v \cdot \nabla \cdot \sigma_{m+1}=-\nabla \cdot (\sigma_{m+1}\cdot \nabla^{m+2}v)$ that 
\begin{multline}\nonumber
\nabla \cdot \aa \nabla w_{m+1}(v) = \sum_{j=0}^{m} \bar \aa_j\cdot \nabla^{j+2} v+S_{{m}}(v) -\nabla\chi_{m}\cdot\nabla^{m+2}v+ \nabla \cdot [(\aa\otimes \varphi_{m+1}-\sigma_{m+1})\cdot\nabla^{m+2}v] ,
\end{multline}
that is, \eqref{e.2scale-1.1} at step $m+1$.

\medskip

\substep{1.2} Reformulation of \eqref{e.2scale-1.1}, and proof of \eqref{e.2scale-0.1} \&~\eqref{e.2scale-0.2}.

\noindent We start with \eqref{e.2scale-0.1}, which is a direct consequence of \eqref{e.2scale-1.1} and the identities $\chi_m \equiv 0$ for all $m \le 1$ and $\bar \aa_1=0$ (cf.~Proposition~\ref{prop:properties-lambda}). We then turn to \eqref{e.2scale-0.2}. 
By the defining formula $-\nabla\chi_{m-1} = \bigtriangleup \chi_{m}+\sum_{j=1}^{m-1} \varphi_j\otimes\bar \aa_{m-1-j}$ for $\chi_{m}$, \eqref{e.2scale-1.1} turns into
\begin{multline}\nonumber
\nabla \cdot \aa \nabla w_{m}(v) \,=\, \sum_{j=0}^{m-1} \bar \aa_j\cdot \nabla^{j+2} v+S_{m-1}(v)  +\sum_{j=1}^{k-1} (\varphi_j\otimes\bar \aa_{m-1-j})\cdot \nabla^{m+1}v\\
+\bigtriangleup \chi_{m}\cdot\nabla^{m+1}v+\nabla \cdot \left[ (\aa \otimes \varphi_{m}-\sigma_{m})\cdot\nabla^{m+1}v\right].
\end{multline}
Combined with \eqref{relation_S} and the identity $\bigtriangleup \chi_{m}\cdot\nabla^{m+1}v=\nabla\cdot(\nabla\chi_{m}\cdot\nabla^{m+1}v)-\nabla\chi_{m}\cdot\nabla^{m+2}v$, this yields  \eqref{e.2scale-0.2}.

\medskip

\step{2} Proof of Corollary~\ref{coro:alaBFFO}: \eqref{e.multiscale-BFFO}.

\noindent Let $u$ and $u_{\ho}$ denote the solutions of \eqref{e.origin-ell} for $\e=1$ and \eqref{e:stand-hom}.
With the choice $v=u_\ho$, substracting \eqref{e.2scale-0.1} from \eqref{e.origin-ell} for $\e=1$ yields
$$
-\nabla \cdot \aa \nabla (u-w_\ell(u_\ho)) \,=\, \nabla \cdot \left[ (\aa \otimes\varphi_{\ell}-\sigma_{\ell}+\nabla\chi_{\ell})\cdot\nabla^{\ell+1}u_\ho \right],
$$
so that \eqref{e.multiscale-BFFO} follows from testing this equation with $u-w_\ell(u_\ho)$, Hypothesis~\ref{hypo:corr}, and a scaling argument.

\medskip

\step{3} Proof of Theorem~\ref{th:2scale-ell}.

\noindent Let $u_{\ell}$ and $u_{\ho,\ell}$ denote the solutions of \eqref{e.mod-ell-well} and \eqref{e.ell-higher-hom} for $\e=1$.
We split the proof into two substeps: we first derive the equation for the error $h_{\ell}:=u_\ell-u_{\ho,\ell}$, and then proceed to the actual estimates.

\medskip

\substep{3.1} Proof of the identity
\begin{multline}\label{e.2scale-3.0}
-\nabla \cdot \aa \nabla h_\ell \,=\,\nabla \cdot \left[ (\aa \otimes\varphi_{\ell}-\sigma_{\ell}+\nabla\chi_{\ell})\cdot\nabla^{\ell+1}u_{\ho,\ell}\right]
+\sum_{p=1}^{\ell-1}\sum_{j=\ell-p}^{\ell-1}(\varphi_j\otimes \bar \aa_p)\cdot\nabla^{j+p+2}u_{\ho , \ell }
\\
-\nabla\chi_{\ell}\cdot \nabla^{\ell+2}u_{\ho,\ell}-\gamma_{\ell}i^{2([\frac{\ell-1}{2}]+1)}\sum_{j=0}^{\ell-1}(\varphi_j\otimes\Id)\cdot\nabla^{j+2([\frac{\ell-1}{2}] +2)}u_{\ho , \ell }.
\end{multline}
The starting point is \eqref{e.2scale-0.2} with $v=u_{\ho,\ell}$:
\begin{multline}\label{e.2scale-3.2}
-\nabla \cdot \aa \nabla w_{\ell}(u_{\ho,\ell}) \,=\, -\sum_{j=0}^{\ell-1} \bar \aa_j\cdot \nabla^{j+2} u_{\ho,\ell}-S_{\ell}(u_{\ho,\ell}) \\ 
+\nabla\chi_{\ell}\cdot \nabla^{\ell+2}u_{\ho,\ell}-\nabla \cdot \left[ (\aa \otimes\varphi_{\ell}-\sigma_{\ell}+\nabla\chi_{\ell})\cdot\nabla^{\ell+1}u_{\ho,\ell}\right].
\end{multline}
Using \eqref{e.ell-higher-hom} for $\e=1$, we reformulate the first RHS term as
\begin{equation}\label{e.2scale-3.3}
-\sum_{j=0}^{\ell-1} \bar \aa_j\cdot \nabla^{j+2} u_{\ho,\ell}\,=\,f+\gamma_{\ell}i^{2([\frac{\ell-1}{2}]+1)}\Id\cdot \nabla^{2([\frac{\ell-1}{2}]+2)}u_{\ho , \ell}.
\end{equation}
In order to reformulate the second RHS term, we take the $j$-th derivative of \eqref{e.ell-higher-hom} for $\e=1$ in form of 
\begin{multline}\label{e.2scale-3.1}
-(\varphi_j \otimes\bar \aa_0 )\cdot \nabla^{j+2}u_{\ho , \ell}\,=\,\varphi_j \cdot\nabla^j f+\sum_{p=1}^{\ell-1}(\varphi_j\otimes \bar \aa_p)\cdot\nabla^{j+p+2}u_{\ho , \ell}\\
 +\gamma_{\ell}i^{2([\frac{\ell-1}{2}]+1)}(\varphi_j\otimes\Id)\cdot\nabla^{j+2([\frac{\ell-1}{2}]+2)}u_{\ho , \ell },
 \end{multline}
which allows us to rewrite $S_{\ell}(u_{\ho,\ell})$ as
\begin{eqnarray*}\nonumber
-S_{\ell}(u_{\ho,\ell})&\stackrel{\eqref{e.fo-S}}{=}&\sum_{j=1}^{\ell-1}(\varphi_j \otimes\bar \aa_0 )\cdot \nabla^{j+2}u_{\ho , \ell}+\sum_{p=1}^{\ell-2}\sum_{j=1}^{\ell-1-p}(\varphi_j \otimes\bar \aa_p)\cdot\nabla^{p+j+2}u_{\ho , \ell} \\
&\stackrel{\eqref{e.2scale-3.1}}{=}&\sum_{j=1}^{\ell-1}\varphi_j \cdot\nabla^j f  -\sum_{p=1}^{\ell-1}\sum_{j=\ell-p}^{\ell-1}(\varphi_j\otimes \bar \aa_p)\cdot\nabla^{j+p+2}u_{\ho , \ell }\\
&&+\gamma_{\ell}i^{2([\frac{\ell-1}{2}]+1)}\sum_{j=1}^{\ell-1}(\varphi_j\otimes\Id)\cdot\nabla^{j+2([\frac{\ell-1}{2}]+2)}u_{\ho , \ell } .
\end{eqnarray*}
Combined with \eqref{e.2scale-3.2} and \eqref{e.2scale-3.3}, this yields
\begin{multline*}
-\nabla \cdot \aa \nabla w_{\ell}(u_{\ho,\ell}) \,=\,\sum_{j=0}^{\ell-1}\varphi_j \cdot\nabla^j f
 -\nabla \cdot \left[ (\aa \otimes\varphi_{\ell}-\sigma_{\ell}+\nabla\chi_{\ell})\cdot\nabla^{\ell+1}u_{\ho,\ell}\right] 
 \\ -\sum_{p=1}^{\ell-1}\sum_{j=\ell-p}^{\ell-1}(\varphi_j\otimes \bar \aa_p)\cdot\nabla^{j+p+2}u_{\ho , \ell }+\nabla\chi_{\ell}\cdot \nabla^{\ell+2}u_{\ho,\ell}
 \\
+\gamma_{\ell}i^{2([\frac{\ell-1}{2}]+1)}\sum_{j=0}^{\ell-1}(\varphi_j\otimes\Id)\cdot\nabla^{j+2([\frac{\ell-1}{2}]+2)}u_{\ho , \ell } .
\end{multline*}
The claim \eqref{e.2scale-3.0} then follows from substracting this identity from \eqref{e.mod-ell-well} for $\e=1$.

\medskip

\substep{3.2} Proof of \eqref{e.multiscale-estim0}. 

\noindent The RHS of \eqref{e.2scale-3.0} is the sum of one term in divergence form and three terms in non-divergence form.
Testing  \eqref{e.2scale-3.0} with $h_\ell$, making an integration by parts for the terms in divergence form, and using Hardy's inequality on the last three RHS terms in form of 
$$
\int_{\R^d} |h_\ell(x) g(x)| dx \,\lesssim\, \Big(\int_{\R^d} |\nabla h_\ell(x)|^2dx\Big)^\frac12 \Big(\int_{\R^d} (1+|x|)^2g(x)^2dx\Big)^\frac12,
$$
the energy estimate for \eqref{e.2scale-3.0} reads
\begin{equation*}\label{e.2scale-3.4}
\|\nabla h_\ell\|_{L^2(\R^d)}^2 \,\lesssim\, I_1^2+I_2^2+I_3^2+I_4^2,
\end{equation*}
where
\begin{eqnarray*}
I_1^2&:=&\int_{\R^d}  \Big(\fint_{B(x)}|\varphi_{\ell}|^2+|\sigma_{\ell}|^2+|\nabla\chi_{\ell}|^2\Big) \sup_{B(x)} \{|\nabla^{\ell+1}u_{\ho,\ell}|^2 \}dx,
\\
I_2^2&:=&\sum_{p=1}^{\ell-1}\sum_{j=\ell-p}^{\ell-1} \int_{\R^d} \Big(\fint_{B(x)} |\varphi_j|^2\Big) (1+|x|)^2\sup_{B(x)}\{|\nabla^{j+p+2}u_{\ho , \ell }|^2\}dx,
\\
I_3^2&:=&\int_{\R^d}\Big(\fint_{B(x)} |\nabla \chi_\ell|^2\Big) (1+|x|)^2\sup_{B(x)}\{|\nabla^{\ell+2}u_{\ho,\ell}|^2\}dx,
\\
I_4^2&:=&\sum_{j=0}^{\ell-1} \int_{\R^d}\Big(\fint_{B(x)} |\varphi_j|^2\Big) (1+|x|)^2 \sup_{B(x)}\{| \nabla^{j+2([\frac{\ell-1}{2}]+2)}u_{\ho , \ell }|^2\}dx.
\end{eqnarray*}
Taking the expectation of the energy estimate, and using Hypothesis~\ref{hypo:corr}, we obtain the claim for $\e=1$. The general
result follows by a scaling argument.

\medskip

\step{4} Proof of Corollary~\ref{coro:alaDLS}: \eqref{e.multiscale-DLS}.

\noindent 
As above, by a scaling argument, it is enough to prove  \eqref{e.multiscale-DLS} for $\e=1$.
Recall that $u_{\ho,\ell}$ is the solution of \eqref{e.hom-eq-alaDLS} for $\e=1$ and $3\le \ell\le 4$.
Starting point is formula~\eqref{e.2scale-0.2}, which we specifically rewrite for $3\le \ell\le 4$ as
\begin{multline}\nonumber
-\nabla \cdot \aa \nabla w_{\ell}(u_{\hom , \ell}) \,= \,-\bar \aa_0\cdot \nabla^{2}u_{\hom , \ell}-\bar \aa_2\cdot \nabla^{4}u_{\hom , \ell}-\sum_{j=1}^{\ell-1} (\varphi_j\otimes\bar \aa_0)\cdot \nabla^{j+2}u_{\hom , \ell} 
\\  -\delta_{\ell=4}(\varphi_1\otimes\bar \aa_2)\cdot \nabla^5u_{\hom , \ell}+\nabla\chi_{\ell}\cdot \nabla^{\ell+2}u_{\hom , \ell}
\\
-\nabla \cdot \left[ (\aa \otimes\varphi_{\ell}-\sigma_{\ell}+\nabla\chi_{\ell})\cdot\nabla^{\ell+1}u_{\hom , \ell}\right],
\end{multline}
where $\delta$ denotes the Kronecker symbol. 
We then appeal to the decomposition of $\bar \aa_2$ provided by Lemma~\ref{lemma:decomp}, and obtain
\begin{eqnarray*}\nonumber
\lefteqn{-\nabla \cdot \aa \nabla w_{\ell}(u_{\hom , \ell}) }
\\
&=&- \bar \aa_0\cdot \nabla^{2}u_{\hom , \ell}+\cc\cdot \nabla^{4}u_{\hom , \ell}-(\bb \otimes\bar \aa_0)\cdot\nabla^4u_{\hom , \ell}\\ \nonumber 
&&-\sum_{j=1}^{\ell-1} (\varphi_j\otimes\bar \aa_0)\cdot \nabla^{j+2}u_{\hom , \ell}  -\delta_{\ell=4}((-\varphi_1\otimes \cc+\varphi_1\otimes \bb \otimes\bar \aa_0)\cdot \nabla^5u_{\hom , \ell})
\\ 
 &&+\nabla\chi_{\ell}\cdot \nabla^{\ell+2}u_{\hom , \ell}-\nabla \cdot \left[ (\aa \otimes\varphi_{\ell}-\sigma_{\ell}+\nabla\chi_{\ell})\cdot\nabla^{\ell+1}u_{\hom , \ell}\right].
\end{eqnarray*}
We now use the defining equation \eqref{e.hom-eq-alaDLS} for $u_{\ho,\ell}$ in form of 
$\bar \aa_0 \cdot \nabla^2u_{\hom , \ell}=-f+\bb\cdot \nabla^2 f+\cc \cdot \nabla^4u_{\hom , \ell}$
to rewrite the following terms as
\begin{eqnarray*}
(\bb \otimes\bar \aa_0)\cdot\nabla^4u_{\hom , \ell}&=&-\bb \cdot \nabla^2 f +\bb ^{\otimes 2}\cdot \nabla^4 f +(\bb \otimes \cc)\cdot \nabla^6 u_{\hom , \ell} , 
\\ 
(\varphi_1 \otimes \bb \otimes \bar \aa_0)\cdot\nabla^5u_{\hom , \ell}&=& -(\varphi_1 \otimes \bb )\cdot\nabla^3 f +(\varphi_1 \otimes \bb ^{\otimes 2})\cdot \nabla^5 f +(\varphi_1 \otimes \bb \otimes \cc)\cdot \nabla^7 u_{\hom , \ell} 
\\ 
 (\varphi_j\otimes \bar \aa_0)\cdot \nabla^{j+2}u_{\hom , \ell} &=&-\varphi_j \cdot \nabla^j f + (\varphi_j \otimes \bb  )\cdot \nabla^{j+2}f +(\varphi_j \otimes \cc)\cdot \nabla^{j+4}u_{\hom , \ell}.
\end{eqnarray*}
Inserting these identities in the above and rearranging the terms, we obtain the identity valid for $3\le \ell \le 4$:
\begin{eqnarray*}
\lefteqn{-\nabla \cdot \aa \nabla w_{\ell}(u_{\hom , \ell}) }
\\
&=& f+\sum_{j=1}^{\ell-1} \varphi_j \cdot \nabla^j f +\sum_{j=\ell-2}^{\ell-1} (\varphi_j \otimes \bb )\cdot \nabla^{j+2}f-\sum_{j=\ell-2}^{\ell-1} (\varphi_j \otimes \cc)\cdot \nabla^{j+4}u_{\hom , \ell}
\\ 
&&-\sum_{j=0}^{\ell-2}(\varphi_j \otimes \bb ^{\otimes 2})\cdot \nabla^{j+4}f-\sum_{j=0}^{\ell-2}(\varphi_j \otimes \bb \otimes \cc)\cdot \nabla^{j+6}u_{\hom , \ell}  
\\  
&&+\nabla\chi_{\ell}\cdot \nabla^{\ell+2}u_{\hom , \ell}-\nabla \cdot \left[ (\aa \otimes \varphi_{\ell}-\sigma_{\ell}+\nabla\chi_{\ell})\cdot\nabla^{\ell+1}u_{\hom , \ell}\right].
\end{eqnarray*}
Substracting this from \eqref{e.origin-ell} yields for $h_{\ell}:=u_{\ell}-u_{\hom , \ell}$:
\begin{multline*}
-\nabla \cdot \aa \nabla h_\ell \,=\,\nabla \cdot \left[ (\aa \otimes \varphi_{\ell}-\sigma_{\ell}+\nabla\chi_{\ell})\cdot\nabla^{\ell+1}u_{\hom , \ell}\right]
\\
-\nabla\chi_{\ell}\cdot \nabla^{\ell+2}u_{\hom , \ell} + \sum_{j=\ell-2}^{\ell-1} (\varphi_j \otimes \bb )\cdot \nabla^{j+2}f-\sum_{j=\ell-2}^{\ell-1} (\varphi_j \otimes \cc)\cdot \nabla^{j+4}u_{\hom , \ell}
\\ 
+\sum_{j=0}^{\ell-2}(\varphi_j \otimes \bb ^{\otimes 2})\cdot \nabla^{j+4}f-\sum_{j=0}^{\ell-2}(\varphi_j \otimes \bb \otimes \cc)\cdot \nabla^{j+6}u_{\hom , \ell}  .
\end{multline*}
Before we proceed with the energy estimate, we reformulate the (higher-order) last RHS term using the homogenized equation to reduce the number of derivatives on $u_{\hom , \ell}$ (and ultimately improve the estimate).
More precisely, we shall use \eqref{e.hom-eq-alaDLS} in the form
\begin{equation*}
\nabla^{\ell}(\textbf{c}\cdot \nabla^{4} u_{\ho,\ell})\,=\, \nabla^\ell(\nabla \cdot \bar \aa_0  \nabla u_{\ho,\ell} +f-\textbf{b}\cdot \nabla^2 f),
\end{equation*}
so that the equation turns into
\begin{multline*}
-\nabla \cdot \aa \nabla h_\ell \,=\,\nabla \cdot \left[ (\aa \otimes \varphi_{\ell}-\sigma_{\ell}+\nabla\chi_{\ell})\cdot\nabla^{\ell+1}u_{\hom , \ell}\right]
\\
-\nabla\chi_{\ell}\cdot \nabla^{\ell+2}u_{\hom , \ell} + \sum_{j=\ell-2}^{\ell-1} (\varphi_j \otimes \bb )\cdot \nabla^{j+2}f-\sum_{j=\ell-2}^{\ell-1} (\varphi_j \otimes \cc)\cdot \nabla^{j+4}u_{\hom , \ell}
\\ 
+\sum_{j=0}^{\ell-2}(\varphi_j \otimes \bb ^{\otimes 2})\cdot \nabla^{j+4}f-\sum_{j=0}^{\ell-3}(\varphi_j \otimes \bb \otimes \cc)\cdot \nabla^{j+6}u_{\hom , \ell} 
\\
 -(\phi_{\ell-2}\otimes \bb) \cdot \nabla^\ell(\nabla \cdot \bar \aa_0  \nabla u_{\ho,\ell} +f-\textbf{b}\cdot \nabla^2 f).
\end{multline*}
The desired estimate  \eqref{e.multiscale-DLS} now follows as in Substep~3.2.

%%%%%%%%%%%%%%%%%
%%%%%%%%%%%%%%%%%
%%%%%%%%%%%%%%%%%
%%%%%%%%%%%%%%%%%
%%%%%%%%%%%%%%%%%
%%%%%%%%%%%%%%%%%

\section{Asymptotic ballistic transport of classical waves}\label{sec:deloc}

\subsection{Statement of the result}

The following definition introduces the notion of asymptotic ballistic transport
for the wave operator $\square=\partial^2_{tt} -\nabla \cdot \aa \nabla$.
\begin{defi}\label{defi:deloc}
Denote by $S:\R_+\times L^2(\R^d)\to L^2(\R^d),(t,u_0)\mapsto S_t(u_0)$ the semi-group associated with the initial value problem
\begin{equation*}
\left\{
\begin{array}{rcl}
\square S_t(u_0)&=&0,
\\
S_0(u_0)&=&u_0,
\\
\partial_t S_t(u_0)|_{t=0}&=&0.
\end{array}
\right.
\end{equation*}
For all $\lambda>0$, let $G_\lambda$ be the centered Gaussian normalized in $L^2(\R^d)$ and of support of size $\textcolor{red}{\lambda^{-1}}$, that is,
\begin{equation}\label{def:gaus}
G_\lambda(x)=\big(\frac{\lambda}{\pi}\big)^{d/2} \exp(-\frac12 \lambda^2 |x|^2).
\end{equation}
For all $T\ge 0$, we set
\begin{equation*}
M(\lambda,T)\,:=\,   \Big(\int_{\R^d} (1+ \lambda |x|)^2 S_T(G_\lambda)^2dx\Big)^\frac12,
\quad
\calM(\lambda,T)\,:=\, \Big(\fint_T^{T+\lambda^{-1}} \expec{M(\lambda,t)^2} dt\Big)^\frac12.
\end{equation*}
We say that $\square$ displays \emph{ballistic transport} at energy $\lambda>0$ if
for all $T\ge 0$,
\begin{equation}\label{e.dyn-deloc}
\calM(\lambda,\lambda^{-1} T)\,\gtrsim \,T.
\end{equation}
We say that $\square$ displays \emph{asymptotic ballistic transport} at $0$ of order $\gamma\ge 0$
if there exists $T>0$ such that for all $0<\e\ll 1$ small enough
\begin{equation}\label{e.as-dyn-deloc}
\calM(\e,\e^{-2-\gamma}T) \,\gtrsim\, \e^{-1-\gamma}T.
\end{equation}
\qed
\end{defi}
Let us comment on this definition.
First note that $\calM(\lambda,0)\sim 1$ by a direct calculation.
In the definition of  $\calM(\lambda,T)$ we average in time over $(T,T+\lambda^{-1})$ instead of considering a pointwise-in-time quantity: indeed, the $L^2$-norm is not a conserved quantity (the invariant quantity involves the kinetic energy as well) and may vanish at some specific times, but not on average (the choice of time $\lambda^{-1}$ is related to the expected speed $\lambda$ of the wave).
Ballistic transport of an initial wave $G_\lambda$ takes place if this wave is essentially transported at speed $\lambda$ (as it is the case for 
a constant-coefficient wave equation). In particular, if $\aa\equiv \Id$, a direct calculation yields for all $\lambda>0$ and $\tau\ge 0$
\begin{equation*}\label{e.ballistic-explain}
\calM(\lambda,\lambda^{-1} T)\,\sim \, T,
\end{equation*}
(that is, $\frac1CT \le \calM(\lambda,\lambda^{-1} T)\le CT$ for some multiplicative constant $C$ independent of $\lambda>0$ and $T\ge 1$),
which, in view of the weighted norm, illustrates that most of the mass is transported at distance $T$ from the origin. This explains \eqref{e.dyn-deloc}.

Let us turn to \eqref{e.as-dyn-deloc}.
Transport is only significant if the support of $G_\lambda$ has moved, 
which requires $T$ in \eqref{e.dyn-deloc} to be at least of order $\lambda^{-2}$ (since the support of $G_\lambda$ has size $\lambda^{-1}$
and the speed of propagation is $\lambda$).
This explains the scaling in $\e$ in \eqref{e.as-dyn-deloc}, and the
wording of \emph{asymptotic ballistic transport}:
\begin{itemize}
\item The result is \emph{asymptotic} because the final time $\tau=\e^{-2-\gamma}T$ one can consider depends on the energy level $\e$.
\item There is effective \emph{transport} because a significant part of the mass has moved by a distance which, measured in the unit $\e^{-1}$ of the typical length-scale at initial time, is bounded by below uniformly in $\e>0$. More precisely, by \eqref{e.as-dyn-deloc} and the definition of the weighted norm, this distance is of order $\e^{-1-\gamma}$, so that the ratio $\frac{\e^{-1-\gamma}}{\e^{-1}}=\e^{-\gamma}$ is isolated from zero as soon as $\gamma\ge 0$. (There would be no asymptotic transport if  \eqref{e.as-dyn-deloc} only held for some $\gamma<0$, as it is the case for the Poisson inclusions in dimensions $d\le 2$, see below).
\item The transport is \emph{ballistic} because it satisfies the ballistic scaling property \eqref{e.dyn-deloc}.
\end{itemize}

\medskip

Note that we could also consider higher-order moments and use $(1+\lambda |x|)^{2p}$ as a weight instead of $(1+\lambda |x|)^2$, in which case
the RHS of \eqref{e.as-dyn-deloc} would be replaced by $(\e^{-1-\gamma}T)^p$ in the definition (and in Theorem~\ref{t.deloc} below, the proof of which adapts straightforwardly).

\medskip

Our main result is as follows:
\begin{theo}\label{t.deloc}
Let ${\ell\ge 2}$, and assume that $(\phi_j,\sigma_j,\chi_j)_{0\le j\le \ell}$ satisfy Hypothesis~\ref{hypo:corr}
for some $\alpha=(\alpha_1,\alpha_2)\in [0,1)\times \R_+$.
Then for all $\gamma\ge 0$, we have for all $\e\ll 1$ and all $T\ge 0$,
\begin{equation}\label{t.deloc1}
\calM(\e,\e^{-2-\gamma}T)\,\gtrsim\, \e^{-1-\gamma}T (1-C\e^{\ell-1-\gamma} T\mu_\alpha(\e^{-2-\gamma}T)),
\end{equation}
where the constant $0<C<\infty$ only depends on $\bar \Gamma_\ell:=\max_{0\le j\le \ell-1} |\bar \aa_j|$, $d$, $\gamma$, $\ell$, and $\alpha$.
In particular, the associated wave operator $\square$ displays asymptotic ballistic transport at $0$ 
provided $\ell=2$ and $\alpha_1< \frac12$ or $\ell>2$ (no condition on $\alpha$), in which case
we have for all  $0\le \gamma<{\frac{\ell-1-2\alpha_1}{1+\alpha_1}}$, all $T<\infty$, and all $0< \e \ll 1$,
\begin{equation}\label{t.deloc3}
\calM(\e,\e^{-2-\gamma}T)\, \gtrsim\, \e^{-1-\gamma}T .
\end{equation}
In the borderline case  ${\ell=2}$ and $\alpha=(\frac12,0)$, $\square$ displays asymptotic ballistic transport at $0$ in the sense that
for $\gamma=0$ and all $0<T\ll1$ small enough, we have for all $0< \e \ll 1$,
\begin{equation}\label{t.deloc2}
\calM(\e,\e^{-2}T)\, \gtrsim\, \e^{-1}T.
\end{equation}
\qed
\end{theo}
If the extended correctors blow up more rapidly than in the assumptions of Theorem~\ref{t.deloc}, we cannot conclude that the support at final time has moved with respect
to the support at initial time in the asymptotic regime $\e \downarrow 0$.
Let us make this result more explicit in three interesting examples:  
\begin{itemize}
\item For periodic coefficients, one can prove that the multiplicative constant in \eqref{t.deloc3} only grows exponentially with $\gamma$, so that for  $0< \e \ll 1$ small enough, one may upgrade  \eqref{t.deloc3} to $\calM(\e,\e^{-1}T)\, \gtrsim\,  T$ for all $T\ge 0$, and obtain ballistic transport at all times (thus recovering this classical result for low frequencies without explicit use of the Bloch theorem).
\item For smooth quasi-periodic coefficients satisfying a diophantine condition, there is asymptotic ballistic transport in any dimension at any order $\gamma\ge 0$. This is however not quite enough to prove ballistic transport at all times since the multiplicative constant in \eqref{t.deloc3} grows in this case more than exponentially with $\gamma$.  We believe there could be ballistic transport at all times, although our approach currently fails to prove so.
\item For Poisson random inclusions (or Gaussian coefficient fields with compactly supported correlations), there is asymptotic ballistic transport in dimensions $d>2$ (cf. \cite{GO3,GO4} and Appendix~\ref{sec:corr}
for the desired bounds on the correctors). More precisely, for all $T\ge 0$ and all $0\le \gamma< [\frac d2]-1$ we have  $\calM(\e,\e^{-2-\gamma}T)\, \gtrsim\,  \e^{-1-\gamma}T$. For odd dimensions $d\ge 3$, one can choose
$\gamma= [\frac d2]-1$ provided $0<T\ll 1$.
In particular, the scaling for asymptotic ballistic transport improves with dimension.
\end{itemize}
Corresponding results for more general statistics of $\aa$ follow in a straightforward way from Theorem~\ref{t.deloc} and Appendix~\ref{sec:corr}.
\begin{rem}
Theorem~\ref{t.deloc} is stated for second moments in probability in view of the definition of $\calM(\lambda,T)$.
As already pointed out in Remark~\ref{rem:sto-int}, if one makes stronger assumptions on the growth of the correctors in probability, 
one gets stronger results in terms of stochastic integrability.
In all the stochastic examples of this article, the bounds we have on the growth of correctors are indeed quenched (or ``path-wise'' if we were talking about thermal fluctuations): they hold almost surely up to multiplicative constants which are random but have (typically) stretched exponential moments.
In particular, in these cases, the asymptotic transport result of Theorem~\ref{t.deloc} is also quenched (in the sense we do not need to take the expectation in the definition of $\calM(\lambda,T)$, in which case the constant $C$ in the RHS of \eqref{t.deloc1} is random with stretched exponential moments).
\qed
\end{rem}

\medskip

Let us emphasize that the choice of a Gaussian initial condition in Theorem~\ref{t.deloc} is convenient but not essential to the proof: we could indeed consider any properly-rescaled function of the Schwartz class. 
In terms of spectral interpretation of Theorem~\ref{t.deloc}, \eqref{t.deloc3} essentially suggests that if there are localized states at energy $\e\ll 1$,
then their supports are expected to scale like at least as $\e^{-1-\gamma}$ (in the spirit of the results \cite{C} for the Schr\"odinger operator).

\subsection{Proof of Theorem~\ref{t.deloc}: Asymptotic ballistic transport}

The general strategy is as follows: To prove asymptotic ballistic transport of the solution to the wave equation, we first consider
an approximation of the solution by Taylor-Bloch waves, then prove asymptotic ballistic transport for this approximate solution, and finally conclude
that the approximation is good enough so that the exact solution inherits the transport properties of the approximate solution.
More precisely, we split the proof into three steps. In the first step we rescale the problem in order to place ourselves in the framework of Section~\ref{sec:wave}
and appeal to Taylor-Bloch waves.
The next step consists in showing that the approximate solution (in form of explicit Taylor-Bloch waves) displays the desired asymptotic ballistic transport,
which is the aim of Step~2. A possible strategy could have been to directly rely on the homogenized wave equation to prove the asymptotic ballistic transport.
The difficulty is that we do not have error estimates in weighted spaces (whereas we have to integrate with respect to $|x|^2dx$ to prove ballistic transport), so
that this natural approach might not be applicable. Instead, we first localize in space (which allows to estimate $|x|^2$ by its supremum on the bounded domain),
and introduce a proxy $\mathcal M_\ell$ for $\mathcal M$. In order not to destroy the structure in frequency space, we localize with Gaussians in the definition
of $\mathcal M_\ell$.
Asymptotic ballistic transport amounts to controlling the quantity $\mathcal M_\ell$ by below. This quantity is an integral with respect to a Gaussian times $|x|^2dx$,
which is easier to estimate in Fourier space --- Step~2 is the most technical step.
In the last step, it remains to show that the solution displays asymptotic ballistic transport if the approximate solution does,  which we prove by combining
the results of Section~\ref{sec:wave} with the decay of the Gaussian cut-off.

\medskip

\step{1} Reformulation.

\noindent By the hyperbolic rescaling $(t,x)\leadsto (t',x')=(\e t,\e x)$, the moment takes the form
$$
\calM(\e,\e^{-1}T)\,=\, \Big(\int_T^{T+1} \expec{\int_{\R^d} (1+ |x|)^2 u_\e(t,x)^2 dx} dt\Big)^\frac12
$$
where $u_\e$ solves the initial value problem
\begin{equation*}
\left\{
\begin{array}{rcl}
\square_\e u_\e(t,x)&=&0,
\\
u_\e(0,x)&=&G_1(x),
\\
\partial_t u_\e(0,x)&=&0,
\end{array}
\right.
\end{equation*}
and $G_1$ is the Gaussian defined in \eqref{def:gaus}. In particular, we are in the realm of large-time homogenization.
Recall the approximate solution of Theorem~\ref{t.wave-eq} in Section~\ref{sec:wave}, given by
\begin{eqnarray*}
u_{\e,\ell}(t,x) &=&  \frac{1}{(2\pi)^d} \int_{\R^d} \omega_{\ell}(\e |k|) \hat G_1(k) e^{ik\cdot x} \cos(\e^{-1} \Lambda_\ell(\e k)t)dk,
\end{eqnarray*}
where $\Lambda_\ell(k)\,:=\,  \sqrt{\tilde \lambda_{k,\ell}}$, which is well-defined since $\tilde \lambda_{\e k,\ell}\ge 0$ when $\omega_{\ell}(\e |k|) \ne 0$. 
We shall compare the moment $\calM(\e,\e^{-1}T)$ to some related moment of the approximate solution $u_{\e,\ell}$.
For reasons which will be clear in Step~3 below, we need a localized moment for the approximate solution $u_{\e,\ell}$.
For some $\mathcal C\gg 1$ that will be fixed in Step~2 (and ultimately only depends on $\bar \Gamma_\ell$ and $d$), we set
$$
\calM_\ell(\e,\e^{-1}T)\,:=\, \Big(\big( \frac{\mathcal CT}{2}\big)^{d}\ \int_T^{T+1} {\int_{\R^d} (1+ |x|)^2 u_{\e,\ell}(t,x)^2G_{( \mathcal CT)^{-1}}^2(x) dx} dt\Big)^\frac12,
$$
where $G_{(\mathcal CT)^{-1}}$ is our Gaussian function \eqref{def:gaus}.
Note that $\sup_{\R^d} (\frac{\mathcal CT}{2}\big)^{d} G_{(\mathcal CT)^{-1}}^2 \lesssim 1$.
We shall argue in Step~2 that $\calM_\ell(\e,\e^{-1}T)$ has the desired ballistic scaling in time provided $\mathcal C$ is chosen large enough,
and then argue in Step~3 that the moment of $u_\e$ is indeed essentially bounded by below by  $\calM_\ell(\e,\e^{-1}T)$ using Theorem~\ref{t.wave-eq} and our choice of localizing the moment of the approximate solution $u_{\e,\ell}$.

\medskip

\step{2} Moment of the approximate solution.

\noindent 
In this step, we use the notation $\lesssim$ to denote $\le C\times$ for a constant $C<\infty$ which \emph{does not} depend on 
the constant $\mathcal C$ of the Gaussian kernel, and we always make the dependence upon $\mathcal C$ explicit.
By Plancherel's formula, we may reformulate the moment of $u_{\e,\ell}$ as
\begin{eqnarray*}
\calM_\ell(\e,\e^{-1}T)^2&=&\big( \frac{\mathcal CT}{2}\big)^{d}\int_T^{T+1} \int_{\R^d} (1+|x|)^2 u_{\e,\ell}^2(t,x)G_{(\mathcal CT)^{-1}}^2(x)dx dt \nonumber \\
&\ge &\big( \frac{\mathcal CT}{2}\big)^{d} \int_T^{T+1} \int_{\R^d} |x|^2 u_{\e,\ell}^2(t,x)G_{(\mathcal CT)^{-1}}^2(x)dx dt \nonumber \\
&=&\int_T^{T+1} \int_{\R^d} |\nabla_k \hat u_{\e,\ell}* \big( \frac{\mathcal CT}{2}\big)^{d/2} \hat G_{(\mathcal CT)^{-1}}|^2(t,k)dkdt.
\end{eqnarray*}
Since $\hat G_{(\mathcal CT)^{-1}}(k)=(2\pi)^{d/2} G_{\mathcal CT}(k)$ and with the notation 
$$\bar G_{\mathcal CT}(k)\,:=\, (\frac{\mathcal CT}{\sqrt{2\pi}})^{d} \exp(-\frac12 \mathcal C^2T^2|k|^2)$$
 (so that this Gaussian has mass unity), we may write the above as
\begin{equation}
\calM_\ell(\e,\e^{-1}T)^2 \,\geq  \, \int_T^{T+1} \int_{\R^d} |\nabla_k \hat u_{\e,\ell}* \bar G_{\mathcal CT}|^2(t,k)dkdt \label{e.T0}.
\end{equation}
Using the following more explicit formula for $u_{\e,\ell}$
\begin{eqnarray*}
u_{\e,\ell}(t,x) &=&  \frac{1}{(2\pi)^d} \int_{\R^d} \omega_{\ell}(\e |k|) (2\pi)^{d/2}\exp(- \frac{|k|^2}{2} ) e^{ik\cdot x} \cos(\e^{-1} \Lambda_\ell(\e k)t)dk,
\end{eqnarray*}
we have
\begin{eqnarray*}
\nabla_k \hat u_{\e,\ell}(t,k)&=& \underbrace{\e \frac{k}{|k|}  \omega_{\ell}'(\e |k|)  (2\pi)^{d/2}\exp(- \frac{|k|^2}{2} ) \cos(\e^{-1} \Lambda_\ell(\e k)t) 
-k \hat u_{\e,\ell}(t,k)}_{\dps =:\,\mathcal T_{1,\e,k,t}}\\
&&
\\
&&- \underbrace{t \nabla \Lambda_\ell(\e k) \omega_{\ell}(\e |k|)  (2\pi)^{d/2}\exp(- \frac{|k|^2}{2} )  \sin(\e^{-1} \Lambda_\ell(\e k)t)}_{\dps =:\,\mathcal T_{2,\e,k,t}}.
\end{eqnarray*}
The dominating term is $\mathcal T_{2,\e,k,t}$, which displays the desired ballistic scaling $t$.
We first prove that the contribution of $\mathcal T_{1,\e,k,t}$ remains of order 1 (this is an \emph{upper bound}), and then 
show that the contribution of $\mathcal T_{2,\e,k,t}$ is indeed ballistic (this is a \emph{lower bound}).
On the one hand, by Young's inequality for convolutions,
$$
\int_T^{T+1} \int_{\R^d}|\mathcal T_{1,\e,k,t} * \bar G_{\mathcal CT}|^2(t,k)dkdt
\,\leq \, \int_T^{T+1} \int_{\R^d}|\mathcal T_{1,\e,k,t}|^2(t,k)dkdt.
$$
On the other hand, by the boundedness of $\omega_{\ell}'$ and $\cos$, this yields  for all $\e \le 1$
\begin{equation}\label{e.T1}
\int_T^{T+1} \int_{\R^d} |\mathcal T_{1,\e,k,t}*   \bar G_{\mathcal CT}|^2dkdt \,\lesssim  \,  \int_T^{T+1} \int_{\R^d} (1+|k|)^2\exp(- |k|^2 )dk dt \,\lesssim \, 1
\end{equation}
(the multiplicative constant is independent of $\mathcal C$).
Let us turn to the contribution of $\mathcal T_{2,\e,k,t}$, which is slightly more subtle.
We write the convolution as follows:
\begin{equation}\label{e.T20}
\mathcal T_{2,\e,k,t} * \bar G_{\mathcal CT} \,=\, \mathcal T_{2,\e,k,t}+\int_{\R^d} (\mathcal T_{2,\e,k',t}-\mathcal T_{2,\e,k,t}) \bar G_{\mathcal CT}(k-k')dk'.
\end{equation}
Indeed, we expect the Gaussian $\bar G_{\mathcal CT}$ to be peaked enough so that it acts as a Dirac mass on $\mathcal T_{2,\e,k,t}$ at leading order, 
which allows us to prove the desired ballistic lower bound.
In order to prove that this decomposition is valid (that is, that the second RHS term is higher order), we start by estimating the first RHS term of \eqref{e.T20} by below.
By Fubini's theorem,
\begin{multline*}
\int_T^{T+1} \int_{\R^d} |\mathcal T_{2,\e,k,t}|^2dkdt \\
\gtrsim \, \int_{\R^d} T^2 |\nabla \Lambda_\ell(\e k)|^2 \omega_{\ell}^2(\e |k|) \exp(- |k|^2)\int_T^{T+1}
\sin^2(\e^{-1} \Lambda_\ell(\e k)t)dtdk.
\end{multline*}
By definition of $\Km$  and $\omega_\ell$ in Section~\ref{sec:Kmax}, for all $\e>0$ and $k\in \R^d$ such that $\omega_\ell(\e |k|)\ne 0$, we have $(\e|k|)^{-1} \Lambda_\ell(\e k)\ge \frac12$. In particular,
$\e^{-1} \Lambda_\ell(\e k)\ge \frac12$ for $|k|\ge 1$, so that for all $k\in \R^d$ such that $\omega_\ell(\e |k|)\ne 0$,
$$
\int_T^{T+1}
\sin^2(\e^{-1} \Lambda_\ell(\e k)t)dt \, \gtrsim \, 1_{|k|\ge 1},
$$
and therefore
$$
\int_T^{T+1} \int_{\R^d} |\mathcal T_{2,\e,k,t}|^2dkdt \,\gtrsim \, T^2 \int_{|k|\ge 1}  |\nabla \Lambda_\ell(\e k)|^2 \omega_{\ell}^2(\e |k|) \exp(- |k|^2) dk.
$$
Recall that for all $k=\kappa e \in \R^d$ such that $\omega_{\ell}(|k|)\ne 0$, 
$$
\Lambda_\ell(k)= \sqrt{\tilde \lambda_{k,\ell}}\,=\,  \kappa\sqrt{\sum_{j\ge 0,2j<\ell} (-1)^j\kappa^{2j} \lambda_{2j}^e},
$$
so that
$$
 |\nabla \Lambda_\ell( k)|\, \ge \, \sqrt{e\cdot \bar \aa_0 e}-c\kappa^{[\ell/2]} \ge 1-c|k|^{[\ell/2]}
$$
for some $0<c<\infty$ that only depends on the ellipticity constant $\Lambda$ and the dimension~$d$. 
Up to slightly reducing $\Km>0$, this yields for all $|k|\le \Km$, $|\nabla \Lambda_\ell(k)| \omega_{\ell}( |k|)\gtrsim 1$.
Hence, for all $\e\ll1$ small enough (where smallness depends only on $\bar \Gamma_\ell$ and $d$), we have
\begin{equation}\label{e.T2}
\int_T^{T+1} \int_{\R^d} |\mathcal T_{2,\e,k,t}|^2dkdt \,\gtrsim \, T^2 \int_{\e^{-1}\Km\ge |k|\ge 1} \exp(- |k|^2) dk\, \gtrsim \, T^2.
\end{equation}
We now address the second RHS term in \eqref{e.T20}.
For all $k\in \R^d$ we write
\begin{multline*}
\int_{\R^d} (\mathcal T_{2,\e,k',t}-\mathcal T_{2,\e,k,t}) \bar G_{\mathcal CT}(k-k')dk'
\,=\,
\int_{|k'-k|\le \frac{|k|}{4} \vee 2} (\mathcal T_{2,\e,k',t}-\mathcal T_{2,\e,k,t}) \bar G_{\mathcal CT}(k-k')dk'
\\
+\int_{|k'-k|> \frac{|k|}{4} \vee 2} (\mathcal T_{2,\e,k',t}-\mathcal T_{2,\e,k,t}) \bar G_{\mathcal CT}(k-k')dk'.
\end{multline*}
For the first integral term we use a Lipschitz bound on $k'\mapsto \mathcal T_{2,\e,k',t}$, whereas for the second integral term we exploit the
exponential decay of the averaging kernel.
Indeed, 
\begin{multline*}
\nabla \mathcal T_{2,\e,k',t}\,=\,t^2 \nabla \Lambda_\ell(\e k')\otimes  \nabla \Lambda_\ell(\e k') \omega_{\ell}(\e |k'|)  (2\pi)^{d/2}\exp(- \frac{|k'|^2}{2} )  \cos(\e^{-1} \Lambda_\ell(\e k')t) 
\\+ \e t \nabla \Lambda_\ell(\e k')\otimes  \frac{k'}{|k'|}\omega_{\ell}'(\e |k'|)  (2\pi)^{d/2}\exp(- \frac{|k'|^2}{2} )  \sin(\e^{-1} \Lambda_\ell(\e k')t) 
\\+\e t \nabla^2 \Lambda_\ell(\e k')\omega_{\ell}(\e |k'|)  (2\pi)^{d/2}\exp(- \frac{|k'|^2}{2} )  \sin(\e^{-1} \Lambda_\ell(\e k')t) 
\\ + t\nabla \Lambda_\ell(\e k')\otimes k'  \omega_{\ell}(\e |k'|)  (2\pi)^{d/2}\exp(- \frac{|k'|^2}{2} )  \sin(\e^{-1} \Lambda_\ell(\e k')t),
\end{multline*}
so that on the set $\mathcal{K}_k:=\{k':|k'-k|\le \frac{|k|}{4} \vee 2\}$, we have for all $\e\le 1$,
$$
\sup_{k'\in \mathcal{K}_k} |\nabla \mathcal T_{2,\e,k',t}| \,\lesssim \, t^2 |k|^{\ell+1} \exp(-\frac{9|k|^2}{32})\,\lesssim\, t^2 \exp(-\frac{|k|^2}{4}).
$$
We thus obtain
\begin{eqnarray}
\lefteqn{\int_{|k'-k|\le \frac{|k|}{4} \vee 2} |\mathcal T_{2,\e,k',t}-\mathcal T_{2,\e,k,t}| \bar G_{\mathcal CT}(k-k')dk'}\nonumber
\\
& \lesssim & t^2 \exp(-\frac{|k|^2}{4}) \int_{|k'-k|\le \frac{|k|}{4} \vee 2} |k-k'| \bar G_{\mathcal CT}(k-k')dk' \nonumber\\
&\lesssim & \frac{t^2}{\mathcal CT}  \exp(-\frac{|k|^2}{4})  \int_{\R^d} \mathcal CT|k-k'| \bar G_{\mathcal CT}(k-k')dk' \nonumber\\
&\lesssim & \frac{t^2}{\mathcal C T} \exp(-\frac{|k|^2}{4}), \label{e.T2-10}
\end{eqnarray}
where the multiplicative constant does depend on $\Lambda$ and $d$, but not on $\mathcal C$.
We treat now the second integral term, and simply bound $\mathcal T_{2,\e,k',t}$ by $t$: 
\begin{eqnarray*}
|\mathcal T_{2,\e,k',t}|&=&|t \nabla \Lambda_\ell(\e k') \omega_{\ell}(\e |k'|)  (2\pi)^{d/2}\exp(- \frac{|k'|^2}{2} )  \sin(\e^{-1} \Lambda_\ell(\e k')t)|
\\
&\lesssim & t |k'|^{[\ell/2]} \exp(- \frac{|k'|^2}{2} )\,\lesssim\, t.
\end{eqnarray*}
Hence, we obtain by direct integration of $\bar G_{\mathcal CT}$
\begin{eqnarray}
\int_{|k'-k|> \frac{|k|}{4} \vee 2} |\mathcal T_{2,\e,k',t}-\mathcal T_{2,\e,k,t}| \bar G_{\mathcal CT}(k-k')dk'
&\lesssim & t \int_{|k'|> \frac{|k|}{4} \vee 2}\bar G_{\mathcal CT}(k')dk'\nonumber
\\
&\lesssim & t \exp\big(-\frac14 \mathcal C^2T^2 (\frac{|k|}{4}\vee 2)^2\big)  . \label{e.T2-11}
\end{eqnarray}
Combining \eqref{e.T2-10} and  \eqref{e.T2-11}, we thus obtain
\begin{eqnarray}
\lefteqn{\int_{T}^{T+1} \int_{\R^d} \Big(\int_{\R^d} (\mathcal T_{2,\e,k',t}-\mathcal T_{2,\e,k,t}) \bar G_{\mathcal CT}(k-k')dk'\Big)^2dkdt }\nonumber
\\
&\lesssim & \int_{\R^d}\Big( \frac{T^2}{\mathcal C} \exp(-\frac{|k|^2}{4})+T^2 \exp(-\frac12 \mathcal C^2T^2 (\frac{|k|}{4}\vee 2)^2\Big) dk ,\label{e.T2-12}
\end{eqnarray}
(note that the multiplicative constant depends on $\Lambda$ and $d$ but not on $\mathcal C$).

\medskip

We are in the position to conclude this step. 
For $\e\ll 1$ and $T\ge 1$, the four estimates \eqref{e.T0}, \eqref{e.T1}, \eqref{e.T2}, and \eqref{e.T2-12} combine to 
$$
\calM_\ell(\e,\e^{-1}T)^2\,\ge \, c T^2(1-\frac1{\mathcal C}-\exp(-\frac12\mathcal C^2))
$$
for some constant $c>0$ which does not depend on $\mathcal C$.
This turns into the desired ballistic estimate
\begin{equation}\label{e.approx-ballistic}
\calM_\ell(\e,\e^{-1}T)\,\gtrsim\, T
\end{equation}
provided $\mathcal C$ is chosen large enough and $T\ge 1$.

\medskip

\step{3} Control of the error between $\calM(\e,\e^{-1}T)$ and $\calM_\ell(\e,\e^{-1}T)$.

\noindent Since $\sup_{\R^d} (\frac{\mathcal CT}{2}\big)^{d} G_{(\mathcal CT)^{-1}}^2 \lesssim 1$, we have 
$$
\calM(\e,\e^{-1}T) \,\gtrsim \, \Big(\int_T^{T+1} \expec{\int_{\R^d} |x|^2 u_\e(t,x)^2 (\frac{\mathcal CT}{2}\big)^{d} G_{(\mathcal CT)^{-1}}^2(x)dx} dt\Big)^\frac12,
$$
which, by the triangle inequality and the definition of $G_{(\mathcal CT)^{-1}}$, yields
\begin{multline}\label{e.T1000}
\calM(\e,\e^{-1}T) \,\gtrsim \, \calM_\ell(\e,\e^{-1}T)
\\
-  \Big(\int_T^{T+1} \expec{\int_{\R^d} |x|^2 |u_\e(t,x)-u_{\e,\ell}(t,x)|^2 \exp(-\frac{|x|^2}{\mathcal C^2T^2})dx} dt\Big)^\frac12.
\end{multline}
It remains to control the second RHS term. Thanks to the exponential weight, we have 
\begin{multline*}
\int_T^{T+1} \expec{\int_{\R^d} |x|^2 |u_\e(t,x)-u_{\e,\ell}(t,x)|^2 \exp(-\frac{|x|^2}{\mathcal C^2T^2})dx} dt
\\ \lesssim \,\mathcal C^2T^2  \sup_{T\le t\le T+1} \expec{\|u_\e(t)-u_{\e,\ell}(t)\|_{L^2(\R^d)}^2} .
\end{multline*}
By Theorem~\ref{t.wave-eq}, this turns into
\begin{multline}\label{e.T1001}
\int_T^{T+1} \expec{\int_{\R^d} |x|^2 |u_\e(t,x)-u_{\e,\ell}(t,x)|^2 \exp(-\frac{|x|^2}{\mathcal C^2T^2})dx} dt\\
\lesssim \, \mathcal C^2T^2\big( \max\{\e,{\e^{\ell}}\mu_\alpha(\e^{-1})\}+{\e^{\ell}}T\mu_{\alpha}(\e^{-1}T)\big)^2.
\end{multline}
The desired estimate~\eqref{t.deloc1} then follows from \eqref{e.T1000}, \eqref{e.T1001}, and \eqref{e.approx-ballistic}
with $T$ replaced by $\e^{-1-\gamma}T$ (for which the condition $\e^{-1-\gamma}T\ge 1$ for \eqref{e.approx-ballistic} in Step~2 is automatically satisfied in the
asymptotic regime $\e\ll 1$).

%%%%%%%%%%%%%%%%%
%%%%%%%%%%%%%%%%%
%%%%%%%%%%%%%%%%%
%%%%%%%%%%%%%%%%%
%%%%%%%%%%%%%%%%%
%%%%%%%%%%%%%%%%%

\appendix

\numberwithin{prop}{section} 
\numberwithin{theo}{section} 
\numberwithin{coro}{section} 
\numberwithin{lemma}{section}

\section{The case of a localized-in-time source term}\label{sec:forcing}

As emphasized in Remark~\ref{rem:forcing}, our approach allows one to deal with the alternative problem
\begin{equation}\label{i3}
\left\{
\begin{array}{rcl}
\square_\e u_\e(t,x)&=&f,
\\
u_\e(0,\cdot)&=&0,
\\
\partial_t u_\e(0,\cdot)&=&0,
\end{array}
\right.
\end{equation}
where $f$ has compact support in time (say in $[0,1]$) and in the Schwartz class in space. 
In the spirit of Section~\ref{sec:wave}, we define an approximation of $u_\e$ by
\begin{equation}\label{sol_approx_interior}
u_{\e, \ell}(t,x):=\frac{1}{(2\pi)^d}\int_{\mathbb{R}^d}\int_{0}^t \omega_\ell(\e \vert k \vert)\hat{f}(s,k)e^{ik\cdot x}\frac{\sin((\e^{-1}\Lambda_{\ell}(\e k))^2(t-s))}{(\e^{-1}\Lambda_{\ell}(\e k))^2}\psi_{\e k , \ell}\left( \frac{x}{\e}\right) \, ds \, dk , 
\end{equation}
where $\hat{f}(s,k)$ denotes the (partial) Fourier transform in the space variable only, which we may simplify in the form
\begin{equation}\label{sol_approx_interior2}
\tilde u_{\e,\ell}(t,x):=\frac{1}{(2\pi)^d}\int_{\mathbb{R}^d}\int_{0}^t \omega_\ell(\e \vert k \vert)\hat{f}(s,k)e^{ik\cdot x}\frac{\sin((\e^{-1}\Lambda_{\ell}(\e k))^2(t-s))}{(\e^{-1}\Lambda_{\ell}(\e k))^2} \, ds \, dk .
\end{equation}
We shall prove the following.
\begin{theo}\label{t.wave-eq-int}
Let $\ell\ge 1$, and assume that $(\phi_j,\sigma_j,\chi_j)_{0\le j\le \ell}$ satisfy Hypothesis~\ref{hypo:corr}
for some $\alpha\in [0,1)\times \R_+$. Then for all $T\ge 1$ and all $1\ge \e>0$, the solution $u_\e$ of \eqref{i3}
and the function $u_{\e,\ell}$ defined in \eqref{sol_approx_interior} satisfy
\begin{multline}\label{e.forcing_term}
\sup_{0\le t\le T} \expec{\|\partial_t(u_\e-u_{\e,\ell})\|_{L^2(\R^d)}^2+\|\nabla(u_\e-u_{\e,\ell})\|_{L^2(\R^d)}^2+T^{-2}\|u_\e-u_{\e,\ell}\|_{L^2(\R^d)}^2}^\frac12
 \\
 \lesssim \, C_\ell(f) (\max\{\e,\e^{\ell}\mu_\alpha(\e^{-1})\}+\e^{\ell}T\mu_{\alpha}(\e^{-1}T)),
\end{multline}
where $C_\ell(f)$ is a generic norm of $f$ which only depends on $\ell$ and $d$, and is finite for measurable $f:\R_+\times \R_d\to \R$ supported in $[0,1]$ in time, such that $f(s,\cdot)\in \mathcal S(\R^d)$ for all $s\in [0,1]$ and such that for all $n\in \N$, $\int_0^1 \|f(s,\cdot)\|_{H^n(\R^d)}ds<\infty$.
Likewise, we have for the simplified version $\tilde u_{\e,\ell}$ of $u_{\e,\ell}$ defined in  \eqref{sol_approx_interior2} 
\begin{equation}\label{e.forcing_term2}
\sup_{0\le t\le T} \expec{\|u_\e-\tilde u_{\e,\ell}\|_{L^2(\R^d)}^2}^\frac12
 \, \lesssim \, C_\ell(f) (T \max\{\e,\e^{\ell}\mu_\alpha(\e^{-1})\}+\e^{\ell}T\mu_{\alpha}(\e^{-1}T)).
\end{equation}
\qed
\end{theo}
\begin{rem}
Compared to Theorem~\ref{t.wave-eq}, Theorem~\ref{t.wave-eq-int} has the advantage to yield accuracy in the energy norm on top of the $L^2$-norm (provided
we keep the correctors). This owes to
the fact that with a source term rather than a nontrivial initial condition, the problem is naturally well-prepared.
The factor $T^{-1}$ in front of the $L^2$-norm of the error is the natural scaling since the only a priori bound on the $L^2$-norm of the solution $u_\e$ in general 
is precisely
$
\int u_\e^2(x,T)dx \,\lesssim \, T^2,
$
so that the accuracy of the error estimate at the level of the $L^2$-norm remains unchanged in relative terms. 
 \qed
\end{rem}
\begin{rem}
Likewise, we can consider a nontrivial initial velocity and vanishing initial position and forcing term, in which case \eqref{e.forcing_term2} still holds while \eqref{e.forcing_term} only survives at the level of the $L^2$-norm only (with a RHS depending on norms of the initial velocity).
\qed
\end{rem}
The proof of Theorem~\ref{t.wave-eq-int} is similar to the proof of Theorem~\ref{t.wave-eq} and relies on the following two arguments:
\begin{itemize}
\item the source term in the interior can be replaced by a well-prepared source term $f_{\e,\ell}$ up to an error uniformly small (in the energy norm) in time;
\item the fact that the Taylor-Bloch waves almost diagonalize the wave operator, and that the error due to the eigendefects can be controlled by suitable energy estimates on the wave equation with well-prepared source term.
\end{itemize}
\begin{lemma}\label{lem:prepare-ap}
Let $\ell\ge 1$, and let $f_{\e,\ell} \in L^2(\R^d)$ be defined by 
$$
f_{\e,\ell}(t, x)\,:=\,\sum_{j=0}^{\ell}\varepsilon^j \varphi_j\big( \frac{x}{\varepsilon}\big) \cdot \nabla^j f_{\e}(t,x),
$$
where $f_{\e}(t,\cdot):=\calF^{-1}(\omega_{\ell}(\e|\cdot|)\hat f(t, \cdot))$, and where $\varphi_j$ stands for the (symmetric) $j$-th order tensor such that $\varphi_j \cdot e^{\otimes j}=\varphi_j^e$, the $j$-th corrector in direction $e$.
Consider the unique weak solution $v_{\e,\ell}\in L^\infty(\R_+,L^2(\R^d))$ of the initial value problem
\begin{equation}\label{e.eq-velleps-append}
\left\{
\begin{array}{rcl}
\square_\e v_{\e,\ell}&=&f_{\e,\ell},
\\
v_{\e,\ell}(0,\cdot)&=&0,
\\
\partial_t v_{\e,\ell}(0,\cdot)&=&0.
\end{array}
\right.
\end{equation}
Then if Hypothesis~\ref{hypo:corr} holds, we have
\begin{multline*}
\sup_{0\le t\le T} \expec{\|\partial_t(u_\e-v_{\e,\ell})\|_{L^2(\R^d)}^2+\|\nabla(u_\e-v_{\e,\ell})\|_{L^2(\R^d)}^2+T^{-2} \|u_\e-v_{\e,\ell}\|_{L^2(\R^d)}^2}^\frac12 \\
\lesssim \, C_\ell(f) \max\{\e,\e^{\ell}\mu_\alpha(\e^{-1})\}.
\end{multline*}
\qed
\end{lemma}
\begin{proof}[Proof of Lemma~\ref{lem:prepare-ap}]
The proof relies on the following energy estimate.
Let the function $g\in L^\infty(\R_+,L^2(\mathbb{R}^d))$ be a source term supported in time in $[ 0,1]$ and consider $w$ the solution of the wave equation:
\begin{equation}\label{e.eq-append}
\left\{
\begin{array}{rcl}
\partial_{tt}^2 w-\nabla \cdot \aa \nabla w &=&g,
\\
w(0,\cdot)&=&0,
\\
\partial_t w(0,\cdot)&=&0,
\end{array}
\right.
\end{equation}
for some uniformly elliptic and bounded matrix field $\aa$.
For all $t\ge 0$, we multiply \eqref{e.eq-append} by $\partial_t w$ and integrate over $[0,t] \times \mathbb{R}^d$, which yields the energy estimate (see Substep~3.1 in the proof of Proposition~\ref{prop:energy-wave}):
$$
\frac{1}{2}\left(\Vert \partial_t w(t,\cdot)\Vert_{L^2(\mathbb{R}^d)}^2 + \Vert \nabla w(t,\cdot)\Vert_{L^2(\mathbb{R}^d)}^2\right) \leq \int_{\left[0,t \right] \times \mathbb{R}^d}g(s,x)\partial_t w(s,x) \, dx \, ds. 
$$
Since $g$ is compactly supported in time in $[ 0,1]$, we may absorb part of the RHS in the LHS by Young's inequality and obtain that for all $T\ge 0$
$$
\sup_{0\le t\leq T}\left(\Vert \partial_t w(t,\cdot)\Vert_{L^2(\mathbb{R}^d)}^2 + \Vert \nabla w(t,\cdot)\Vert_{L^2(\mathbb{R}^d)}^2\right) \lesssim \int_{\left[0,1 \right] \times \mathbb{R}^d} g(s,x)^2 \, dx \, ds. 
$$
We then integrate the PDE from $0$ to $T$ and argue as in Substep~3.2 of the proof of Proposition~\ref{prop:energy-wave} to 
obtain
$$
\sup_{0\le t\leq T}\Vert  w(t,\cdot)\Vert_{L^2(\mathbb{R}^d)}^2 \,\lesssim \,  T^2\int_{\left[0,1 \right] \times \mathbb{R}^d} g(s,x)^2 \, dx \, ds,
$$
where this time, we have an additional factor $T^2$.
\end{proof}
We then turn to the second point of the proof, and estimate $v_{\e , \ell}-u_{\e , \ell}$.
\begin{prop}\label{prop:energy-wave-appen}
For $\ell\ge 1$, let $ u_{\e,\ell}$ be defined in \eqref{sol_approx_interior}, and $v_{\e,\ell}$ be the unique weak solution of \eqref{e.eq-velleps-append}.
Then if Hypothesis~\ref{hypo:corr} holds we have for all $T\ge 1$ and $1\gg \e>0$,
\begin{multline}\label{e.energy-wave-appen}
\sup_{0\le t\le T} \expec{ \|\partial_t(v_{\e,\ell}-u_{\e,\ell})\|_{L^2(\R^d)}^2+\|\nabla(v_{\e,\ell}-u_{\e,\ell})\|_{L^2(\R^d)}^2+ T^{-2} \|v_{\e,\ell}-u_{\e,\ell}\|_{L^2(\R^d)}^2}^\frac12 
\\
\lesssim \, C_\ell(f)\e^{\ell}T\mu_{\alpha}(\e^{-1}T).
\end{multline}
\qed
\end{prop}
\begin{proof}[Proof of Proposition~\ref{prop:energy-wave-appen}]
We split the proof into two steps.

\medskip

\step{1} Reformulation.

\noindent We first note that $f_{\e,\ell}$ satisfies 
$$
f_{\e,\ell}(t,x)\,=\,\frac{1}{(2\pi)^d}\int_{\R^d} \hat f_\e(t,k)e^{ik\cdot x} \psi_{\e k,\ell} \big(\frac x\e\big) \ dk.
$$
We then compute $\square_\varepsilon u_{\e , \ell}$, which, in view of  Proposition~\ref{prop:bloch}, satisfies
\begin{eqnarray}\nonumber
\square_\varepsilon u_{\varepsilon ,\ell}  &=& f_{\e,\ell}\\ \nonumber
&-&\frac{\varepsilon^{\ell-1}}{(2\pi)^d}\int_{\mathbb{R}^d}\int_0^t (i\kappa)^{\ell+1} \hat f_{\e}(s,k)e^{ik\cdot x}(\nabla \cdot \Phi_{1,\ell}^e)\big(\frac{x}{\varepsilon}\big) \frac{\sin((\varepsilon^{-1}\Lambda_\ell(\varepsilon k ))^2(t-s))}{(\varepsilon^{-1}\Lambda_\ell(\varepsilon k ))^2}ds \, dk \\ \nonumber 
&-&\frac{\varepsilon^{\ell}}{(2\pi)^d}\int_{\mathbb{R}^d}\int_0^t (i\kappa)^{\ell+2}\hat f_{\e}(s,k) e^{ik\cdot x}\Phi_{2, \ell , \kappa}^e
\big( \frac{x}{\varepsilon}\big) \frac{\sin((\varepsilon^{-1}\Lambda_\ell(\varepsilon k ))^2(t-s))}{(\varepsilon^{-1}\Lambda_\ell(\varepsilon k ))^2}ds \, dk,
\end{eqnarray}
where $\Phi_{1,\ell}^e$ and $\Phi_{2,\ell,\kappa}^e$ are still given by
$$\Phi_{1,\ell}^e:=-\sigma_{\ell}^e e+\aa e\varphi_{\ell}^e +\nabla \chi_{\ell}^e \text{, } \quad \Phi_{2,\ell,\kappa}^e:=e\cdot \aa e\varphi_{\ell}^e-\sum_{j=1}^{\ell} \sum_{p=\ell-j}^{\ell} (i\e\kappa)^{j+p-\ell}  \lambda_{p}\phi_{j}^e.$$ 
Hence, the difference $w_{\e , \ell}:=v_{\e, \ell}-u_{\e , \ell}$ satisfies
\begin{equation}
\nonumber
\left\{
\begin{array}{rcl}
\square_\e w_{\e , \ell}(t,x)&=&\varepsilon^{\ell}(f_{\varepsilon , 1,\ell} +f_{\varepsilon , 2,\ell} ) ,
\\
\partial_t w_{\varepsilon ,\ell}(0,\cdot)&=&0,
\\
w_{\varepsilon ,\ell}(0,\cdot)&=&0.
\end{array}
\right.
\end{equation}
with the source terms
\begin{multline*}
f_{\e ,1,\ell}(t,x):=\\-\frac{1}{(2\pi)^d} \int_{\mathbb{R}^d}\int_0^t(i\kappa)^{\ell+1} \hat f_{\e}(s,k)\nabla \cdot \left(e^{ik\cdot x}\Phi_{1,\ell}^e\big( \frac{x}{\varepsilon}\big) \right) \frac{\sin((\varepsilon^{-1}\Lambda_\ell(\varepsilon k ))^2(t-s))}{(\varepsilon^{-1}\Lambda_\ell(\varepsilon k ))^2}ds \, dk,
\end{multline*}
and
\begin{multline*}
 f_{\e ,2,\ell}(t,x):=\\-\frac{1}{(2\pi)^d}\int_{\mathbb{R}^d}\int_0^t(i\kappa)^{\ell+2} \hat f_{\e}(s,k)  e^{ik\cdot x}\Phi_{2,\ell,\kappa}^e\big( \frac{x}{\varepsilon}\big)  \frac{\sin((\varepsilon^{-1}\Lambda_\ell(\varepsilon k ))^2(t-s))}{(\varepsilon^{-1}\Lambda_\ell(\varepsilon k ))^2}ds \, dk ,
\end{multline*}

\medskip

\step2 Proof of \eqref{e.energy-wave-appen}.

\noindent We first prove that 
\begin{multline}\label{e.energy-wave-appen-b}
\sup_{0\le t\le T} \expec{ \|\partial_t(v_{\e,\ell}-u_{\e,\ell})\|_{L^2(\R^d)}^2+\|\nabla(v_{\e,\ell}-u_{\e,\ell})\|_{L^2(\R^d)}^2}^\frac12
%+\sup_{t\le 1} \expec{  \|v_{\e,\ell}-u_{\e,\ell}\|_{L^2(\R^d)}^2}^\frac12 
\,
\lesssim \, C_\ell(f)\e^{\ell}T\mu_{\alpha}(\e^{-1}T).
\end{multline}
Indeed, since $\hat f_{\e}(\cdot , k) $ is supported in time in $[0,1]$, this estimate
follows from Substep~3.1  in the proof of Proposition~\ref{prop:energy-wave}.
% for the first two LHS terms and from Substep~3.2
% (note that the integration is limited to time 1) in the proof of Proposition~\ref{prop:energy-wave}.
To control the $L^2$-norm, we proceed as in Substep~3.2  in the proof of Proposition~\ref{prop:energy-wave},
but lose this time an additional factor of $T$.
%
%It remains to upgrade this estimate to \eqref{e.energy-wave-appen}.
%Note that $v_{\e,\ell}-u_{\e,\ell}$ satisfies on $[1,\infty)\times \R^d$
%%
%\begin{equation*}
%\left\{
%\begin{array}{rcl}
%\partial_{tt}^2 (v_{\e,\ell}-u_{\e,\ell})-\nabla \cdot \aa_\e \nabla (v_{\e,\ell}-u_{\e,\ell}) &=&0,
%\\
%(v_{\e,\ell}-u_{\e,\ell})(1,\cdot)&=&(v_{\e,\ell}-u_{\e,\ell})(1,\cdot),
%\\
%\partial_t (v_{\e,\ell}-u_{\e,\ell})(1,\cdot)&=&\partial_t (v_{\e,\ell}-u_{\e,\ell})(1,\cdot),
%\end{array}
%\right.
%\end{equation*}
%%
%so that the energy estimate~\eqref{e.appen-apriori} in the proof of Lemma~\ref{lem:prepare-ap} yields for all $T\ge 1$,
%%
%$$
%\sup_{1\le t\le T} \Vert  (v_{\e,\ell}-u_{\e,\ell})(t,\cdot)\Vert_{L^2(\mathbb{R}^d)}^2  \lesssim    \Vert  (v_{\e,\ell}-u_{\e,\ell})(1,\cdot)\Vert_{L^2(\mathbb{R}^d)}^2+ \Vert  \partial_t (v_{\e,\ell}-u_{\e,\ell})(1,\cdot)\Vert_{L^2(\mathbb{R}^d)}^2.
%$$
%%
%Combined with \eqref{e.energy-wave-appen-b} for $T=1$ and for general $T\ge 0$, this proves \eqref{e.energy-wave-appen}.
\end{proof}
Theorem~\ref{t.wave-eq-int} then essentially follows from Lemma~\ref{lem:prepare-ap} and Proposition~\ref{prop:energy-wave-appen}.

\medskip

We conclude this section with a long-time homogenization result for \eqref{i3}, for which we have a corrector result (that is, convergence in the
energy norm) since the problem is naturally well-prepared.
\begin{theo}\label{t.hom-long-source}
Let $\ell\ge 1$, and assume that $(\phi_j,\sigma_j,\chi_j)_{0\le j\le \ell}$ satisfy Hypothesis~\ref{hypo:corr}
for some $\alpha\in [0,1)\times \R_+$. Assume that $\gamma_\ell\ge 0$ is large enough so that $\calL_{\ho,\e,\ell}$ (defined in \eqref{first_ell_operator}) is a positive elliptic operator (see Lemma~\ref{lem:elliptic-well-posed}), and let $f$ be as in Theorem~\ref{t.wave-eq-int}.
For all $\e>0$, let $u_{\e}$ denote the solution of \eqref{i3} and let $w_{\e, \ell}$  denote the solution of  the homogenized equation
\begin{equation}\label{e.eq-uepsell-append} 
\left\{
\begin{array}{rcl}
\partial^2_{tt} w_{\e, \ell}+ \calL_{\ho,\e,\ell} w_{\e, \ell} &=&f,\\
w_{\e, \ell}(0,\cdot)&=&0,\\
\partial_t w_{\e, \ell}(0,\cdot)&=&0.
\end{array}
\right.
\end{equation}
Then we have for all $T\ge 1$
\begin{equation}\nonumber
\sup_{0 \le t\le T} \expec{\|u_\e-w_{\e, \ell}\|_{L^2(\R^d)}^2}^\frac12
\, \lesssim \, C_\ell(f)T\big(\varepsilon  \\+\varepsilon^{\ell}T\mu_\alpha(\varepsilon^{-1}T)\big),
\end{equation}
where $C_\ell(f)$ is a generic (finite) norm of $f$ which only depends on $\ell$.
In addition, if we consider the multiscale expansion $\tilde w_{\e, \ell}$ of $w_{\e, \ell}$ defined as
$$
\tilde w_{\e, \ell}(t,x)\,:=\, \sum_{j=0}^{\ell}\varepsilon^j \varphi_j\big( \frac{x}{\varepsilon}\big) \cdot \nabla^j w_{\e, \ell}(t,x),
$$
then we have the following long-time estimate in the energy norm (the so-called corrector estimate)
\begin{equation}\nonumber
\sup_{t\le T} \expec{\|\nabla (u_\e- \tilde w_{\e, \ell})\|_{L^2(\R^d)}^2+\|\partial_t( u_\e- \tilde w_{\e, \ell})\|_{L^2(\R^d)}^2}^\frac12
\, \lesssim \, C_\ell(f)\big(\varepsilon  \\+\varepsilon^{\ell}T\mu_\alpha(\varepsilon^{-1}T)\big).
\end{equation}
\qed
\end{theo}
The proof of Theorem~\ref{t.hom-long} is a straightforward adaptation of the proof of Theorem~\ref{t.hom-long} and is left to the reader.

\section{The case of systems}\label{sec:systems}

For systems the problem acquires an additional dimension, say $d$ (with linear elasticity in mind). In that case, we define $d$ families of extended higher-order correctors.
In this section, we assume that $\aa:\R^d \to \mathcal{M}_{d\times d}(\R)$ (the set of symmetric fourth-order tensors) is uniformly bounded and satisfies the strong ellipticity condition
$$
\xi \cdot \aa(x) \xi \ge \lambda |\xi|^2
$$
for some $\lambda>0$ and for almost all $x\in \R^d$ and all $\xi \in \R^{d\times d}$. In view of \cite{GNO3}, we can also consider the weaker notion of functional coercivity
$$
\int_{\R^d} \nabla v \cdot \aa \nabla v \ge \lambda \int_{\R^d} |\nabla v|^2
$$
for all $v\in \mathcal S(\R^d,\R^d)$, which allows us to deal with the system of linear elasticity.

\medskip

Fix a direction $e \in \R^{d}$. 
We define $d$ families of extended correctors $(\phi_j^m,\sigma_j^m,\chi_j^m)_{1\le m\le d,j}$ in direction $e$ as follows.
\begin{defi}\label{defi:ext-corr-systems}
For all $\ell\ge 0$, we say that $(\phi_j^m,\sigma_j^m,\chi_j^m)_{1\le m\le d,0\le j\le \ell}$ are 
 the first $\ell$ extended correctors in direction $e$ if these 
functions are locally square-integrable, if for all $0<j\le \ell$ the functions $(\nabla \phi_j^m,\nabla \sigma_j^m)_{1\le m\le d}$ are $\Z^d$-stationary and
satisfy $\expec{\int_Q |(\nabla \phi^m_j,\nabla \sigma^m_j)|^2}$ $<\infty$, if for all $0<j<\ell$ the functions $(\phi_j^m,\sigma_j^m,\nabla \chi_j^m)_{1\le m\le d}$ are $\Z^d$-stationary and satisfy
$\expec{\int_Q (\phi_j^m,\sigma_j^m,\nabla \chi_j^m)}=0$ and $\expec{\int_Q |( \phi_j^m, \sigma_j^m, \nabla \chi_j^m)|^2}<\infty$, and if the following extended corrector equations on $\R^d$
are satisfied:
\begin{itemize}
\item for all $1\le m\le d$, $\phi_0^m\equiv e_m$, and for all $j\ge 1$, $\phi_j^m$ is a vector field that satisfies
$$-\nabla \cdot \aa\nabla \phi_j^m=\nabla \cdot (-\sigma_{j-1}^m e+\aa (e\otimes\phi_{j-1}^m)+\nabla \chi_{j-1}^m);$$
\item for all $j\ge 0$, the symmetric fourth order tensor $\tilde \aa_j$, the symmetric $(j+4)$-th order tensor $\bar \aa_j$, and the symmetric matrix $\lambda_j$ are given 
for all $1\le m\le d$ 
$$\bar \aa_j  (e^{\otimes j} \otimes e\otimes e_m) = \tilde\aa_j (e \otimes  e_m):= \expec{\int_Q\aa (\nabla \phi_{j+1}^m+ e\otimes\phi_j^m)}, \quad \lambda_j:= e \cdot \tilde \aa_j  e;$$
\item  for all $1\le m\le d$, $\chi_0^m\equiv 0$, $\chi_1^m\equiv 0$, and for all $j\ge 2$, $\chi_j^m$ is a vector field that satisfies
$$
-\triangle \chi_j^m=\nabla\chi_{j-1}^m\cdot e +\sum_{p=1}^{j-1}\lambda_{j-1-p}\phi_{p}^m;
$$
\item  for all $1\le m\le d$ and all $j\ge 1$, $q_j^m$ is a matrix field (a higher-order flux) given by
$$q_j^m:=  \aa (\nabla\phi_j^m+e\otimes\phi_{j-1}^m)-\tilde\aa_{j-1} (e \otimes  e_m)+\nabla \chi_{j-1}^m-\sigma_{j-1}^me, \quad \expec{\int_Q q_j^m}=0;$$
\item  for all $1\le m\le d$, $\sigma_0^m\equiv 0$, and for all $j\ge 1$, $\sigma_j^m$ is a skew-symmetric third-order tensor field (a higher-order flux corrector), i.e. $\sigma_{jkln}^m=-\sigma_{jlkn}^m=-\sigma_{jknl}^m=-\sigma_{jnlk}^m$, that satisfies
$$
-\triangle \sigma_j^m = \nabla \times q_j^m, \quad \nabla \cdot \sigma_j^m= q_j^m,
$$
with the three-dimensional notation: $[\nabla \times q_j^m]_{pn}=\nabla_p [q_j^m]_n-\nabla_n [q_j^m]_p$,
and where the divergence is taken with respect to the third index, i.~e.~$(\nabla \cdot \sigma_j^m)_{kl}:=\sum_{n=1}^d \partial_n \sigma_{jkln}^m$.
\end{itemize}
\qed
\end{defi}
Proposition~\ref{prop:properties-lambda} holds in the following form:
\begin{itemize}
\item For all unit directions $e'\in \R^d$, and all $j$ odd, $e'\cdot \lambda_j e'=0$;
\item For all unit directions $e'\in \R^d$, $e'\cdot \lambda_0 e'>0$ and $e'\cdot \lambda_2 e'\ge 0$.
\end{itemize}
More precisely, the proof displayed in the scalar setting holds mutatis mutandis for each entry $e'\cdot \lambda_j e'$ of the symmetric matrix in the case of systems.

\medskip

As in the scalar case, we can introduce Taylor-Bloch waves, ``eigenvalues'' and eigendefects.
For all $\ell\ge 1$ and all $k=\kappa e\in \R^d$ we define the Taylor-Bloch ``eigenvalue'' (in form of a symmetric matrix)
$$\tilde \lambda_{k,\ell} \,:=\,\kappa^2 \sum_{j=0}^{\ell-1} (i\kappa)^j \lambda_j , $$
and for all $1\leq m \leq d$ we define the Taylor-Bloch wave $\psi_{k,\ell}^m$ (a vector) and the eigendefect $\mathfrak{d}_{k,\ell}^m$ (also a vector) by
\begin{multline*}
\psi_{k,\ell}^m \,:=\,\sum_{j=0}^\ell (i\kappa)^j \phi_j^m, 
\\
 \mathfrak{d}_{k,\ell}^m= \nabla \cdot (-\sigma_{\ell}^me+\aa (e\otimes\varphi_{\ell}^m) +\nabla \chi_{\ell}^m)+i\kappa \Big(e\cdot \aa (e\otimes \varphi_{\ell}^m)-\sum_{j=1}^\ell \sum_{p=\ell-j}^{\ell-1} (i\kappa)^{j+p-\ell}  \lambda_{p}\phi_{j}^m\Big) .
\end{multline*}
Proposition \ref{prop:bloch} then holds in the following form: For all $1\leq m \leq d$ we have the eigendefect identity:
\begin{equation}\nonumber
-(\nabla +ik) \cdot \aa (\nabla + ik )\psi_{k,\ell}^m \,=\, \tilde \lambda_{k,\ell} \psi_{k,\ell}^m -
(i\kappa)^{\ell+1} \mathfrak{d}_{k,\ell}^m.
\end{equation}
The proof is identical to the scalar case.

\medskip

There is a significant difference between the Fourier transform and the approximate Floquet-Bloch transform in the case of systems: 
Fourier modes are diagonal, whereas for Taylor-Bloch modes, $\tilde \lambda_{k,\ell}$ is not diagonal in general (the $\tilde \lambda_{k,\ell}$'s do not commute for different $k$) and the modes are coupled. This owes to the well-known fact
that spectral projectors are the natural objects for systems (rather than eigenvectors).

\medskip

We turn now to the approximation of the solution of the initial value problem \eqref{e.eq-ueps}, and quickly argue how to extend Theorem~\ref{t.wave-eq} to systems.
To this aim, we first define the quantity (now a matrix) $\Lambda_\ell$: \\

For all $\ell\ge 0$ and $k\in \R^d$
such that $\tilde \lambda_{k,\ell}$ is a well-defined and non-negative matrix (which holds for $|k|\ll 1$ since $\lambda_0$ is invertible for all unit vectors $e\in \R^d$), 
we set $\Lambda_\ell(k):=\sqrt{\tilde \lambda_{k,\ell}}$ (i.e.~the square-root of a symmetric non-negative matrix).
Recall the definition of the low-pass $\omega_{\ell}$ the role of which is to filter frequencies $k$ for which $\tilde \lambda_{k,\ell}$ is not non-negative.
Assume that Hypothesis \ref{hypo:corr} holds for each familly of extended correctors $(\phi_j^m,\sigma_j^m,\chi_j^m)_{0\le j\le \ell}$.
As for the proof of Theorem~\ref{t.wave-eq}, we start with preparing the data, and we replace $u_0$ by
$$
u_{0,\e,\ell}(x)\,:=\,\sum_{m=1}^d\sum_{j=0}^{\ell}\varepsilon^j \varphi_j^m\big( \frac{x}{\varepsilon}\big) \cdot \nabla^j [u_{0,\e}]_m(x) \in \R^d,
$$
where $u_{0,\e}=\calF^{-1}(\omega_{\ell}(\e|\cdot|)\hat u_0)$ is the filtering of $u_0$ by $\omega_\ell$, and we consider the solution $v_{\e,\ell}$ associated with this well-prepared initial condition. The estimate of Lemma~\ref{lem:prepare} is unchanged.
We then define an approximation $\tilde v_{\e,\ell}$ of $v_{\e,\ell}$ using Taylor-Bloch waves
$$
\tilde v_{\e,\ell}(t,x) \,:=\,\sum_{m=1}^d\frac{1}{(2\pi)^d}\int_{\R^d} \omega_{\ell}(\e|k|)[\hat u_{0}]_m(k)e^{ik\cdot x}  \cos(\e^{-1} \Lambda_\ell(\e k)t) \psi_{\e k,\ell}^m \big(\frac x\e\big)\ dk, 
$$
where $M\mapsto \cos(M)$ denotes the cosinus function on matrices (defined as the real part of the complex exponential of a matrix).
By controlling the growth in time of the eigendefect, we obtain the estimate of Proposition~\ref{prop:energy-wave} for systems.
It remains to simplify the approximate solution by defining
$$ u_{\e,\ell}(t,x)\, :=\sum_{m=1}^d  \frac{1}{(2\pi)^d} \int_{\R^d} \omega_{\ell}(\e |k|) [\hat u_{0}]_m(k)  e^{ik\cdot x} \cos(\e^{-1} \Lambda_\ell(\e k)t)e_m dk ,$$ 
which remains accurate in the sense of Lemma~\ref{lem:throw-away}.
 
 \medskip

Based on this, Sections~\ref{sec:disper} and~\ref{sec:deloc} are easily extended to systems, and we leave the details to the reader.

\section{Estimates of the extended higher-order correctors}\label{sec:corr}

\subsection{Existence of higher-order correctors}

Recall that $\phi_0\equiv 1,\sigma_0\equiv 0,\chi_0\equiv 0,\chi_1\equiv 0$, and
that $\phi_1$ and $\sigma_1$ are the classical corrector and flux corrector in homogenization which are well-defined for stationary ergodic coefficients (cf. \cite{GNO3}).
As it is standard in stochastic homogenization, one might modify the (higher-order) corrector equations 
by adding a zero-order term of magnitude $T^{-1}$ for some $T\gg 1$. This would yield existence and uniqueness
of stationary approximations of the extended correctors $(\phi_j,\sigma_j,\nabla \chi_j)$ at any order.
In fact, interpreting $T$ as a time-scale, one could even let $T$ depend in a nontrivial way on $\e$ in the various estimates, and work with
approximate correctors only. The price to pay to work with these well-defined approximate correctors is that the crucial identity
$\nabla \cdot \sigma_j=q_j$ would only hold up to some defect (depending on $T$). Similarly, there would be an additional defect 
in the eigenvalue/eigenvector relation for the associated approximate Taylor-Bloch wave.
Last, we would have approximations of the tensors $\bar \aa_j$ depending on $T$.
To avoid this additional approximation, we directly work with the higher-order correctors without massive approximation.
In this case however, higher-order extended correctors are not necessarily well-defined, and the associated existence/uniqueness theory
makes heavy use of quantitative homogenization methods.

\medskip

Let us start with a soft result: the existence and uniqueness of a (non-stationary) $(\phi_{j},\sigma_{j})$ 
provided $\phi_{j-1},\sigma_{j-1},\nabla \chi_{j-1}$ are stationary fields with finite second moment.
\begin{lemma}
Let $\phi_{j-1},\sigma_{j-1},\nabla \chi_{j-1},q_j$ be as in Definition~\ref{defi:ext-corr}.
Assume that $\phi_{j-1},\sigma_{j-1}$, and $\nabla \chi_{j-1}$ are stationary fields with finite second moments,
which implies that $q_j$ is also stationary with finite second moment.
Then there exist random fields $\phi_j,\sigma_j$ solving 
\begin{eqnarray*}
-\nabla \cdot \aa\nabla \phi_j\,=\,\nabla \cdot (\sigma_{j-1}e+\aa e\phi_{j-1}+\nabla \chi_{j-1}),
\\
-\triangle \sigma_j = \nabla \times q_j, \quad \nabla \cdot \sigma_j= q_j,
\end{eqnarray*}
such that $\nabla \phi_j,\nabla \sigma_j$ are uniquely defined stationary fields
with finite second moments.
\qed
\end{lemma}

\medskip

Let us now distinguish between the assumptions on $\aa$.
We start with the periodic and quasi-periodic setting, then turn to the almost periodic setting, and conclude with the random setting. 

%%%%%%%%%%%%%%%%%

\subsection{Periodic and quasi-periodic coefficients}

The following result is a direct consequence of the Poincar\'e inequality on the torus
and of spectral theory.
\begin{prop}\label{prop:corr-perio}
Let $\aa$ be a measurable periodic coefficient field. Then for all $j\ge 1$, there exist unique periodic extended correctors 
$\phi_j,\sigma_j,\chi_j\in H^1_\loc(\R^d)$ with zero average.
If in addition $\aa$ is symmetric, then for all $j\ge 1$ and unit direction $e$,
$\lambda_{2j+1}=0$.
\qed
\end{prop}
\begin{proof}[Proof of Proposition~\ref{prop:corr-perio}]
Once the structure of the correctors is clear (only the definition of the flux correctors $\sigma_j$ is delicate --- cf. the discussion
on closed forms after Definition~\ref{defi:ext-corr}), the only subtle result is that $\lambda_{2j+1}=0$ for all $j\ge 0$.
We give here the classical proof of this fact (which we already proved by a direct approach in Proposition~\ref{prop:properties-lambda}).
Let $k=\kappa e$, where $e$ is a fixed unit direction and $\kappa \in \R_+$.
Since $-(\nabla +i\kappa e)\cdot\aa (\nabla+i\kappa e)$ has compact resolvent on the torus, we can consider
the first eigenvalue $\lambda_1(\kappa)$. As a function of $\kappa$, $\lambda_1$ is real analytic on a neighborhood of the
origin (so that its derivatives are
given by the extended correctors), cf.~\cite{CV}. 
A direct computation shows that $\lambda_1^{(j)}(0)=i^{j+1}\lambda_{j+1}$ (the 
$j$-th derivative of $\lambda_1$ at zero is given by $i^{j+1}\lambda_{j+1}$ from Definition~\ref{defi:ext-corr}).
Since $\lambda_j \in \R$ and $\lambda_1^{(j)}(0)\in \R$, this implies $\lambda_{2j+1}=0$, as claimed.
\end{proof}

Similar results as in Proposition~\ref{prop:corr-perio} hold in the case of smooth quasi-periodic coefficient fields first considered by Kozlov \cite{K}.
The arguments  of \cite[Theorem~4]{GH}, based on an diophantine condition in the form of a weak Poincar\'e inequality, on
Garding's inequality, and elliptic regularity, indeed allow to prove the following.
\begin{prop}
Let $\tilde \aa$ be a smooth coefficient field on a higher-dimensional torus $\T_m$, $m>d$,
 let $M$ be a winding $m\times d$-matrix, and set $\aa:\R^d\to \Md, x\mapsto \aa(x):=\tilde \aa(Mx)$.
 If $M$ satisfies a diophantine condition, then for all $j\ge 1$, there exist unique smooth quasi-periodic extended correctors 
$\phi_j,\sigma_j,\chi_j\in H^1_\loc(\R^d)$ with zero average. {In particular} all the extended correctors $\phi_j,\sigma_j,\chi_j$ are bounded.
\qed
\end{prop}
We then turn the more general case of almost-periodic coefficient fields $\aa$.

%%%%%%%%%%%%%%%%%

\subsection{Almost-periodic coefficients}

We first recall
the quantitative measure of almost-periodicity introduced in \cite{AGK}.
Given $f:\R^d \to \R^k$ and $x,y,z\in\R^d$, we define
\begin{equation*}
T_zf(x):= f(x+z)
\end{equation*}
and the difference operator
\begin{equation}
\label{e.diff-op}
\Delta_{yz}f(x):= \frac12 \left( T_yf(x) - T_zf(x)\right) = \frac12 \left(f(x+y) - f(x+z)\right).
\end{equation}
Let $\mathcal T_k = \left( (y_1,z_1),\ldots,(y_k, z_k) \right) \in (\R^d\times\R^d)^k$ be a $k$-tuple formed by couples $(y_j ,z_j) \in \R^d \times \R^d$. For a function $f:\R^d \to \R^{m \times n}$, $m,n \in \N$, we define a difference operator $\Delta_{{\mathcal T}_k}$ acting on $f$ by
\begin{equation}
\label{e.iterate-diff}
\Delta_{\mathcal T_k} f(x) = \Delta_{y_k z_k} \cdots \Delta_{y_1 z_1} f(x)\,.
\end{equation}
Let $\mathcal P_{j,k}$, $j \in \{0,1,\ldots,k\}$, stand for a set of increasing ordered subsets of $\{1,\ldots,k\}$ with $j$ members. In other words, for $j>0$ we define
\begin{equation*} \label{}
\mathcal P_{j,k}:= \left\{ \zeta \in \left\{ 1,\ldots,k\right\}^j \,:\,    \zeta_i < \zeta_{i+1}  \  \forall  i \in \{ 1,\ldots,j-1\} \right\}
\end{equation*}
and, for $j=0$, we set $\mathcal P_{0,k} = \emptyset$. By abuse of notation, we also think of $\zeta \in \mathcal{P}_{j,k}$ as being ordered subsets of $\{ 1,\ldots,k\}$. Then, for $\zeta \in \mathcal P_{j,k}$, we denote by $\zeta^c$ the unique member of $\mathcal P_{k-j,k}$ such that $\{1,\ldots,k\} = \zeta \cup \zeta^c$. By $|\zeta|$ we denote the number of elements in $\zeta \in P_{j,k}$, i.e., $|\zeta| = j$. For $\mathcal T_k$ as above and for $\zeta \in \mathcal P_{j,k}$ we denote by the $j$-tuple $\zeta(\mathcal T_k)$  the set  $\left( ( y_{\zeta_1}, z_{\zeta_1}), \ldots, (y_{\zeta_{j}}, z_{\zeta_{j} }) \right)$ for $j>0$, and if $\zeta \in \mathcal P_{0,k}$, we set $\zeta(\mathcal T_k) = \emptyset$  and  $\Delta_{\zeta(\mathcal T_k)} f = 1$.  Furthermore, we let $\mathcal P_k$ stand for the family of subsets $(\zeta^1,\ldots,\zeta^k) \in \mathcal P_{j_1,k} \times \cdots \times \mathcal P_{j_k,k}$ with $\sum_{i=1}^k j_i = k$.

 \smallskip

We are now in position to recall the quantitative measure of almost periodicity introduced in \cite{AGK}. 
For a given $f \in C(\R^d;\R^{m \times n})$, $m,n \in \N$, and $\mathcal T_k = \{(y_1,z_1),\ldots,(y_k, z_k)\}$ we define
\begin{equation} \label{e.G_k}
G_k(f,\mathcal T_k) := \max_{(\zeta^1,\ldots,\zeta^k) \in \mathcal P_k} \left\{ \prod_{j=1}^k \left\| \Delta_{\zeta^j(\mathcal T_k)} f  \right\|_{L^\infty(\R^d;\R^{m \times n})} \right\}\,,
\end{equation}
that is, the maximum is taken over all (increasing, ordered) partitions of $\mathcal P_k$. Then, we define a quantity $\rho_k$, for each $k \in \N$ and $R \geq 1$, by
\begin{equation} \label{e.defrhok}
 \rho_k(f,R) : = \sup_{y_1\in\R^d} \inf_{z_1\in B_R} \cdots \sup_{y_k\in\R^d} \inf_{z_k\in B_R} G_k\left(f , \left( (y_1, z_1),\ldots, (y_k,z_k)†\right) \right),
 \end{equation}
which are the building blocks for the quantitative measure of almost periodicity of \cite{AGK}:
\begin{equation}
\label{e.defrhostar}
\rho_*(\aa,R):= \inf_{k \in\N \cap [1,R] } C^k k! \rho_k\left(\aa, k^{-1} R\right),
\end{equation}
where the constant $C$ in~\eqref{e.defrhostar} only depends on $d,\Lambda$.
The main quantitative ergodicity assumption that we make on the coefficients is therefore that there exists an exponent $\delta > 0$ and a constant $K\geq 1$ such that, for every $R\geq 1$, 
\begin{equation}
\label{e.ergodicassump}
\rho_*(\aa,R) \leq K R^{-\delta}.
\end{equation}
We introduce the following notation: for all $\delta>0$, integer $j\ge 1$, and $t\ge 0$, we set
$$
\nu_{\delta,j}(t)\,:=\,\left\{
\begin{array}{lll}
1 &\text{ for }&j-\delta <0,\\
\log(2+t)^\frac{1}{2}& \text{ for }&j-\delta=0,\\
t^{j-\delta}&\text{ for } &0<j-\delta<1.
\end{array}
\right.
$$
Although it is not straightforward, under assumption \eqref{e.ergodicassump}, methods similar to \cite{AGK} (see also \cite{AKM2}, with some care for the borderline case $j=\delta$) essentially allow to prove the following control of the extended correctors:
\begin{prop}\label{prop:alper}
Let $\aa$ be an almost-periodic coefficient field satisfying \eqref{e.ergodicassump} for some $K\ge 1$ and $\delta>0$.
Then, for all $1\le j<1+\delta$, $(\phi_j,\sigma_j,\nabla \chi_j)$ are well-defined and satisfy
for all $x\in \R^d$
$$
|\phi_j(x)|+|\sigma_j(x)|+|\nabla \chi_{j}(x)| \,\lesssim\, \nu_{\delta,j}(|x|).
$$ 
\qed
\end{prop}
We finally consider random coefficient fields with decaying correlations.

%%%%%%%%%%%%%%%%%

\subsection{Random coefficients}

In this last subsection, we address the representative example of a family of Gaussian coefficient fields.
More precisely, we consider Gaussian ensembles
of scalar fields $a(x)$.
In order to get an example of an ensemble of uniformly elliptic coefficient
fields $\aa$, one applies a pointwise nonlinear Lipschitz transform to possibly several copies of the above.
Let $\Pm'$ (with expectation $\mathbb{E}'$) stand for the distribution of a scalar Gaussian field
$a(x)$ that is stationary and centered, and thus characterized by its covariance
\begin{equation}\nonumber
c(x):=\mathbb{E}'[ a(x)a(0)].
\end{equation}
We assume that the covariance is radial and decays mildly in the sense that there exists $\beta >0$  such that 
\begin{equation}\label{e.beta}
|c(x)|\lesssim \gamma_\beta(x):=(1+|x|)^{-\beta}.
\end{equation}
With a slight abuse of notation, we shall say that the law $\Pm$ (with expectation $\mathbb{E}$) of $\aa$ is Gaussian with parameter $\beta>0$.
Under this assumption, we have the validity of a weighted logarithmic-Sobolev inequality (cf.~\cite{DG1,DG2}), which is key to 
establish moment bounds on correctors in \cite{GNO3,GNO3b}.
%We define $\gamma_{\beta,*}(t):=\fint_{|x|<t} \gamma_\beta(|x|)dx$. The latter satisfies
%%
%$$
%\gamma_{\beta,*}(t)\,\sim \, \left\{
%\begin{array}{lll}
%t^{-d} &\text{ for }&\beta>d,\\
%t^{-d} \log(2+t)& \text{ for }&\beta=d,\\
%t^{-\beta}&\text{ for } &\beta<d.
%\end{array}
%\right.
%$$
%%
%We further introduce the following notation: for all $\beta>0$, $d\ge 1$, integer $j\ge 1$, and $t\ge 0$, we set
%%%
%%$$
%%\mu_{d,\beta,j}(t)\,:=\,\Big(1+\int_0^{t^2} \sqrt{\gamma_{\beta,*}(\tau)} (1+\tau)^{j-1}d\tau\Big)^\frac12.
%%$$
%%%
%%
%$$
%\mu_{d,\beta,j}(t)\,:=\,\Big(1+\int_0^{t} {\gamma_{\beta,*}(\tau)} (1+\tau)^{2j-1}d\tau\Big)^\frac12.
%$$
%%
Although it is not straightforward, proceeding as in \cite{GNO3b} (or using a semi-group approach as in \cite{GNO1,GO4}), one can prove the following sharp estimates (see also \cite{BFFO} for the case $j=2$):
\begin{prop}\label{prop:Gauss}
Let $\aa$ be a Gaussian coefficient field satisfying \eqref{e.beta} for some $\beta>0$.
Then, for all $j\in \N$, provided $d\ge 2j$ and $\beta>2(j-1)$, correctors of order $j$ exist and 
satisfy
\begin{multline*}
\expec{|\phi_j(x)|^2}^\frac12+\expec{|\sigma_j(x)|^2}^\frac12+\expec{|\nabla \chi_{j}(x)|^2}^\frac12 \\
\lesssim\, 
\left\{
\begin{array}{lll}
1&\text{ for }&\beta>2j,d>2j,\\
\log^\frac12(2+|x|)&\text{ for }&\beta>2j,d=2j,\\
\log(2+|x|)&\text{ for }&\beta=2j,d\ge2j,\\
1+|x|^{1-\frac\beta{2j}}&\text{ for }&2(j-1)<\beta<2j,d\ge2j.\\
\end{array}
\right.
\end{multline*}
\qed
\end{prop}
%

%%%%%%%%%%%%%

\section*{Acknowledgements}
The authors acknowledge financial support from the European Research Council under
the European Community's Seventh Framework Programme (FP7/2014-2019 Grant Agreement
QUANTHOM 335410). 
We wish to thank Gr\'egoire Allaire and Jeffrey Rauch for inspiring discussions on the subject (which led to Appendix~\ref{sec:forcing})
and Laszlo Erd\"os for his comments on a preliminary version of this manuscript.

%%%%%%%%%%%%%

%%%%%%%%%%%%

\end{document}